\documentclass[11pt,twoside]{article}

\usepackage{amsfonts,amsmath,amssymb,amsthm}
\usepackage{fullpage}
\usepackage{graphics,graphicx}
\usepackage{cases}
\usepackage{comment}

\newlength{\widebarargwidth}
\newlength{\widebarargheight}
\newlength{\widebarargdepth}
\DeclareRobustCommand{\widebar}[1]{%
  \settowidth{\widebarargwidth}{\ensuremath{#1}}%
  \settoheight{\widebarargheight}{\ensuremath{#1}}%
  \settodepth{\widebarargdepth}{\ensuremath{#1}}%
  \addtolength{\widebarargwidth}{-0.3\widebarargheight}%
  \addtolength{\widebarargwidth}{-0.3\widebarargdepth}%
  \makebox[0pt][l]{\hspace{0.3\widebarargheight}%
    \hspace{0.3\widebarargdepth}%
    \addtolength{\widebarargheight}{0.3ex}%
    \rule[\widebarargheight]{0.95\widebarargwidth}{0.1ex}}%
  {#1}}

\makeatletter
\long\def\@makecaption#1#2{
        \vskip 0.8ex
        \setbox\@tempboxa\hbox{\small {\bf #1:} #2}
        \parindent 1.5em  %% How can we use the global value of this???
        \dimen0=\hsize
        \advance\dimen0 by -3em
        \ifdim \wd\@tempboxa >\dimen0
                \hbox to \hsize{
                        \parindent 0em
                        \hfil 
                        \parbox{\dimen0}{\def\baselinestretch{0.96}\small
                                {\bf #1.} #2
                                } 
                        \hfil}
        \else \hbox to \hsize{\hfil \box\@tempboxa \hfil}
        \fi
        }
\makeatother

\usepackage[colorlinks]{hyperref}

\newcommand{\condind}{\ensuremath{\perp\!\!\!\perp}}
\newcommand{\opnorm}[1]{\left|\!\left|\!\left|{#1}\right|\!\right|\!\right|}

\newcommand{\E}{\ensuremath{\mathbb{E}}}
\newcommand{\mprob}{\ensuremath{\mathbb{P}}}

\newcommand{\betahat}{\ensuremath{\widehat{\beta}}}
\newcommand{\betastar}{\ensuremath{\beta^*}}
\newcommand{\betatil}{\ensuremath{\widetilde{\beta}}}

\newcommand{\real}{\ensuremath{\mathbb{R}}}
\newcommand{\defn}{\ensuremath{:=}}
\newcommand{\inprod}[2]{\ensuremath{\langle #1 , \, #2 \rangle}}
\DeclareMathOperator{\sign}{sign}
\newcommand{\supp}{\ensuremath{\operatorname{supp}}}

\newcommand{\order}{{\mathcal{O}}}

\newcommand{\nutil}{\ensuremath{\widetilde{\nu}}}
\newcommand{\Loss}{\ensuremath{\mathcal{L}}}

\newcommand{\T}{\ensuremath{\mathcal{T}}}
\newcommand{\Ball}{\ensuremath{\mathbb{B}}}
\newcommand{\var}{\ensuremath{\operatorname{var}}}
\newcommand{\scriptE}{\ensuremath{\mathcal{E}}}
\newcommand{\ftil}{\ensuremath{\widetilde{f}}}
\newcommand{\scriptD}{\ensuremath{\mathcal{D}}}

\newcommand{\boracle}{\ensuremath{\betahat^{\mathcal{O}}}}
\newcommand{\zhat}{\ensuremath{\widehat{z}}}
\newcommand{\Lossbar}{\ensuremath{\widebar{\Loss}}}
\newcommand{\etabar}{\ensuremath{\widebar{\eta}}}

\newcommand{\scriptM}{\ensuremath{\mathcal{M}}}
\newcommand{\betamin}{\ensuremath{\betastar_{\min}}}
\newcommand{\ztil}{\ensuremath{\widetilde{z}}}
\newcommand{\Ztil}{\ensuremath{\widetilde{Z}}}
\newcommand{\htil}{\ensuremath{\widetilde{h}}}
\newcommand{\Qhat}{\ensuremath{\widehat{Q}}}

\newtheorem{lem*}{Lemma}
\newtheorem{thm*}{Theorem}
\newtheorem{alg*}{Algorithm}
\newtheorem{cor*}{Corollary}
\newtheorem{prop*}{Proposition}

\newtheorem{assumption}{Assumption}
\newtheorem*{definition*}{Definition}

\begin{document}

\begin{center}
	
{\bf{\LARGE{Statistical consistency and asymptotic normality for high-dimensional robust $M$-estimators}}}
	
\vspace*{.2in}

\begin{tabular}{c}
{\large{Po-Ling Loh}} \\
{\large{\texttt{loh@wharton.upenn.edu}}} \\
\vspace*{.005in} \\
Department of Statistics \\
The Wharton School \\
University of Pennsylvania \\
Philadelphia, PA 19104
\end{tabular}

\vspace*{.2in}

\today

\end{center}

\begin{abstract}

We study theoretical properties of regularized robust $M$-estimators, applicable when data are drawn from a sparse high-dimensional linear model and contaminated by heavy-tailed distributions and/or outliers in the additive errors and covariates. We first establish a form of local statistical consistency for the penalized regression estimators under fairly mild conditions on the error distribution: When the derivative of the loss function is bounded and satisfies a local restricted curvature condition, all stationary points within a constant radius of the true regression vector converge at the minimax rate enjoyed by the Lasso with sub-Gaussian errors. When an appropriate nonconvex regularizer is used in place of an $\ell_1$-penalty, we show that such stationary points are in fact unique and equal to the local oracle solution with the correct support---hence, results on asymptotic normality in the low-dimensional case carry over immediately to the high-dimensional setting. This has important implications for the efficiency of regularized nonconvex $M$-estimators when the errors are heavy-tailed. Our analysis of the local curvature of the loss function also has useful consequences for optimization when the robust regression function and/or regularizer is nonconvex and the objective function possesses stationary points outside the local region. We show that as long as a composite gradient descent algorithm is initialized within a constant radius of the true regression vector, successive iterates will converge at a linear rate to a stationary point within the local region. Furthermore, the global optimum of a convex regularized robust regression function may be used to obtain a suitable initialization. The result is a novel two-step procedure that uses a convex $M$-estimator to achieve consistency and a nonconvex $M$-estimator to increase efficiency. We conclude with simulation results that corroborate our theoretical findings.

\end{abstract}

%%%%%%%%%%%%%%

\section{Introduction}

Ever since robustness entered the statistical scene in Box's classical paper of 1953~\cite{Box53}, many significant steps have been taken toward analyzing and quantifying robust statistical procedures---notably the work of Tukey~\cite{Tuk60}, Huber~\cite{Hub64}, and Hampel~\cite{Ham68}, among others. Huber's seminal work on $M$-estimators~\cite{Hub64} established asymptotic properties of a class of statistical estimators containing the maximum likelihood estimator, and provided initial theory for constructing regression functions that are robust to deviations from normality. Despite the substantial body of now existent work on robust $M$-estimators, however, research on high-dimensional regression estimators has mostly been limited to penalized likelihood-based approaches (e.g., \cite{Tib96, FanLi01, FriEtal08, RavEtal10}). Several recent papers~\cite{NegEtal12, LohWai13, LohWai14} have shed new light on high-dimensional $M$-estimators, by presenting a fairly unified framework for analyzing statistical and optimization properties of such estimators. However, whereas the $M$-estimators studied in those papers are finite-sample versions of globally convex functions, many important $M$-estimators, such as those arising in classical robust regression, only possess convex curvature over local regions---\emph{even at the population level}. In this paper, we present new theoretical results, based only on \emph{local} curvature assumptions, which may be used to establish statistical and optimization properties of regularized $M$-estimators with highly nonconvex loss functions.

Broadly, we are interested in linear regression estimators that are robust to the following types of deviations:

\begin{itemize}

\item[(a)] \emph{Model misspecification.}  The ordinary least squares objective function may be viewed as a maximum likelihood estimator for linear regression when the additive errors $\epsilon_i$ are normally distributed. It is well known that the $\ell_1$-penalized ordinary least squares estimator is still consistent when the $\epsilon_i$'s are sub-Gaussian~\cite{BicEtal08, Wai09}; however, if the distribution of the $\epsilon_i$'s deviates more wildly from the normal distribution (e.g., the $\epsilon_i$'s are heavy-tailed), the regression estimator based on the least squares loss no longer converges at optimal rates. In addition, whereas the usual regularity assumptions on the design matrix such as the restricted eigenvalue condition have been shown to hold with high probability when the covariates are sub-Gaussian~\cite{RasEtal10, RudZho13}, we wish to devise estimators that are also consistent under weaker assumptions on the distribution of the covariates.

\item[(b)] \emph{Outliers.} Even when the covariates and error terms are normally distributed, the regression estimator may be inconsistent when observations are contaminated by outliers in the predictors and/or response variables~\cite{RouLer05}. Whereas the standard ordinary least squares loss function is non-robust to outliers in the observations, alternative estimators exist in a low-dimensional setting that are robust to a certain degree of contamination. We wish to extend this theory to high-dimensional regression estimators, as well.

\end{itemize}

\noindent Inspired by the classical theory on robust estimators for linear regression~\cite{Hub81, MarEtal06, HamEtal11}, we study regularized versions of low-dimensional robust regression estimators and establish statistical guarantees in a high-dimensional setting. As we will see, the regularized robust regression functions continue to enjoy good behavior in high dimensions, and we can quantify the degree to which the high-dimensional estimators are robust to the types of deviations described above.

Our first main contribution is to provide a general set of sufficient conditions under which optima of regularized robust $M$-estimators are statistically consistent, even in the presence of heavy-tailed errors and outlier contamination. The conditions involve a bound on the derivative of the regression function, as well as restricted strong convexity of the loss function in a neighborhood of constant radius about the true parameter vector, and the conclusions are given in terms of the tails of the error distribution. The notion of restricted strong convexity, as used previously in the literature~\cite{NegEtal12, AgaEtal12, LohWai13, LohWai14}, traditionally involves a global condition on the behavior of the loss function. However, due to the highly nonconvex behavior of the robust regression functions of interest, we assume only a \emph{local} condition of restricted strong convexity in the development of our statistical results. Consequently, our main theorem provides guarantees only for stationary points within the local region of strong curvature. We show that all such local stationary points are statistically consistent estimators for the true regression vector; when the covariates are sub-Gaussian, the rate of convergence agrees (up to a constant factor) with the rate of convergence for $\ell_1$-penalized ordinary least squares regression with sub-Gaussian errors. We also use the same framework to study generalized $M$-estimators and provide results for statistical consistency of local stationary points under weaker distributional assumptions on the covariates.

The wide applicability of our theorem on statistical consistency of high-dimensional robust $M$-estimators opens the door to an important question regarding the design of robust regression estimators, which is the topic of our second contribution: In the setting of heavy-tailed errors, if all regression estimators with bounded derivative are statistically consistent with rates agreeing up to a constant factor, what are the advantages of using a complicated nonconvex regression function over a simple convex function such as the Huber loss? In the low-dimensional setting, several independent lines of work provide reasons for using nonconvex $M$-estimators over their convex alternatives~\cite{Hub81, SheEtal08}. One compelling justification is from the viewpoint of statistical efficiency. Indeed, the log likelihood function of the heavy-tailed $t$-distribution with one degree of freedom gives rise to the nonconvex Cauchy loss, which is consequently asymptotically efficient~\cite{LehCas98}. In our second main theorem, we prove that by using a suitable nonconvex regularizer~\cite{FanLi01, Zha10}, we may guarantee that local stationary points of the regularized robust $M$-estimator agree with a local oracle solution defined on the correct support. Thus, provided the sample size scales sufficiently quickly with the level of sparsity, results on asymptotic normality of low-dimensional $M$-estimators with a diverging number of parameters~\cite{Hub73, YohMar79, Por85, Mam89, HeSha00} may be used to establish asymptotic normality of the corresponding high-dimensional estimators, as well. In particular, when the loss function equals the negative log likelihood of the error distribution, stationary points of the high-dimensional $M$-estimator will also be efficient in an asymptotic sense. Our oracle result and subsequent conclusions regarding asymptotic normality resemble a variety of other results in the literature on nonconvex regularization~\cite{FanPen04, BraEtal11, LiEtal11}, but our result is stronger because it provides guarantees for \emph{all} stationary points in the local region. Our proof technique leverages the primal-dual witness construction recently proposed in Loh and Wainwright~\cite{LohWai14}; however, we require a more refined analysis here in order to extend the result to one involving only local properties of the loss function.

Our third and final contribution addresses algorithms used optimize our proposed $M$-estimators. Since our statistical consistency and oracle results only provide guarantees for the behavior of \emph{local} solutions, we need to devise an optimization algorithm that always converges to a stationary point inside the local region. Indeed, local optima that are statistically inconsistent are the bane of nonconvex $M$-estimators, even in low-dimensional settings~\cite{FreDia82}. To remedy this issue, we propose a novel two-step algorithm that is \emph{guaranteed} to converge to a stationary point within the local region of restricted strong convexity. Our algorithm consists of optimizing two separate regularized $M$-estimators in succession, and may be applied to situations where both the loss and regularizer are nonconvex. In the first step, we optimize a convex regularized $M$-estimator to obtain a sufficiently close point that is then used to initialize an optimization algorithm for the original (nonconvex) $M$-estimator in the second step. We use the composite gradient descent algorithm~\cite{Nes07} in both steps of the algorithm, and prove rigorously that if the initial point in the second step lies within the local region of restricted curvature, all successive iterates will continue to lie in the region and converge at a linear rate to an appropriate stationary point. Any convex, statistically consistent $M$-estimator suffices for the first step; we use the $\ell_1$-penalized Huber loss in our simulations involving sub-Gaussian covariates with heavy-tailed errors, since global optima are statistically consistent by our earlier theory. Our resulting two-step estimator, which first optimizes a convex Huber loss to obtain a consistent estimator and then optimizes a (possibly nonconvex) robust $M$-estimator to obtain a more efficient estimator, is reminiscent of the one-step estimators common in the robust regression literature~\cite{Bic75}---however, here we require full runs of composite gradient descent in each step of the algorithm, rather than a single Newton-Raphson step. Note that if the goal is to optimize an $M$-estimator involving a convex loss and nonconvex regularizer, such as the SCAD-penalized Huber loss, our two-step algorithm is also applicable, where we optimize the $\ell_1$-penalized loss in the first step.

%%%%%%

\paragraph{Related work:}

We close this section by highlighting three recent papers on related topics. The analysis in this paper most closely resembles the work of Lozano and Meinshausen~\cite{LozMei13}, in that we study stationary points of nonconvex functions used for robust high-dimensional linear regression within a local neighborhood of the true regression vector. Although the technical tools we use here are similar, we focus on regression functions are expressible as $M$-estimators; the minimum distance loss function proposed in that paper does not fall into this category. In addition, we formalize the notion of basins of attraction for optima of nonconvex $M$-estimators and develop a two-step optimization algorithm that consists of optimizing successive regularized $M$-estimators, which goes beyond their results about local convergence of a composite gradient descent algorithm.

Another related work is that of Fan et al.~\cite{FanEtal14}. While that paper focuses exclusively on developing estimation bounds for penalized robust regression with the Huber loss function, the results presented in our paper are strictly more general, since they hold for \emph{nonconvex} $M$-estimators, as well. The analysis of the $\ell_1$-penalized Huber loss is still relevant to our analysis, however, because as shown below, its global convergence guarantees provide us with a good initialization point for the composite gradient algorithm that we will apply in the first step of our two-step algorithm.

Finally, we draw attention to the recent work by Mendelson~\cite{Men14}. In that paper, careful derivations based on empirical process theory demonstrate the advantage of using differently parametrized convex loss functions tuned according to distributional properties of the additive noise in the model. Our analysis also reveals the impact of different parameter choices for the regression function on the resulting estimator, but the rates of Mendelson~\cite{Men14} are much sharper than ours (albeit agreeing up to a constant factor). However, our analysis is not limited to convex loss functions, and covers nonconvex loss functions possessing local curvature, as well. Finally, note that while Mendelson~\cite{Men14} is primarily concerned with optimizing the regression estimator with respect to $\ell_1$- and $\ell_2$-error, our oracle results suggest that it may be instructive to consider second-order properties as well. Indeed, taking into account attributes such as the variance and asymptotic efficiency of the estimator may lead to a different parameter choice for a robust loss function than if the primary goal is to minimize the bias alone. \\

The remainder of our paper is organized as follows: In Section~\ref{SecBackground}, we provide the basic background concerning $M$- and generalized $M$-estimators, and introduce various robust loss functions and regularizers to be discussed in the sequel. In Section~\ref{SecMain}, we present our main theorem concerning statistical consistency of robust high-dimensional $M$-estimators and unpack the distributional conditions required for the assumptions of the theorem to hold for specific robust estimators through a series of propositions. We also present our main theorem concerning oracle properties of nonconvex regularized $M$-estimators, with a corollary illustrating the types of asymptotic normality conclusions that may be derived from the oracle result. Section~\ref{SecOpt} provides our two-step optimization algorithm and corresponding theoretical guarantees. We conclude in Section~\ref{SecSims} with a variety of simulation results. A brief review of robustness measures is provided in Appendix~\ref{AppRobust}, and proofs of the main theorems and all supporting lemmas and propositions are contained in the remaining supplementary appendices.

%%%%%

\paragraph{Notation:}

For functions $f(n)$ and $g(n)$, we write $f(n) \precsim g(n)$ to mean
that \mbox{$f(n) \le c g(n)$} for some universal constant $c \in (0,
\infty)$, and similarly, $f(n) \succsim g(n)$ when \mbox{$f(n) \ge c'
  g(n)$} for some universal constant $c' \in (0, \infty)$. We write
$f(n) \asymp g(n)$ when $f(n) \precsim g(n)$ and $f(n) \succsim g(n)$
hold simultaneously. For a vector $v \in \real^p$ and a subset $S
\subseteq \{1, \dots, p\}$, we write $v_S \in \real^S$ to denote the
vector $v$ restricted to $S$. For a matrix $M$, we write $\opnorm{M}_2$ to denote the spectral
norm. For a function $h: \real^p \rightarrow \real$, we write $\nabla h$ to denote
a gradient or subgradient of the function.

%%%%%%%%%%%%%%

\section{Background and problem setup}
\label{SecBackground}

In this section, we provide some background on $M$- and generalized $M$-estimators for robust regression. We also describe the classes of nonconvex regularizers that will be covered by our theory.

Throughout, we will assume that we have $n$ i.i.d.\ observations $\{(x_i, y_i)\}_{i=1}^n$ from the linear model
\begin{equation}
\label{EqnLinModel}
y_i = x_i^T \betastar + \epsilon_i, \qquad \forall 1 \le i \le n,
\end{equation}
where $x_i \in \real^p$, $y_i \in \real$, and $\betastar \in \real^p$ is a $k$-sparse vector. We also assume that $x_i \condind \epsilon_i$ and both are zero-mean random variables. We are interested in high-dimensional regression estimators of the form
\begin{equation}
\label{EqnGeneral}
\betahat \in \arg\min_{\|\beta\|_1 \le R} \left\{\Loss_n(\beta) + \rho_\lambda(\beta)\right\},
\end{equation}
where $\Loss_n$ is the empirical loss function and $\rho_\lambda$ is a penalty function. For instance, the Lasso program is given by the loss $\Loss_n(\beta) = \frac{1}{n} \sum_{i=1}^n (x_i^T \beta - y_i)^2$ and penalty $\rho_\lambda(\beta) = \lambda \|\beta\|_1$, but this framework allows for much more general settings. Since we are interested in cases where the loss and regularizer may be nonconvex, we include the side condition $\|\beta\|_1 \le R$ in the program~\eqref{EqnGeneral} in order to guarantee the existence of local/global optima. We will require $R \ge \|\betastar\|_1$, so that the true regression vector $\betastar$ is feasible for the program.

In the scenarios below, we will consider loss functions $\Loss_n$ that satisfy
\begin{equation}
\label{EqnZeroG}
\E\left[\nabla \Loss_n(\betastar)\right] = 0.
\end{equation}
When the population-level loss $\Loss(\beta) \defn \E[\Loss_n(\beta)]$ is a convex function, equation~\eqref{EqnZeroG} implies that $\betastar$ is a global optimum of $\Loss(\beta)$. When $\Loss$ is nonconvex, the condition~\eqref{EqnZeroG} ensures that $\betastar$ is at least a stationary point of the function. Our goal is to develop conditions under which certain stationary points of the program~\eqref{EqnGeneral} are statistically consistent estimators for $\betastar$.

%%%%%%

\subsection{Robust $M$-estimators}
\label{SecMEst}

We wish to study loss functions $\Loss_n$ that are robust to outliers and/or model misspecification. Consequently, we borrow our loss functions from the classical theory of robust regression in low dimensions; the additional regularizer $\rho_\lambda$ appearing in the program~\eqref{EqnGeneral} encourages sparsity in the solution and endows it with appealing behavior in high dimensions. Here, we provide a brief review of $M$-estimators used for robust linear regression. For a more detailed treatment of the basic concepts of robust regression, see the books~\cite{Hub81, MarEtal06, HamEtal11} and the many references cited therein.

Let $\ell$ denote the regression function defined on an individual observation pair $(x_i, y_i)$. The corresponding $M$-estimator is then given by
\begin{equation}
\label{EqnLoss}
\Loss_n(\beta) = \frac{1}{n} \sum_{i=1}^n \ell(x_i^T \beta - y_i).
\end{equation}
Note that 
\begin{equation*}
\E\left[\nabla \Loss_n(\betastar)\right] = \E\left[\ell'(x_i^T \betastar - y_i) x_i\right] = \E\left[\ell'(\epsilon_i) x_i\right] = \E\left[\ell'(\epsilon_i)\right] \cdot \E\left[x_i\right] = 0,
\end{equation*}
so the condition~\eqref{EqnZeroG} is always satisfied. In particular, the maximum likelihood estimator corresponds to the choice $\ell(u) = - \log p_\epsilon(u)$, where $p_\epsilon$ is the probability density function of the additive errors $\epsilon_i$. Note that when $\epsilon_i \sim N(0, 1)$, the MLE corresponds to the choice $\ell(u) = \frac{u^2}{2}$, and the resulting loss function is convex. \\

Some of the loss functions that we will analyze in this paper include the following:

\paragraph{Huber loss:} We have
\begin{equation*}
\ell(u) =
\begin{cases}
\frac{u^2}{2}, & \text{if } |u| \le \xi, \\
\xi |u| - \frac{\xi^2}{2}, & \text{if } |u| > \xi.
\end{cases}
\end{equation*}
In this case, $\ell$ is a convex function. Although $\ell''$ does not exist everywhere, $\ell'$ is continuous and $\|\ell'\|_\infty \le \xi$.

\paragraph{Tukey's biweight:} We have
\begin{equation*}
\ell(u) =
\begin{cases}
\frac{\xi^2}{6} \left(1 - \left(1 - \frac{u^2}{\xi^2}\right)^3\right), & \text{if } |u| \le \xi, \\
\frac{\xi^2}{6}, & \text{if } |u| > \xi.
\end{cases}
\end{equation*}
Note that $\ell$ is \emph{nonconvex}. We also compute the first derivative
\begin{equation*}
\ell'(u) =
\begin{cases}
u \left(1 - \frac{u^2}{\xi^2}\right)^2, & \text{if } |u| \le \xi, \\
0, & \text{if } |u| > \xi,
\end{cases}
\end{equation*}
and second derivative
\begin{equation*}
\ell''(u) =
\begin{cases}
\left(1 - \frac{u^2}{\xi^2}\right) \left(1 - \frac{5u^2}{\xi^2}\right), & \text{if } |u| \le \xi, \\
0, & \text{if } |u| > \xi.
\end{cases}
\end{equation*}
Note that $\ell''$ is continuous. Furthermore, $\|\ell'\|_\infty \le \frac{16\xi}{25 \sqrt{5}}$. One may check that Tukey's biweight function is \emph{not} an MLE. Furthermore, although $\ell''$ exists everywhere and is continuous, $\ell'''$ does not exist for $u \in \left\{\pm \xi, \; \pm \frac{\xi}{\sqrt{5}}\right\}$. \\

\paragraph{Cauchy loss:} We have
\begin{equation*}
\ell(u) = \frac{\xi^2}{2} \log\left(1 + \frac{u^2}{\xi^2}\right).
\end{equation*}
Note that $\ell$ is \emph{nonconvex}. When $\xi = 1$, the function $\ell(u)$ is proportional to the MLE for the $t$-distribution with one degree of freedom (a heavy-tailed distribution). This suggests that for heavy-tailed distributions, nonconvex loss functions may be more desirable from the point of view of statistical efficiency, although optimization becomes more difficult; we will explore this idea more fully in Section~\ref{SecVariance} below. For the Cauchy loss, we have
\begin{equation*}
\ell'(u) = \frac{u}{1 + u^2/\xi^2}, \qquad \text{and} \qquad \ell''(u) = \frac{1 - u^2/\xi^2}{(1 + u^2/\xi^2)^2}.
\end{equation*}
In particular, $|\ell'(u)|$ is maximized when $u^2 = \xi^2$, so $\|\ell'\|_\infty \le \frac{\xi}{2}$. We may also check that $\|\ell''\|_\infty \le 1$ and $\|\ell'''\|_\infty \le \frac{3}{2\xi}$. \\

Although second and third derivatives do not always exist for the loss functions above, a unifying property is that the derivative $\ell'$ is \emph{bounded} in each case. This turns out to be an important property for robustness of the resulting estimator. Intuitively, we may view a solution $\betahat$ of the program~\eqref{EqnGeneral} as an approximate sparse solution to the estimating equation $\nabla \Loss_n(\beta) = 0$, or equivalently,
\begin{equation}
\label{EqnEstimating}
\frac{1}{n} \sum_{i=1}^n \ell'(x_i^T \beta - y_i) x_i = 0.
\end{equation}
When $\beta = \betastar$, equation~\eqref{EqnEstimating} becomes
\begin{equation}
\label{EqnEst}
\frac{1}{n} \sum_{i=1}^n \ell'(\epsilon_i) x_i = 0.
\end{equation}
In particular, if a pair $(x_i, y_i)$ satisfies the linear model~\eqref{EqnLinModel} but $\epsilon_i$ is an outlier, its contribution to the sum in equation~\eqref{EqnEst} is bounded when $\ell'$ is bounded, lessening the contamination effect of gross outliers.

In the robust regression literature, a \emph{redescending $M$-estimator} has the additional property that there exists $\xi_0 > 0$ such that $|\ell'(u)| = 0$, for all $|u| \ge \xi_0$. Then $\xi_0$ is known as a \emph{finite rejection point}, since outliers $(x_i, y_i)$ with $|\epsilon_i| \ge \xi_0$ will be completely eliminated from the summand in equation~\eqref{EqnEst}. For instance, Tukey's biweight function gives rise to a redescending $M$-estimator.\footnote{The Cauchy loss has the property that $\lim_{u \rightarrow \infty} |\ell'(u)| = 0$, but it is not redescending for any finite $\xi_0$.} Note that redescending $M$-estimators will always be nonconvex, so computational efficiency will be sacrificed at the expense of finite rejection properties. For an in-depth discussion of redescending $M$-estimators vis-\`{a}-vis different measures of robustness, see the article by Shevlyakov et al.~\cite{SheEtal08}.

%%%%%

\subsection{Generalized $M$-estimators}
\label{SecGMEst}

Whereas the $M$-estimators described in Section~\ref{SecMEst} are robust with respect to outliers in the additive noise terms $\epsilon_i$, they are non-robust to outliers in the covariates $x_i$. This may be quantified using the concept of influence functions (see Appendix~\ref{AppRobust}). Intuitively, an outlier in $x_i$ may cause the corresponding term in equation~\eqref{EqnEst} to behave arbitrarily badly. This motivates the use of \emph{generalized} $M$-estimators that downweight large values of $x_i$ (also known as leverage points). The resulting estimating equation is then defined as follows:
\begin{equation}
\label{EqnGeneralizedM}
\sum_{i=1}^n \eta(x_i, x_i^T \beta - y_i) x_i = 0,
\end{equation}
where $\eta: \real^p \times \real \rightarrow \real$ is defined appropriately. As will be discussed in the sequel, generalized $M$-estimators may allow us to relax the distributional assumptions on the covariates; e.g., from sub-Gaussian to sub-exponential.

We will focus on functions $\eta$ that take the form
\begin{equation}
\label{EqnDecompGM}
\eta(x_i, r_i) = w(x_i) \; \ell'(r_i \cdot v(x_i)),
\end{equation}
where $w, v > 0$ are weighting functions. Note that the $M$-estimators considered in Section~\ref{SecMEst} may also be written in this form, where $w \equiv v \equiv 1$. \\

Some popular choices of $\eta$ of the form presented in equation~\eqref{EqnDecompGM} include the following:

\paragraph{Mallows estimator~\cite{Mal75}:} We take $v(x) \equiv 1$ and $w(x)$ of the form
\begin{equation}
\label{EqnWDefn}
w(x) = \min\left\{1, \frac{b}{\|Bx\|_2}\right\}, \qquad \text{or} \qquad w(x) = \min\left\{1, \frac{b^2}{\|Bx\|_2^2}\right\},
\end{equation} 
for parameters $b > 0$ and $B \in \real^{p \times p}$. Note that $\|w(x)x\|_2$ is indeed bounded for a fixed choice of $b$ and $B$, since
\begin{equation*}
\|w(x) x\|_2 \le \frac{\|bx\|_2}{\|Bx\|_2} \le b \opnorm{B^{-1}}_2 \defn b_0.
\end{equation*}
The Mallows estimator effectively shrinks data points for which $\|x_i\|_2$ is large toward an elliptical shell defined by $B$, and likewise pushes small data points closer to the shell. \\

\paragraph{Hill-Ryan estimator~\cite{Hil77}:} We take $v(x) = w(x)$, where $w$ is defined such that $\|w(x)x\|_2$ is bounded (e.g., equation~\eqref{EqnWDefn}). In addition to downweighting the influence function similarly to the Mallows estimator, the Hill-Ryan estimator scales the residuals according to the leverage weight of the $x_i$'s. \\

\paragraph{Schweppe estimator~\cite{MerSch71}:} For a parameter $B \in \real^{p \times p}$, we take $w(x) = \frac{1}{\|Bx\|_2}$ and $v(x) = \frac{1}{w(x)}$. Like the Mallows estimator, Schweppe's estimator downweights the contribution of data points with high leverage as a summand in the estimating equation~\eqref{EqnGeneralizedM}. If $\ell'$ is locally linear around the origin and flattens out for larger values, Schweppe's estimator additionally dampens the effect of a residual $r_i$ only when it is large compared to the leverage of $x_i$. As discussed in Hampel et al.~\cite{HamEtal11}, Schweppe's estimator is designed to be optimal in terms of a measure of variance robustness, subject to a bound on the influence function. \\

\begin{comment}
However, in order for an optimal choice of $B$ to exist that minimizes variance, the parameter $\gamma$ defining the robust regression function $\ell$ must scale with $\sqrt{p}$. Since we require $\gamma$ to be bounded by a constant to retain statistical consistency of the estiamator in a high-dimensional setting, the usefulness of Schweppe's estimator is less clear and we do not analyze it specifically in this paper.
\end{comment}

Note that when $\eta$ takes the form in equation~\eqref{EqnDecompGM}, the estimating equation~\eqref{EqnGeneralizedM} may again be seen as a zero-gradient condition $\nabla \Loss_n(\beta) = 0$, where
\begin{equation}
\label{EqnLossWeight}
\Loss_n(\beta) \defn \frac{1}{n} \sum_{i=1}^n \frac{w(x_i)}{v(x_i)} \; \ell\left((x_i^T \beta - y_i)v(x_i)\right).
\end{equation}
Under reasonable conditions, such as oddness of $\ell'$ and symmetry of the error distribution, the condition~\eqref{EqnZeroG} may be seen to hold (cf.\ condition 2 of Proposition~\ref{PropGradBound} below and the following remark). The overall program for a generalized $M$-estimator then takes the form
\begin{equation*}
\betahat \in \arg\min_{\|\beta\|_1 \le R} \left\{\frac{1}{n} \sum_{i=1}^n \frac{w(x_i)}{v(x_i)} \; \ell\left((x_i^T \beta - y_i)v(x_i)\right) + \rho_\lambda(\beta) \right\}.
\end{equation*}

%%%%

\subsection{Nonconvex regularizers}

Finally, we provide some background on the types of regularizers we will use in our analysis of the composite objective function~\eqref{EqnGeneral}. Following the theoretical development of Loh and Wainwright~\cite{LohWai13, LohWai14}, we require the regularizer $\rho_\lambda$ to satisfy the following properties:

\begin{assumption} [Amenable regularizers]
\label{AsAmenable}
The regularizer is coordinate-separable:
\begin{equation*}
\rho_\lambda(\beta) = \sum_{j=1}^p \rho_\lambda(\beta_j),
\end{equation*}
for some scalar function $\rho_\lambda: \real \mapsto \real$. In addition:
\begin{itemize}
\item[(i)] The function $t \mapsto \rho_\lambda(t)$ is symmetric around zero and $\rho_\lambda(0) = 0$.
\item[(ii)] The function $t \mapsto \rho_\lambda(t)$ is nondecreasing on $\real^+$.
\item[(iii)] The function $t \mapsto \frac{\rho_\lambda(t)}{t}$ is nonincreasing on $\real^+$.
\item[(iv)] The function $t \mapsto \rho_\lambda(t)$ is differentiable for $t \neq 0$.
\item[(v)] $\lim_{t \rightarrow 0^+} \rho'_\lambda(t) = \lambda$.
\item[(vi)] There exists $\mu > 0$ such that the function $t \mapsto \rho_\lambda(t) + \frac{\mu}{2} t^2$ is convex.
\item[(vii)] There exists $\gamma \in (0, \infty)$ such that $\rho_\lambda'(t) = 0$ for all $t \ge \gamma \lambda$.
\end{itemize}
\end{assumption}
If $\rho_\lambda$ satisfies conditions (i)--(vi) of Assumption~\ref{AsAmenable}, we say that $\rho_\lambda$ is $\mu$-\emph{amenable}. If $\rho_\lambda$ also satisfies condition (vii), we say that $\rho_\lambda$ is $(\mu, \gamma)$-\emph{amenable}~\cite{LohWai14}. In particular, if $\rho_\lambda$ is $\mu$-amenable, then $q_\lambda(t) \defn \lambda |t| - \rho_\lambda(t)$ is everywhere differentiable. Defining the vector version $q_\lambda: \real^p \rightarrow \real$ accordingly, it is easy to see that $\frac{\mu}{2} \|\beta\|_2^2 - q_\lambda(\beta)$ is convex. \\

Some examples of amenable regularizers are the following:

\paragraph{Smoothly clipped absolute deviation (SCAD) penalty:}  This 
penalty, due to Fan and Li~\cite{FanLi01}, takes the form
\begin{align}
\label{EqnSCADdefn}
\rho_\lambda(t) & \defn \begin{cases} \lambda |t|, & \mbox{for $|t|
    \le \lambda$,} \\
-\frac{t^2 - 2a\lambda |t| + \lambda^2}{2(a-1)}, & \mbox{for $\lambda <
  |t| \le a \lambda$,} \\
\frac{(a+1)\lambda^2}{2}, & \mbox{for $|t| > a\lambda$},
\end{cases}
\end{align}
where $a > 2$ is fixed. The SCAD penalty is $(\mu, \gamma)$-amenable, with $\mu =
\frac{1}{a-1}$ and $\gamma = a$. \\

\paragraph{Minimax concave penalty (MCP):} This penalty, due to 
Zhang~\cite{Zha12}, takes the form
\begin{align}
\label{EqnMCPdefn}
\rho_\lambda(t) & \defn \sign(t) \, \lambda \cdot \int_0^{|t|} \left(1
- \frac{z}{\lambda b}\right)_+ dz,
\end{align}
where $b > 0$ is fixed. The MCP regularizer is $(\mu,
\gamma)$-amenable, with $\mu = \frac{1}{b}$ and $\gamma =
b$. \\

\paragraph{Standard $\ell_1$-penalty:}  The $\ell_1$-penalty $\rho_\lambda(t) = \lambda |t|$ is an example of a regularizer that is $0$-amenable, but not
$(0, \gamma)$-amenable, for any $\gamma < \infty$. \\

As studied in detail in Loh and Wainwright~\cite{LohWai14} and leveraged in the results of Section~\ref{SecVariance} below, using $(\mu, \gamma)$-amenable regularizers allows us to derive a powerful oracle result concerning local stationary points, which will be useful for our discussion of asymptotic normality.

%%%%%

\section{Main statistical results}
\label{SecMain}

We now present our core statistical results concerning stationary points of the high-dimensional robust $M$-estimators described in Section~\ref{SecBackground}. We begin with a general deterministic result that ensures statistical consistency of stationary points of the program~\eqref{EqnGeneral} when the loss function satisfies restricted strong convexity and the regularizer is $\mu$-amenable. Next, we interpret the consequences of our theorem for specific $M$-estimators and generalized $M$-estimators through a series of propositions, and provide conditions on the distributions of the covariates and error terms in order for the assumptions of the theorem to hold with high probability. Lastly, we provide a theorem establishing that stationary points are equal to a local oracle estimator when the regularizer is nonconvex and $(\mu, \gamma)$-amenable.

Recall that $\betatil$ is a \emph{stationary point} of the program~\eqref{EqnGeneral} if
\begin{equation*}
\inprod{\nabla \Loss_n(\betatil) + \nabla \rho_\lambda(\betatil)}{\beta - \betatil} \ge 0,
\end{equation*}
for all feasible $\beta$, where with a slight abuse of notation, we write $\nabla \rho_\lambda(\betatil) = \lambda \sign(\betatil) - \nabla q_\lambda(\betatil)$ (recall that $q_\lambda$ is differentiable by our assumptions). In particular, the set of stationary points includes all local and global minima, as well as interior local maxima~\cite{Ber99, Clarke83}.

%%%%%%

\subsection{General statistical theory}
\label{SecStat}

We require the loss function $\Loss_n$ to satisfy the following local RSC condition:

\begin{assumption} [RSC condition]
\label{AsLocalRSC}
There exist $\alpha, \tau > 0$ and a radius $r > 0$ such that
\begin{equation}
\label{EqnLocalRSC}
\inprod{\nabla \Loss_n(\beta_1) - \nabla \Loss_n(\beta_2)}{\beta_1 - \beta_2} \ge \alpha \|\beta_1 - \beta_2\|_2^2 - \tau \frac{\log p}{n} \|\beta_1 - \beta_2\|_1^2,
\end{equation}
for all $\beta_1, \beta_2 \in \real^p$ such that $\|\beta_1 - \betastar\|_2, \|\beta_2 - \betastar\|_2 \le r$.
\end{assumption}
Note that the condition~\eqref{EqnLocalRSC} imposes \emph{no} conditions on the behavior of $\Loss_n$ outside the ball of radius $r$ centered at $\betastar$. In this way, it differs from the RSC condition used in Loh and Wainwright~\cite{LohWai13}, where a weaker inequality is assumed to hold for vectors outside the local region. This paper focuses on the local behavior of stationary points around $\betastar$, since the loss functions used for robust regression may be more wildly nonconvex away from the origin. As discussed in more detail below, we will take $r$ to scale as a constant independent of $n, p$, and $k$. The ball of radius $r$ essentially cuts out a local basin of attraction around $\betastar$ in which stationary points of the $M$-estimator are well-behaved. Furthermore, our optimization results in Section~\ref{SecOpt} guarantee that we may efficiently locate stationary points within this constant-radius region via a two-step $M$-estimator. \\

We have the following main result, which requires the regularizer and loss function to satisfy the conditions of Assumptions~\ref{AsAmenable} and~\ref{AsLocalRSC}, respectively. The theorem guarantees that stationary points within the local region where the loss function satisfies restricted strong convexity are statistically consistent.

\begin{thm*}
\label{ThmStationary}
Suppose $\Loss_n$ satisfies the RSC condition~\eqref{EqnLocalRSC} with $\beta_2 = \betastar$ and $\rho_\lambda$ is $\mu$-amenable, with $\frac{3}{4} \mu < \alpha$. Suppose $n \ge Cr^2 \cdot k \log p$ and $R \ge \|\betastar\|_1$ and
\begin{equation}
\label{EqnGradBound}
\lambda \ge \max\left\{4 \|\nabla \Loss_n(\betastar)\|_\infty, \; 8\tau R\frac{\log p}{n}\right\}.
\end{equation}
Let $\betatil$ be a stationary point of the program~\eqref{EqnGeneral} such that $\|\betatil - \betastar\|_2 \le r$. Then $\betatil$ exists and satisfies the bounds
\begin{equation*}
\|\betatil - \betastar\|_2 \le \frac{24 \lambda \sqrt{k}}{4\alpha - 3\mu}, \qquad \text{and} \qquad \|\betatil - \betastar\|_1 \le \frac{96 \lambda k}{4\alpha - 3\mu}.
\end{equation*}
\end{thm*}
\noindent The proof of Theorem~\ref{ThmStationary} is contained in Section~\ref{SecThmStationary}. Note that the statement of Theorem~\ref{ThmStationary} is entirely deterministic, and the distributional properties of the covariates and error terms in the linear model come into play in verifying that the inequality~\eqref{EqnGradBound} and the RSC condition~\eqref{EqnLocalRSC} hold with high probability under the prescribed sample size scaling.

\paragraph{Remark:} Although Theorem~\ref{ThmStationary} only guarantees the statistical consistency of stationary points within the local region of radius $r$, it is essentially the strongest conclusion one can draw based on the local RSC assumption~\eqref{EqnLocalRSC} alone. The power of Theorem~\ref{ThmStationary} lies in the fact that when $r$ is chosen to be a constant and $\frac{n}{k \log p} = o(1)$, as is the case in our robust regression settings of interest, all stationary points within the constant-radius region are actually guaranteed to fall within a shrinking ball of radius $\order\left(\sqrt{\frac{k \log p}{n}}\right)$ centered around $\betastar$. Hence, the stationary points in the local region are statistically consistent at the usual minimax rate expected for $\ell_1$-penalized ordinary least squares regression with sub-Gaussian data. As we will illustrate in more detail in the next section, if robust loss functions with bounded derivatives are used in place of the ordinary least squares loss, the statistical consistency conclusion of Theorem~\ref{ThmStationary} still holds even when the additive errors follow a heavy-tailed distribution or are contaminated by outliers.

%%%

\subsection{Establishing sufficient conditions}
\label{SecRSC}

From Theorem~\ref{ThmStationary}, we see that the key ingredients for statistical consistency of local stationary points are (i) the boundedness of $\|\nabla \Loss_n(\betastar)\|_\infty$ in inequality~\eqref{EqnGradBound}, which ultimately dictates the $\ell_2$-rate of convergence of $\betatil$ to $\betastar$ up to a factor of $\sqrt{k}$, and (ii) the local RSC condition~\eqref{EqnLocalRSC} in Assumption~\ref{AsLocalRSC}. We provide more interpretable sufficient conditions in this section via a series of propositions.

For the results of this section, we will require some boundedness conditions on the derivatives of the loss function $\ell$, which we state in the following assumption:
\begin{assumption}
\label{AsEllDeriv}
Suppose there exist $\kappa_1, \kappa_2 \ge 0$ such that
\begin{align}
\label{EqnKappa1}
|\ell'(u)| \le \kappa_1, & \qquad \forall u, \\
\label{EqnKappa2}
\ell''(u) \ge - \kappa_2, & \qquad \forall u.
\end{align}
\end{assumption}
Note that the bounded derivative assumption~\eqref{EqnKappa1} holds for all the robust loss functions highlighted in Section~\ref{SecBackground} (but \emph{not} for the ordinary least squares loss), and $\kappa_1 \asymp \xi$ in each of those cases. Furthermore, inequality~\eqref{EqnKappa2} holds with $\kappa_2 = 0$ when $\ell$ is convex and twice-differentiable, but the inequality also holds for nonconvex losses such as the Tukey and Cauchy loss with $\kappa_2 > 0$. By a more careful argument, we may eschew the condition~\eqref{EqnKappa2} if $\ell$ is a convex function that is in $C^1$ but not $C^2$, as in the case of the Huber loss, since Theorem~\ref{ThmStationary} only requires first-order differentiability of $\Loss_n$ and $q_\lambda$; however, we state the propositions with Assumption~\ref{AsEllDeriv} for the sake of simplicity.

We have the following proposition, which establishes the gradient bound~\eqref{EqnGradBound} with high probability under fairly mild assumptions:
\begin{prop*}
\label{PropGradBound}
Suppose $\ell$ satisfies the bounded derivative condition~\eqref{EqnKappa1} and the following conditions also hold:
\begin{itemize}
\item[(1)] $w(x_i) x_i$ is sub-Gaussian with parameter $\sigma_w^2$.
\item[(2)] Either
\begin{itemize}
\item[(a)] $v(x_i) = 1$ and $\E[w(x_i) x_i] = 0$, or
\item[(b)] $\E\left[\ell'\left(\epsilon_i \cdot v(x_i)\right) \mid x_i\right] = 0$.
\end{itemize}
\end{itemize}
With probability at least $1 - c_1 \exp(-c_2 \log p)$, the loss function defined by equation~\eqref{EqnLossWeight} satisfies the bound
\begin{equation*}
\|\nabla \Loss_n(\betastar)\|_\infty \le c \kappa_1 \sigma_w \sqrt{\frac{\log p}{n}}.
\end{equation*}
\end{prop*}

\noindent The proof of Proposition~\ref{PropGradBound} is a simple but important application of sub-Gaussian tail bounds and is provided in Appendix~\ref{AppPropGradBound}. 

\paragraph{Remark:}

Note that for the unweighted $M$-estimator~\eqref{EqnLoss}, conditions (1) and (2a) of Proposition~\ref{PropGradBound} hold when $x_i$ is sub-Gaussian and $\E[x_i] = 0$. If the $x_i$'s are not sub-Gaussian, condition (1) nonetheless holds whenever $w(x_i) x_i$ is bounded. Furthermore, condition (2b) holds whenever $\epsilon_i$ has a symmetric distribution and $\ell'$ is an odd function. We further highlight the fact that aside from a possible mild requirement of symmetry, the concentration result given in Proposition~\ref{PropGradBound} is \emph{independent of the distribution of $\epsilon_i$}, and holds equally well for heavy-tailed error distributions. The distributional effect of the $x_i$'s is captured in the sub-Gaussian parameter $\sigma_w$; in settings where the contaminated data still follow a sub-Gaussian distribution, but the sub-Gaussian parameter is inflated due to large leverage points, using a weight function as defined in equation~\eqref{EqnWDefn} may lead to a significant decrease in the value of $\sigma_w$. This decreases the finite-sample bias of the overall estimator. \\

Establishing the local RSC condition in Assumption~\ref{AsLocalRSC} is more subtle, and the propositions described below depend in a more complex fashion on the distribution of the $\epsilon_i$'s. As noted above, the statistical consistency result in Theorem~\ref{ThmStationary} only requires Assumption~\ref{AsLocalRSC} to hold when $\beta_2 = \betastar$. However, for the stronger oracle result of Theorem~\ref{ThmOracle}, we will require the full form of Assumption~\ref{AsLocalRSC} to hold over all pairs $(\beta_1, \beta_2)$ in the local region. We will quantify the parameters of the RSC condition in terms of an additional parameter $T > 0$, which is treated as a fixed constant. Define the tail probability
\begin{equation}
\label{EqnEpsilonTau}
\epsilon_T \defn \mprob\left(|\epsilon_i| \ge \frac{T}{2}\right),
\end{equation}
and the lower-curvature bound
\begin{equation}
\label{EqnAlphaTau}
\alpha_T \defn \min_{|u| \le T} \ell''(u) > 0,
\end{equation}
where $\ell''$ is assumed to exist on the interval $[-T, T]$. We assume that $T$ is chosen small enough so that $\alpha_T > 0$.

We first consider the case where the loss function takes the usual form of an unweighted $M$-estimator~\eqref{EqnLoss}. We have the following proposition, proved in Appendix~\ref{AppPropMRSC}:
\begin{prop*}
\label{PropMRSC}
Suppose the $x_i$'s are drawn from a sub-Gaussian distribution with parameter $\sigma_x^2$ and the loss function is defined by equation~\eqref{EqnLoss}. Also suppose the bound
\begin{equation}
\label{EqnTricky}
c \sigma_x^2 \left( \epsilon_T^{1/2} + \exp\left(- \frac{c' T^2}{\sigma_x^2 r^2}\right)\right) \le \frac{\alpha_T}{\alpha_T + \kappa_2} \cdot \frac{\lambda_{\min}(\Sigma_x)}{2}
\end{equation}
holds. Suppose $\ell$ satisfies Assumption~\ref{AsEllDeriv}, and suppose the sample size satisfies $n \ge c_0 k \log p$. With probability at least $1 - c \exp(-c' \log p)$, the loss function $\Loss_n$ satisfies Assumption~\ref{AsLocalRSC} with
\begin{equation*}
\alpha = \alpha_T \cdot \frac{\lambda_{\min}(\Sigma_x)}{16}, \qquad \text{and} \qquad \tau = \frac{C(\alpha_T + \kappa_2)^2 \sigma_x^2 T^2}{r^2}.
\end{equation*}
\end{prop*}

\paragraph{Remark:}

Note that for a fixed value of $T$, inequality~\eqref{EqnTricky} places a tail condition on the distribution of $\epsilon_i$ via the term $\epsilon_T$. This may be interpreted as a bound on the variance of the error distribution when $\epsilon_i$ is sub-Gaussian, or a bound on the fraction of outliers when $\epsilon_i$ has a contaminated distribution. Furthermore, the exponential term decreases as a function of the ratio $\frac{T}{r}$. Hence, for a larger value of $\epsilon_T$, the radius $r$ will need to be smaller in order to satisfy the bound~\eqref{EqnTricky}. This agrees with the intuition that the local basin of good behavior for the $M$-estimator is smaller for larger levels of contamination. Finally, note that although $\alpha_T$ and $\kappa_2$ are deterministic functions of the known regression function $\ell$ and could be computed, the values of $\lambda_{\min}(\Sigma_x)$ and $\sigma_x^2$ are usually unknown a priori. Hence, Proposition~\ref{PropMRSC} should be viewed as more of a qualitative result describing the behavior of the RSC parameters as the amount of contamination of the error distribution increases, rather than a bound that can be used to select a suitable robust loss function. \\

The situation where $\Loss_n$ takes the form of a generalized $M$-estimator~\eqref{EqnLossWeight} is more difficult to analyze in its most general form, so we will instead focus on verifying the RSC condition~\eqref{EqnLocalRSC} for the Mallows and Hill-Ryan estimators described in Section~\ref{SecGMEst}. We will show that the RSC condition holds under weaker conditions on the distribution of the $x_i$'s. We have the following lemmas, proved in Appendices~\ref{AppPropMallows} and~\ref{AppPropHillRyan}:
\begin{prop*} [Mallows estimator]
\label{PropMallows}
Suppose the $x_i$'s are drawn from a sub-exponential distribution with parameter $\sigma_x^2$ and the loss function is defined by
\begin{equation*}
\Loss_n(\beta) = \frac{1}{n} \sum_{i=1}^n w(x_i) \ell(x_i^T \beta - y_i),
\end{equation*}
and $w(x_i) = \min\left\{1, \; \frac{b}{\|Bx_i\|_2}\right\}$. Also suppose the bound
\begin{equation*}
cb \opnorm{B^{-1}}_2\sigma_x^2 \left(\epsilon_T^{1/2} + \exp\left(-\frac{c'T}{\sigma_x r}\right)\right) \le \frac{\alpha_T}{2(\alpha_T + \kappa_2)} \cdot \lambda_{\min} \left(\E\left[w(x_i) x_i x_i^T\right]\right)
\end{equation*}
holds. Suppose $\ell$ satisfies Assumption~\ref{AsEllDeriv}, and suppose the sample size satisfies $n \ge c_0 k \log p$. With probability at least $1 - c \exp(-c' \log p)$, the loss function $\Loss_n$ satisfies Assumption~\ref{AsLocalRSC} with
\begin{equation*}
\alpha = \alpha_T \cdot \frac{\lambda_{\min}\left(\E\left[w(x_i) x_i x_i^T\right]\right)}{16}, \qquad \text{and} \qquad \tau = \frac{C(\alpha_T + \kappa_2)^2 \sigma_x^2 T^2}{r^2}.
\end{equation*}
\end{prop*}

\begin{prop*} [Hill-Ryan estimator]
\label{PropHillRyan}
Suppose the loss function is defined by
\begin{equation*}
\Loss_n(\beta) = \frac{1}{n} \sum_{i=1}^n w(x_i) \ell\left((x_i^T \beta - y_i) w(x_i)\right),
\end{equation*}
where $w(x_i) = \min\left\{1, \; \frac{b}{\|Bx_i\|_2}\right\}$. Also suppose the bound
\begin{equation}
\label{EqnTrickyWeight}
cb^2 \opnorm{B^{-1}}_2^2 \left(\epsilon_T^{1/2} + \exp\left(- \frac{c'T^2}{b^2 \opnorm{B^{-1}}_2^2 \sigma_x^2 r^2}\right)\right) \le \frac{\alpha_T}{2(\alpha_T + \kappa_2)} \cdot \lambda_{\min}\left(\E\left[w(x_i) x_i x_i^T \right]\right)
\end{equation}
holds. Suppose $\ell$ satisfies Assumption~\ref{AsEllDeriv}, and suppose the sample size satisfies $n \ge c_0 k \log p$. With probability at least $1 - c \exp(-c' \log p)$, the loss function $\Loss_n$ satisfies Assumption~\ref{AsLocalRSC} with
\begin{equation*}
\alpha = \alpha_T \cdot \frac{\lambda_{\min}(w(x_i) x_i x_i^T)}{16}, \qquad \text{and} \qquad \tau = \frac{C(\alpha_T + \kappa_2) b^2 \opnorm{B^{-1}}_2^2 T^2}{r^2}.
\end{equation*}
\end{prop*}

\paragraph{Remark:}

Due to the presence of the weighting function $w(x_i)$, Proposition~\ref{PropMallows} imposes weaker distributional requirements on the $x_i$'s than Proposition~\ref{PropMRSC}, and the requirements imposed in Proposition~\ref{PropHillRyan} are still weaker. In fact, a version of Proposition~\ref{PropMallows} could be derived with $w(x_i) = \min\left\{1, \; \frac{b^2}{\|Bx_i\|_2^2}\right\}$, which would not require the $x_i$'s to be sub-exponential. The tradeoff in comparing Proposition~\ref{PropHillRyan} to Propositions~\ref{PropMRSC} and~\ref{PropMallows} is that although the RSC condition holds under weaker distributional assumptions on the covariates, the absolute bound $b^2 \opnorm{B^{-1}}_2^2$ used in place of the sub-Gaussian/exponential parameter $\sigma_x^2$ may be much larger. Hence, the relative size of $\epsilon_T$ and the radius $r$ will need to be smaller in order for inequality~\eqref{EqnTrickyWeight} to be satisfied, relative to the requirement for inequality~\eqref{EqnTricky}. \\

In Section~\ref{SecSims} below, we explore the consequences of Propositions~\ref{PropGradBound}--\ref{PropHillRyan} for heavy-tailed, outlier, and sub-exponential distributions.

%%%%

\subsection{Oracle results and asymptotic normality}
\label{SecVariance}

As discussed in the preceding two subsections, penalized robust $M$-estimators produce local stationary points that enjoy $\ell_1$- and $\ell_2$-consistency whenever $\ell'$ is bounded and the errors and covariates satisfy suitable mild assumptions. In fact, a distinguishing aspect of different robust regression loss functions $\ell$ lies not in first-order comparisons, but in second-order considerations concerning the variance of the estimator. This is a well-known concept in classical robust regression analysis, where $p$ is fixed, $n \rightarrow \infty$, and the objective function does not contain a penalty term. By the Cram\'{e}r-Rao bound and under fairly general regularity conditions~\cite{LehCas98}, the optimal choice of $\ell$ that minimizes the asymptotic variance in the low-dimensional setting is the MLE function, $\ell(u) = - \log p_\epsilon(u)$, where $p_\epsilon$ is the probability density function of $\epsilon_i$. When the class of regression functions is constrained to those with bounded influence functions (or bounded values of $\ell'$), however, a more complex analysis reveals that choices of $\ell$ corresponding, e.g., to the losses introduced in Section~\ref{SecGMEst} produce better performance~\cite{Hub81}.

In this section, we establish oracle properties of penalized robust $M$-estimators. Our main result shows that under many of the assumptions stated earlier, local stationary points of the regularized $M$-estimators actually agree with the local oracle result, defined by
\begin{equation}
\label{EqnBoracle}
\boracle_S \defn \arg\min_{\beta \in \real^S: \; \|\beta - \betastar\|_2 \le r} \left\{\Loss_n(\beta)\right\}.
\end{equation}
This is particularly attractive from a theoretical standpoint, because the oracle result implies that local stationary points inherit all the properties of the lower-dimensional oracle estimator $\boracle_S$, which is covered by previous theory.

Note that $\boracle_S$ is truly an oracle estimator, since it requires knowledge of both the actual support set $S$ of $\betastar$ \emph{and} of $\betastar$ itself; the optimization of the loss function is taken only over a small neighborhood around $\betastar$. In cases where $\Loss_n$ is convex or global optima of $\Loss_n$ that are supported on $S$ lie in the ball of radius $r$ centered around $\betastar$, the constraint $\|\beta - \betastar\|_2 \le r$ may be omitted. If $\Loss_n$ satisfies a local RSC condition~\eqref{EqnLocalRSC}, the oracle program~\eqref{EqnBoracle} is guaranteed to be convex, as stated in the following simple lemma, proved in Appendix~\ref{AppLemConvex}:
\begin{lem*}
\label{LemConvex}
Suppose $\Loss_n$ satisfies the local RSC condition~\eqref{EqnLocalRSC} and $n \ge \frac{2\tau}{\alpha} k \log p$. Then $\Loss_n$ is strongly convex over the region $S_r \defn \left\{\beta \in \real^p: \supp(\beta) \subseteq S, \; \|\beta - \betastar\|_2 \le r\right\}$.
\end{lem*}

\noindent In particular, the oracle estimator $\boracle_S$ is guaranteed to be unique. \\

Our central result of this section shows that when the regularizer is $(\mu, \gamma)$-amenable and the loss function satisfies the local RSC condition in Assumption~\ref{AsLocalRSC}, stationary points of the $M$-estimator~\eqref{EqnGeneral} within the local neighborhood of $\betastar$ are in fact unique and equal to the oracle estimator~\eqref{EqnBoracle}. We also require a beta-min condition on the minimum signal strength, which we denote by $\betamin \defn \min_{j \in S} |\betastar_j|$. For simplicity, we state the theorem as a probabilistic result for sub-Gaussian covariates and the unweighted $M$-estimator~\eqref{EqnLoss}; similar results could be derived for generalized $M$-estimators under weaker distributional assumptions, as well.
\begin{thm*}
\label{ThmOracle}
Suppose the loss function $\Loss_n$ is given by the $M$-estimator~\eqref{EqnLoss} and is twice differentiable in the $\ell_2$-ball of radius $r$ around $\betastar$. Suppose the regularizer $\rho_\lambda$ is $(\mu, \gamma)$-amenable. Under the same conditions of Theorem~\ref{ThmStationary}, suppose in addition that $\|\betastar\|_1 \le \frac{R}{2}$ and $\frac{160 \lambda k}{2\alpha - \mu} < R$, and $\betamin \ge C \sqrt{\frac{\log k}{n}} + \gamma \lambda$. Suppose the sample size satisfies $n \ge c_0 \max\{k^2, k \log p\}$. With probability at least $1 - c \exp(-c' \min\{k, \log p\})$, any stationary point $\betatil$ of the program~\eqref{EqnGeneral} such that $\|\betatil - \betastar\|_2 \le r$ satisfies $\supp(\betatil) \subseteq S$ and $\betatil_S = \boracle_S$.
\end{thm*}
The proof of Theorem~\ref{ThmOracle} builds upon the machinery developed in the recent paper~\cite{LohWai14}. However, the argument here is slightly simpler, because we only need to prove the oracle result for stationary points within a radius $r$ of $\betastar$. For completeness, we include a proof of Theorem~\ref{ThmOracle} in Section~\ref{SecThmOracle}, highlighting the modifications that are necessary to obtain the statement in the present paper.

\paragraph{Remark:}

Several other papers~\cite{FanPen04, BraEtal11, LiEtal11} have established oracle results of a similar flavor, but only in cases where the $M$-estimator takes the form described in Section~\ref{SecMEst} and the loss function is convex. Furthermore, the results of previous authors only concern global optima and/or guarantee the existence of local optima with the desired oracle properties. Hence, our conclusions are at once more general and more complex, since we need a more careful treatment of possible local optima. \\

In fact, since the oracle program~\eqref{EqnBoracle} is essentially a $k$-dimensional optimization problem, Theorem~\ref{ThmOracle} allows us to apply previous results in the literature concerning the asymptotic behavior of low-dimensional $M$-estimators to simultaneously analyze the asymptotic distribution of $\boracle_S$ and $\betatil$. Huber~\cite{Hub73} studied asymptotic properties of $M$-estimators when the loss function is convex, and established asymptotic normality assuming $\frac{p^3}{n} \rightarrow 0$, a result which was improved upon by Yohai and Maronna~\cite{YohMar79}. Portnoy~\cite{Por85} and Mammen~\cite{Mam89} extended these results to nonconvex $M$-estimators.
\begin{comment}
but focused on $M$-estimators for which the loss function is four times differentiable. Portnoy assumes the scaling $\frac{(p \log n)^{3/2}}{n} \rightarrow 0$ and appropriate tail conditions on the design matrix, which are satisfied, e.g., when the $x_i$'s are sub-Gaussian, and Mammen tightens the required scaling to $\frac{p^{3/2} \log n}{n} \rightarrow 0$.
\end{comment}
Fewer results exist concerning generalized $M$-estimators: Bai and Wu~\cite{BaiWu97} and He and Shao~\cite{HeSha96} established asymptotic normality for a fairly general class of estimators, but the assumption is that $p$ is fixed and $n \rightarrow \infty$. He and Shao~\cite{HeSha00} extended their results to the case where $p$ is also allowed to grow and proved asymptotic normality when $\frac{p^2 \log p}{n} \rightarrow 0$, assuming a convex loss.

Although the overall $M$-estimator may be highly nonconvex, the \emph{restricted} program~\eqref{EqnBoracle} defining the oracle estimator is nonetheless convex (cf.\ Lemma~\ref{LemConvex} above). Hence, the standard convex theory for $M$-estimators with a diverging number of parameters applies without modification. Since the regularity conditions existing in the literature that guarantee asymptotic normality vary substantially depending on the form of the loss function, we only provide a sample corollary for a specific (unweighted) case, as an illustration of the types of results on asymptotic normality that may be derived from Theorem~\ref{ThmOracle}.
\begin{cor*}
\label{CorDist}
Suppose the loss function $\Loss_n$ is given by the $M$-estimator~\eqref{EqnLoss} and the regularizer $\rho_\lambda$ is $(\mu, \gamma)$-amenable. Under the same conditions of Theorem~\ref{ThmOracle}, suppose in addition that $\ell \in C^3$, $\E[\ell''(\epsilon_i)] \in (0, \infty)$, and $k \ge C \log n$. Let $\betatil$ be any stationary point of the program~\eqref{EqnGeneral} such that $\|\betatil - \betastar\|_2 \le r$. If $\frac{k \log^3 k}{n} \rightarrow 0$, then $\|\betatil - \betastar\|_2 = \order_P\left(\sqrt{\frac{k}{n}}\right)$. If $\frac{k^2 \log k}{n} \rightarrow 0$, then for any $v \in \real^p$, we have
\begin{equation*}
\frac{\sqrt{n}}{\sigma_v} \cdot v^T (\betatil - \betastar) \stackrel{d}{\longrightarrow} N(0, 1),
\end{equation*}
where
\begin{equation*}
\sigma^2_v \defn \frac{1}{\E\left[\ell''(\epsilon_i)\right] \cdot \E\left[\left(\ell'(\epsilon_i)\right)^2\right]} \cdot v^T \left(\frac{X^TX}{n}\right) v.
\end{equation*}
\end{cor*}

\noindent The proof of Corollary~\ref{CorDist} is provided in Appendix~\ref{AppCorDist}. Analogous results may be derived for other loss functions considered in this paper under slightly different regularity assumptions, by modifying appropriate low-dimensional results with diverging dimensionality (e.g., \cite{Por85, Mam89}).

%\textbf{Is there also an analogous result for $\ell_1$-penalized robust regression estimators, assuming incoherence of design?}

%%%%

\section{Optimization}
\label{SecOpt}

We now discuss how our statistical theory gives rise to a useful two-step algorithm for optimizing the resulting high-dimensional $M$-estimators. We first present some theory for the composite gradient descent algorithm, including rates of convergence for the regularized problem. We then describe our new two-step algorithm, which is guaranteed to converge to a stationary point within the local region where the RSC condition holds, even when the $M$-estimator is nonconvex.

%%%%%

\subsection{Composite gradient descent}
\label{SecCompGrad}

In order to obtain stationary points of the program~\eqref{EqnGeneral}, we use the composite gradient descent algorithm~\cite{Nes07}. Denoting $\Lossbar_n(\beta) \defn \Loss_n(\beta) - q_\lambda(\beta)$, we may rewrite the program as
\begin{equation*}
\betahat \in \arg\min_{\|\beta\|_1 \le R} \left\{\Lossbar_n(\beta) + \lambda \|\beta\|_1\right\}.
\end{equation*}
Then the composite gradient iterates are given by
\begin{equation}
\label{EqnIterates}
\beta^{t+1} \in \arg\min_{\|\beta\|_1 \le R} \left\{\frac{1}{2} \left\|\beta - \left(\beta^t - \frac{\nabla \Lossbar_n(\beta^t)}{\eta}\right)\right\|_2^2 + \frac{\lambda}{\eta} \|\beta\|_1\right\},
\end{equation}
where $\eta$ is the stepsize parameter. Defining the soft-thresholding operator $S_{\lambda/\eta}(\beta)$ componentwise according to
\begin{equation*}
S_{\lambda/\eta}^j \defn \sign(\beta_j) \left(|\beta_j| - \frac{\lambda}{\eta}\right)_+,
\end{equation*}
a simple calculation shows that the iterates~\eqref{EqnIterates} take the form
\begin{equation}
\label{EqnSoftThresh}
\beta^{t+1} = S_{\lambda/\eta} \left(\beta^t - \frac{\nabla \Lossbar_n(\beta^t)}{\eta}\right).
\end{equation}

The following theorem guarantees that the composite gradient descent algorithm will converge at a linear rate to point near $\betastar$ as long as the initial point $\beta^0$ is chosen close enough to $\betastar$. We will require the following assumptions on $\Loss_n$, where
\begin{equation*}
\T'(\beta_1, \beta_2) \defn \Loss_n(\beta_1) - \Loss_n(\beta_2) - \inprod{\nabla \Loss_n(\beta_2)}{\beta_1 - \beta_2}
\end{equation*}
denotes the Taylor remainder.
\begin{assumption}
\label{AsOpt}
Suppose $\Loss_n$ satisfies the restricted strong convexity condition
\begin{equation}
\label{EqnRSC2}
\T'(\beta_1, \beta_2) \ge \alpha' \|\beta_1 - \beta_2\|_2^2 - \tau' \frac{\log p}{n} \|\beta_1 - \beta_2\|_1^2,
\end{equation}
for all $\beta_1, \beta_2 \in \real^p$ such that $\|\beta_1 - \betastar\|_2, \|\beta_2 - \betastar\|_2 \le r$. In addition, suppose $\Loss_n$ satisfies the restricted smoothness condition
\begin{equation}
\label{EqnRSM}
\T'(\beta_1, \beta_2) \le \alpha'' \|\beta_1 - \beta_2\|_2^2 + \tau'' \frac{\log p}{n} \|\beta_1 - \beta_2\|_1^2, \qquad \forall \beta_1, \beta_2 \in \real^p.
\end{equation}
\end{assumption}
Note that the definition of $\T'$ differs slightly from the definition of the related Taylor difference used in Assumption~\ref{AsLocalRSC}. However, one may verify the RSC condition~\eqref{EqnRSC2} in exactly the same way as we verify the RSC condition~\eqref{EqnLocalRSC} via the mean value theorem argument of Section~\ref{SecRSC}, so we do not repeat the proofs here. The restricted smoothness condition~\eqref{EqnRSM} is fairly mild and is easily seen to hold with $\tau'' = 0$ when the loss function $\ell$ appearing in the definition of the $M$-estimator has a bounded second derivative. We will also assume for simplicity that $q_\lambda$ is convex, as is the case for the SCAD and MCP regularizers; the theorem may be extended to situations where $q_\lambda$ is nonconvex, given an appropriate quadratic bound on the Taylor remainder of $q_\lambda$.

We have the following theorem, proved in Appendix~\ref{SecThmOptimization}. It guarantees that as long as the initial point $\beta^0$ of the composite gradient descent algorithm is chosen close enough to $\betastar$, the log of the $\ell_2$-error between iterates $\beta^t$ and a global minimizer $\betahat$ of the regularized $M$-estimator~\eqref{EqnGeneral} will decrease linearly with $t$ up to the order of the statistical error $\|\betahat - \betastar\|_2$.
\begin{thm*}
\label{ThmOptimization}
Suppose $\Loss_n$ satisfies the RSC condition~\eqref{EqnRSC2} and the RSM condition~\eqref{EqnRSM}, and suppose $\rho_\lambda$ is $\mu$-amenable with $\mu < 2 \alpha$ and $q_\lambda$ is convex. Suppose the regularization parameters satisfy the scaling
\begin{equation*}
C \max\left\{\|\nabla \Loss_n(\betastar)\|_\infty \; \tau \sqrt{\frac{\log p}{n}}\right\} \le \lambda \le \frac{C' \alpha}{R}.
\end{equation*}
Also suppose $\betahat$ is a global optimum of the objective~\eqref{EqnGeneral} over the set $\|\betahat - \betastar\|_2 \le \frac{r}{2}$. Suppose $\eta \ge 2 \alpha''$ and
\begin{equation}
\label{EqnScaling}
n \ge \frac{4(2\tau' + \tau'')}{\alpha' - \mu/2 + \eta/2} \cdot \frac{\alpha' - \mu/2}{\alpha' - \mu/2 + \eta/2} \cdot \frac{r^2}{4} \cdot R^2 \log p.
\end{equation}
If $\beta^0$ is chosen such that $\|\beta^0 - \betastar\|_2 \le \frac{r}{2}$, successive iterates of the composite gradient descent algorithm satisfy the bound
\begin{equation*}
\|\beta^t - \betahat\|_2^2 \le \frac{c}{2\alpha - \mu} \left(\delta^2 + \frac{\delta^4}{\tau} + c\tau \frac{k \log p}{n} \|\betahat - \betastar\|_2^2\right), \qquad \forall t \ge T^*(\delta),
\end{equation*}
where $\delta^2 \ge \frac{c'\|\betahat - \betastar\|_2^2}{1-\kappa}$ is a tolerance parameter, $\kappa \in (0,1)$, and $T^*(\delta) = \frac{c''\log(1/\delta^2)}{\log(1/\kappa)}$.
\end{thm*}

\paragraph{Remark:}

It is not obvious a priori that even if $\beta^0$ is chosen within a small constant radius of $\betastar$, successive iterates will also remain close by. Indeed, the hard work to establish this fact is contained in the proof of Lemma~\ref{LemBasin} in Appendix~\ref{SecThmOptimization}. Furthermore, note that we cannot expect a global convergence guarantee to hold in general, since the only assumption on $\Loss_n$ is the local version of RSC. Hence, a local convergence result such as the one stated in Theorem~\ref{ThmOptimization} is the best we can hope for in this scenario. \\

In the simulations of Section~\ref{SecSims}, we see cases where initializing the composite gradient descent algorithm outside the local basin of attraction where the RSC condition holds causes iterates to converge to a stationary point outside the local region, and the resulting stationary point is \emph{not} consistent for $\betastar$. Hence, the assumption imposed in Theorem~\ref{ThmOptimization} concerning the proximity of $\beta^0$ to $\betastar$ is necessary in order to ensure good behavior of the optimization trajectory for nonconvex robust estimators.

%%%%

\subsection{Two-step estimators}
\label{SecTwoStep}

As discussed in Section~\ref{SecMain} above, whereas different choices of the regression function $\ell$ with bounded derivative yield estimators that are asymptotically unbiased and satisfy the same $\ell_2$-bounds up to constant factors, certain $M$-estimators may be more desirable from the point of view of asymptotic efficiency. When $\ell$ is nonconvex, we can no longer guarantee fast global convergence of the composite gradient descent algorithm---indeed, the algorithm may even converge to statistically inconsistent local optima. Nonetheless, Theorem~\ref{ThmOptimization} guarantees that the composite gradient descent algorithm will converge quickly to a desirable stationary point if the initial point is chosen within a constant radius of the true regression vector. We now propose a new two-step algorithm, based on Theorem~\ref{ThmOptimization}, that may be applied to optimize high-dimensional robust $M$-estimators. Even when the regression function is nonconvex, our algorithm will always converge to a stationary point that is statistically consistent for $\betastar$.

\paragraph{Two-step procedure:}

\begin{itemize}
\item[(1)] Run composite gradient descent using a convex regression function $\ell$ with convex $\ell_1$-penalty, such that $\ell'$ is bounded.
\item[(2)] Use the output of step (1) to initialize composite gradient descent on the desired high-dimensional $M$-estimator.
\end{itemize}

According to our results on statistical consistency (cf.\ Theorem~\ref{ThmStationary}), step (1) will produce a global optimum $\betahat^1$ such that $\|\betahat^1 - \betastar\|_2 \le c\sqrt{\frac{k \log p}{n}}$, as long as the regression function $\ell$ is chosen appropriately.\footnote{The rate of convergence may be sublinear in the initial iterations~\cite{Nes07}, but we are still guaranteed to have convergence.} Under the scaling $n \ge Cr^2 \cdot k \log p$, we then have $\|\betahat^1 - \betastar\|_2 \le r$. Hence, by Theorem~\ref{ThmOptimization}, composite gradient descent initialized at $\betahat^1$ in step (2) will converge to a stationary point of the $M$-estimator at a linear rate. By our results of Section~\ref{SecMain}, the final output $\betahat^2$ in step (2) is then statistically consistent and agrees with the local oracle estimator if we use a $(\mu, \gamma)$-amenable penalty.

\paragraph{Remark}

Our proposed two-step algorithm bears some similarity to classical algorithms used for locating optima of robust regression estimators in low-dimensional settings. Recall the notion of a one-step $M$-estimator~\cite{Bic75}, which is obtained by taking a single step of the Newton-Raphson algorithm starting from a properly chosen initial point. Yohai~\cite{Yoh87} and Simpson et al.~\cite{SimEtal92} study asymptotic properties of one-step $GM$- and $MM$-estimators in the setting where $p$ is fixed, and show that the resulting regression estimators may simultaneously enjoy high-breakdown and high-efficiency properties. Welsh and Ronchetti~\cite{WelRon02} present a finer-grained analysis of the asymptotic distribution and influence function of one-step $M$-estimators as a function of the initialization point. Most directly related is the suggestion of Hampel et al.~\cite{HamEtal11} for optimizing redescending $M$-estimators using a one-step procedure initialized using a least median of squares estimator, in order to overcome the problem of nonconvexity and possibly multiple local optima; however, the method is mostly justified heuristically. Although each step of our two-step method involves running a composite gradient descent algorithm fully until convergence, the overall goal is still to produce an estimator at the end of the second step that is more efficient and has better theoretical properties than the solution of the first step alone. \\

The simulations in the next section demonstrate the efficacy of our two-step algorithm and the importance of step (1) in obtaining a proper initialization to the composite gradient procedure in step (2).

%%%%%%%%%%%

\section{Simulations}
\label{SecSims}

In this section, we expound upon some concrete instances of our theoretical results and provide simulation results. Throughout, we generate i.i.d.\ data from the linear model
\begin{equation*}
y_i = x_i^T \betastar + \epsilon_i, \qquad \forall 1 \le i \le n.
\end{equation*}

%%%%%

\subsection{Statistical consistency}

In the first set of simulations, we verify the $\ell_2$-consistency of high-dimensional robust regression estimators when data are generated from various distributions.

We begin our discussion with a lemma that demonstrates the failure of the Lasso to achieve the minimax $\order\left(\sqrt{\frac{k \log p}{n}}\right)$ rate when the $\epsilon_i$'s are drawn from an $\alpha$-stable distribution with $\alpha < 2$. Recall that a variable $X_0$ has an $\alpha$-stable distribution with scale parameter $\gamma$ if the characteristic function of $X_0$ is given by
\begin{equation}
\label{EqnAlphaChar}
\E[\exp(itX_0)] = \exp\left(-\gamma^\alpha |t|^\alpha\right), \qquad \forall t \in \real,
\end{equation}
and $\alpha \in (0, 2]$~\cite{Nol15}. In particular, the standard normal distribution is an $\alpha$-stable distribution with $(\alpha, \gamma) = \left(2, \frac{1}{\sqrt{2}}\right)$, and the standard Cauchy distribution (also known as a $t$-distribution with one degree of freedom) is an $\alpha$-stable distribution with $(\alpha, \gamma) = (1,1)$. The lemma is proved in Appendix~\ref{AppLemStableLasso}.

\begin{lem*}
\label{LemStableLasso}
Suppose $X$ is a sub-Gaussian matrix and $\epsilon$ is an i.i.d.\ vector of $\alpha$-stable random variables with scale parameter 1. Suppose $\lambda \asymp \sqrt{\frac{\log p}{n}}$. If $\alpha < 2$ and $\log p = o\left(n^{\frac{2 - \alpha}{\alpha}}\right)$, we have
\begin{equation*}
\mprob\left(\left\|\frac{X^T\epsilon}{n}\right\|_\infty \ge \lambda\right) \ge c_\alpha > 0,
\end{equation*}
where $c_\alpha \le 1$ is a constant that depends only on the sub-Gaussian parameter of the rows of $X$ and does \emph{not} scale with the problem dimensions. In particular, if $\epsilon$ is an i.i.d.\ vector of Cauchy random variables, the Lasso estimator is \emph{inconsistent}.
\end{lem*}

\noindent In contrast, as established in Theorem~\ref{ThmStationary} and the propositions of Section~\ref{SecRSC}, replacing the ordinary least squares loss by an appropriate robust loss function yields estimators that are consistent at the usual $\order\left(\sqrt{\frac{k \log p}{n}}\right)$ rate. \\

In our first set of simulations, we generated $\epsilon_i$'s from a Cauchy distribution with scale parameter $0.1$, and the $x_i$'s from a standard normal distribution. We ran simulations for three problem sizes: $p = 128, 256$, and $512$, with sparsity level $k \approx \sqrt{p}$. In each case, we set $\betastar = \left(\frac{1}{\sqrt{k}}, \dots, \frac{1}{\sqrt{k}}, 0, \dots, 0\right)$. Figure~\ref{FigConsistency}(a) shows the results when the loss function $\Loss_n$ is equal to the Huber, Tukey, and Cauchy robust losses, and the regularizer is the $\ell_1$-penalty. The estimator $\betahat$ was obtained using the composite gradient descent algorithm described in Section~\ref{SecCompGrad} in the case of the Huber loss, and the two-step algorithm described in Section~\ref{SecTwoStep} in the cases of the Tukey and Cauchy losses, with the output of the Huber estimator used to initialize the second step of the algorithm. In each case, we set the regularization parameters $\lambda = 0.3 \sqrt{\frac{\log p}{n}}$ and $R = 1.1 \; \|\betastar\|_1$, and averaged the results over 50 randomly generated data sets. As shown in the figure, the $\ell_1$-penalized robust regression functions all yield statistically consistent estimators. Furthermore, the curves for different problem sizes align when the $\ell_2$-error is plotted against the rescaled sample size $\frac{n}{k \log p}$, agreeing with the theoretical bound in Theorem~\ref{ThmStationary}.

\begin{figure}
\begin{center}
\begin{tabular}{cc}
\includegraphics[width=0.5\textwidth]{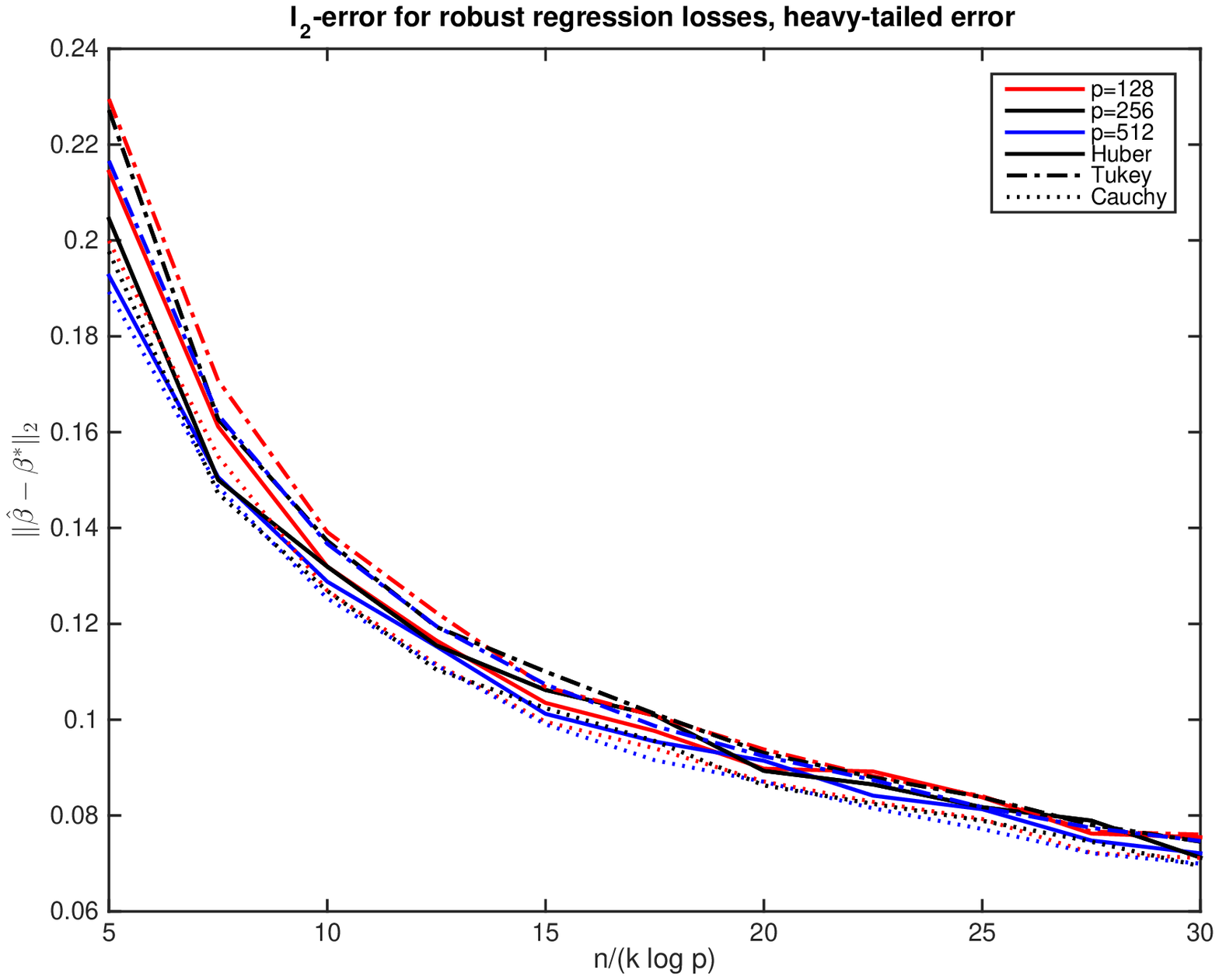} & \includegraphics[width=0.5\textwidth]{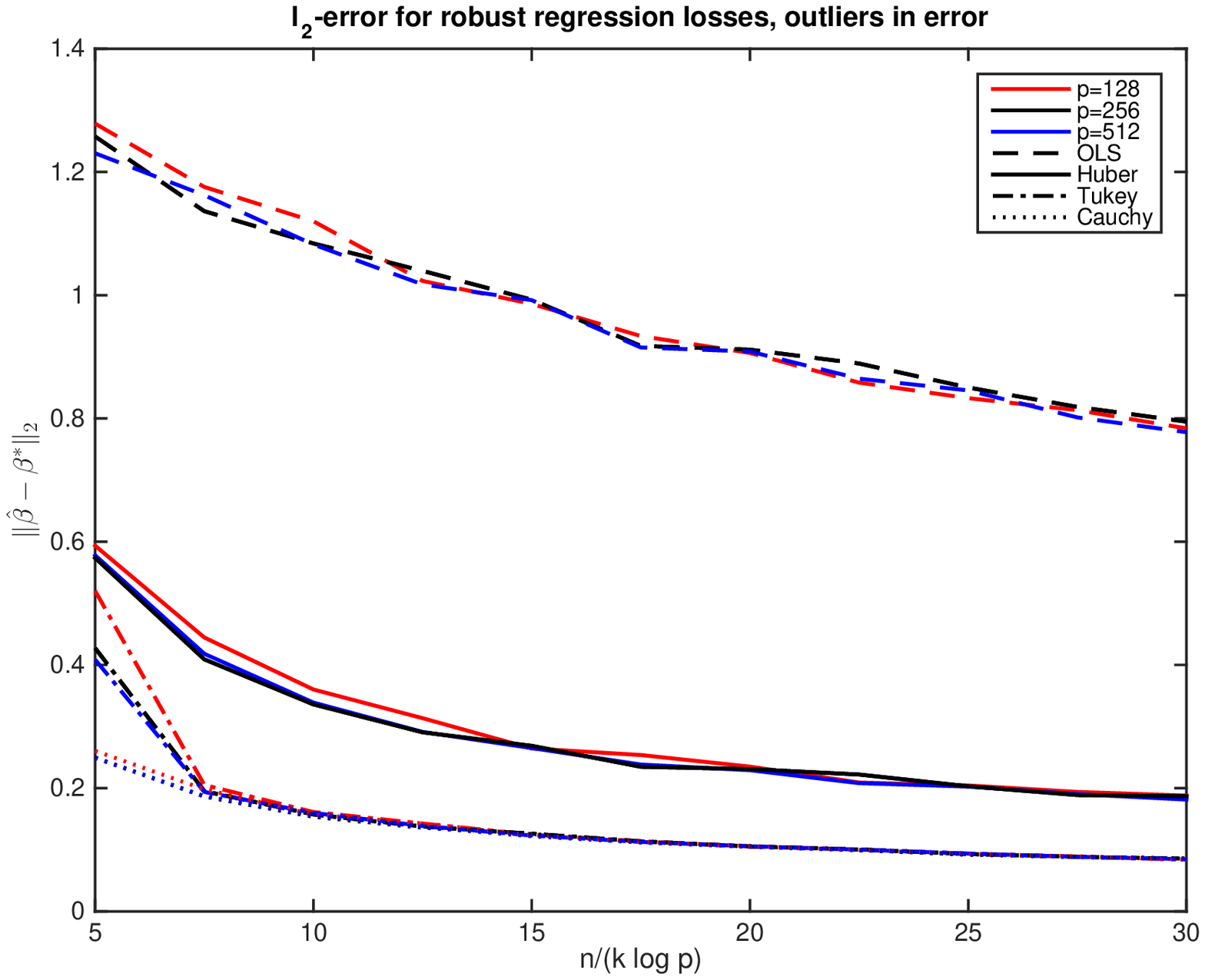} \\
(a) & (b)
\end{tabular}
\caption{Plots showing statistical consistency of $\ell_1$-penalized robust regression functions, when the $x_i$'s are normally distributed but the $\epsilon_i$'s follow a heavy-tailed or normal mixture distribution with a constant fraction of outliers. Each point represents an average over 50 trials. The $\ell_2$-error is plotted against the rescaled sample size $\frac{n}{k \log p}$. Curves correspond to the Huber (solid), Tukey (dash-dotted), Cauchy (dotted), and ordinary least squares (dashed) losses, and are color-coded according to the problem sizes $p = 128$ (red), 256 (black), and 512 (blue). (a) Plots for a heavy-tailed Cauchy error distribution. The Huber, Tukey, and Cauchy robust losses all yield statistically consistent results, as predicted by Theorem~\ref{ThmStationary} and Propositions~\ref{PropGradBound} and~\ref{PropMRSC}. (b) Plots for a mixture of normals error distribution with $30\%$ large-variance outliers. Since the error distribution is sub-Gaussian, the ordinary least squares loss also yields a statistically consistent estimator at minimax rates, up to a constant prefactor; however, the robust regression losses provide a significant improvement in the prefactor.}
\label{FigConsistency}
\end{center}
\end{figure}

We also ran a similar set of simulations when the $\epsilon_i$'s were generated from a mixture of normals, representing a contaminated distribution with a constant fraction of outliers. With probability $0.7$, the value of $\epsilon_i$ was distributed according to $N(0, (0.1)^2)$, and was otherwise drawn from a $N(0, 10^2)$ distribution. Figure~\ref{FigConsistency}(b) shows the results of the simulations. Again, we see that the robust regression functions all give rise to statistically consistent estimators with $\ell_2$-error scaling as $\order\left(\sqrt{\frac{k \log p}{n}}\right)$. We also include the plots for the standard Lasso estimator with the ordinary least squares objective. Since the distribution of $\epsilon_i$ is sub-Gaussian for the mixture distribution, the Lasso estimator is also $\ell_2$-consistent; however, we see that the robust loss functions improve upon the $\ell_2$-error of the Lasso by a constant factor.

Finally, we ran simulations to test the statistical consistency of generalized $M$-estimators under relaxed distributional assumptions on the covariates. We generated $x_i$'s from a sub-exponential distribution, given by independent chi-square variables with 10 degrees of freedom, and recentered to have mean zero. The $\epsilon$'s were drawn from a Cauchy distribution with scale parameter $0.1$. We ran trials for problem sizes $p = 128, 256$, and $512$, with $k \approx \sqrt{p}$ and $\betastar = \left(\frac{1}{\sqrt{k}}, \dots, \frac{1}{\sqrt{k}}, 0, \dots, 0\right)$. We used the $\ell_1$-penalized Mallows estimator described in Proposition~\ref{PropMallows}, with $b = 3$, $B = I_p$, and $\ell$ equal to the Huber loss function, and optimized the function using the composite gradient descent algorithm with random initializations, with the regularization parameters $\lambda = 0.3 \sqrt{\frac{\log p}{n}}$ and $R = 1.1 \; \|\betastar\|_1$. Figure~\ref{FigSubExp} shows the result of the simulations, from which we observe that the Mallows estimator is indeed statistically consistent, as predicted by Theorem~\ref{ThmStationary} and Proposition~\ref{PropMallows}. We also plotted the results for $\ell_1$-penalized Huber regression. It is not difficult to see from the proof of Theorem~\ref{ThmStationary} that $\|\nabla \Loss_n(\betastar)\|_\infty$ is also of the order $\order\left(\sqrt{\frac{k \log p}{n}}\right)$ when the $x_i$'s are sub-exponential, but with a larger prefactor than the Mallows loss. We observe in Figure~\ref{FigSubExp} that the Huber loss indeed appears to yield a statistically consistent estimator as well, but at a relatively slower rate. In our simulations, we needed a slightly larger value $\lambda = \sqrt{\frac{\log p}{n}}$ for the Huber loss in order to achieve statistical consistency.

\begin{figure}
\begin{center}
\includegraphics[width=0.5\textwidth]{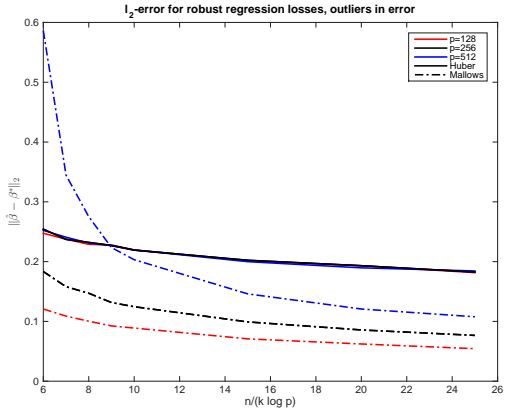}
\caption{Plot showing simulation results for the $\ell_1$-penalized Mallows generalized $M$-estimator with a Huber loss function, when covariates are drawn from a sub-exponential distribution and errors are drawn from a heavy-tailed Cauchy distribution. Results for the $\ell_1$-penalized Huber loss are shown for comparison. Each point represents an average over 50 trials. Although both estimators appear to be statistically consistent, the Mallows estimator exhibits better performance. The plot agrees with the behavior predicted by Theorem~\ref{ThmStationary} and Proposition~\ref{PropMallows}.}
\label{FigSubExp}
\end{center}
\end{figure}

%%%%%

\subsection{Convergence of optimization algorithm}

Next, we ran simulations to verify the convergence behavior of the composite gradient descent algorithm described in Section~\ref{SecOpt}. We set $p = 128$, $k \approx \sqrt{p}$, and $n \approx 20 k \log p$, and generated $\epsilon_i$'s from a Cauchy distribution with scale parameter $0.1$, and the $x_i$'s from a standard normal distribution. We set $\betastar = \left(\frac{1}{\sqrt{k}}, \dots, \frac{1}{\sqrt{k}}, 0, \dots, 0\right)$. We then simulated the solution paths for the Huber and Cauchy loss functions with an $\ell_1$-penalty, with regularization parameters $\lambda = 0.3 \sqrt{\frac{\log p}{n}}$ and $R = 1.1 \; \|\betastar\|_1$. Panel (a) of Figure~\ref{FigSolnPaths} shows solution paths for the composite gradient descent algorithm with the Huber loss from 10 different starting points, chosen randomly from a $N(0, 6^2 I_p)$ distribution. An estimate of the global optimum $\betahat$ was obtained from preliminary runs of the optimization error, and the log optimization error $\log(\|\beta^t  - \betahat\|_2)$ for each of the initializations was computed accordingly. In addition, we plot the statistical error $\log(\|\betahat - \betastar\|_2)$ in red for comparison. As seen in the plot, the log errors decay roughly linearly in $t$. Since the $\ell_1$-penalized Huber objective is convex, our theory guarantees sublinear convergence of the iterates initially and then linear convergence locally around $\betastar$ within the radius $\frac{r}{2}$, as specified by Theorem~\ref{ThmOptimization}. Indeed, our plots suggest nearly linear convergence even outside the local RSC region. All iterates converge to the unique global optimum $\betatil$ (the apparent bifurcation is due to the small nonzero error tolerance provided in our implementation of the algorithm as a criterion for convergence.)

\begin{figure}
\begin{center}
\begin{tabular}{cc}
\includegraphics[width=0.5\textwidth]{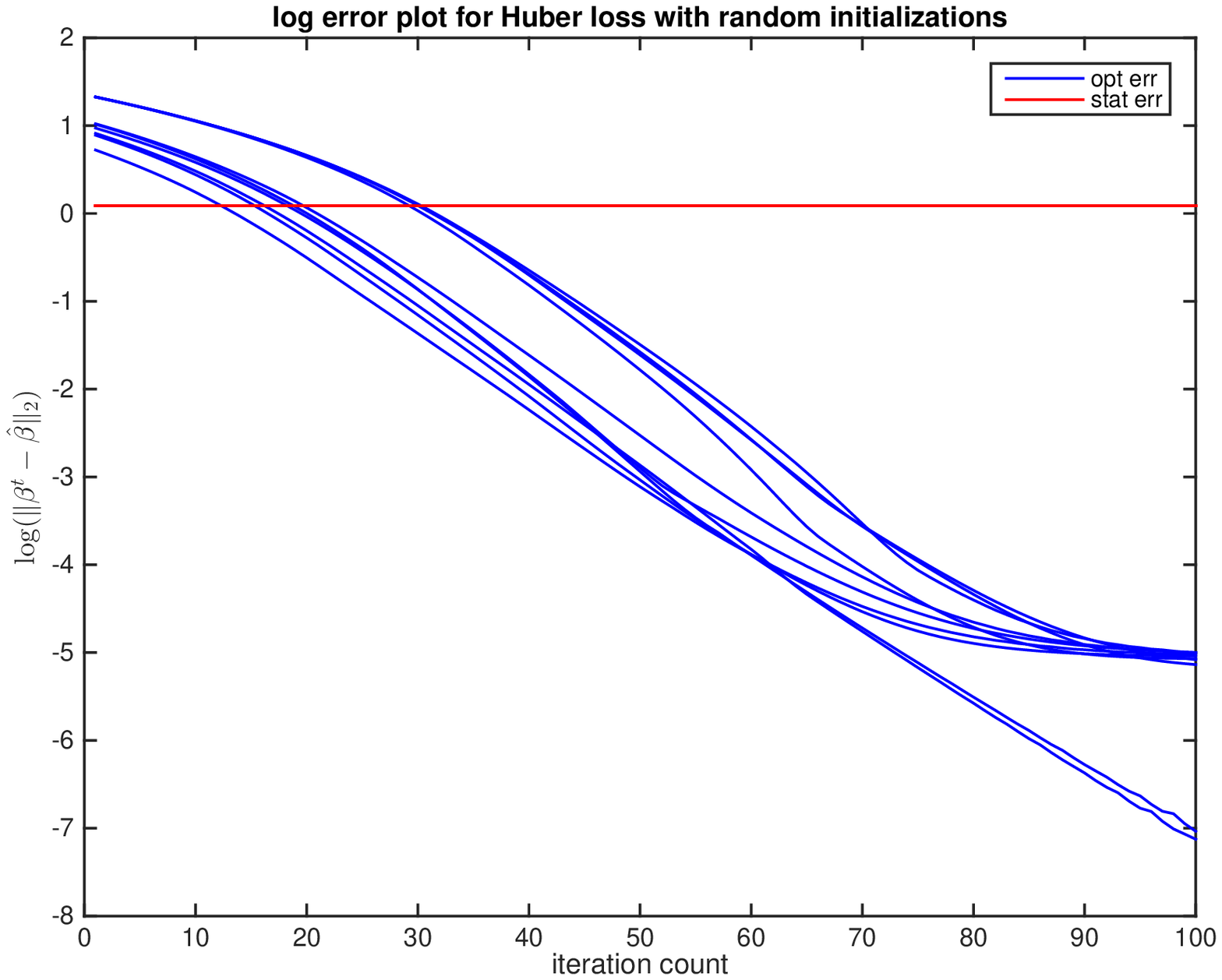} & \includegraphics[width=0.5\textwidth]{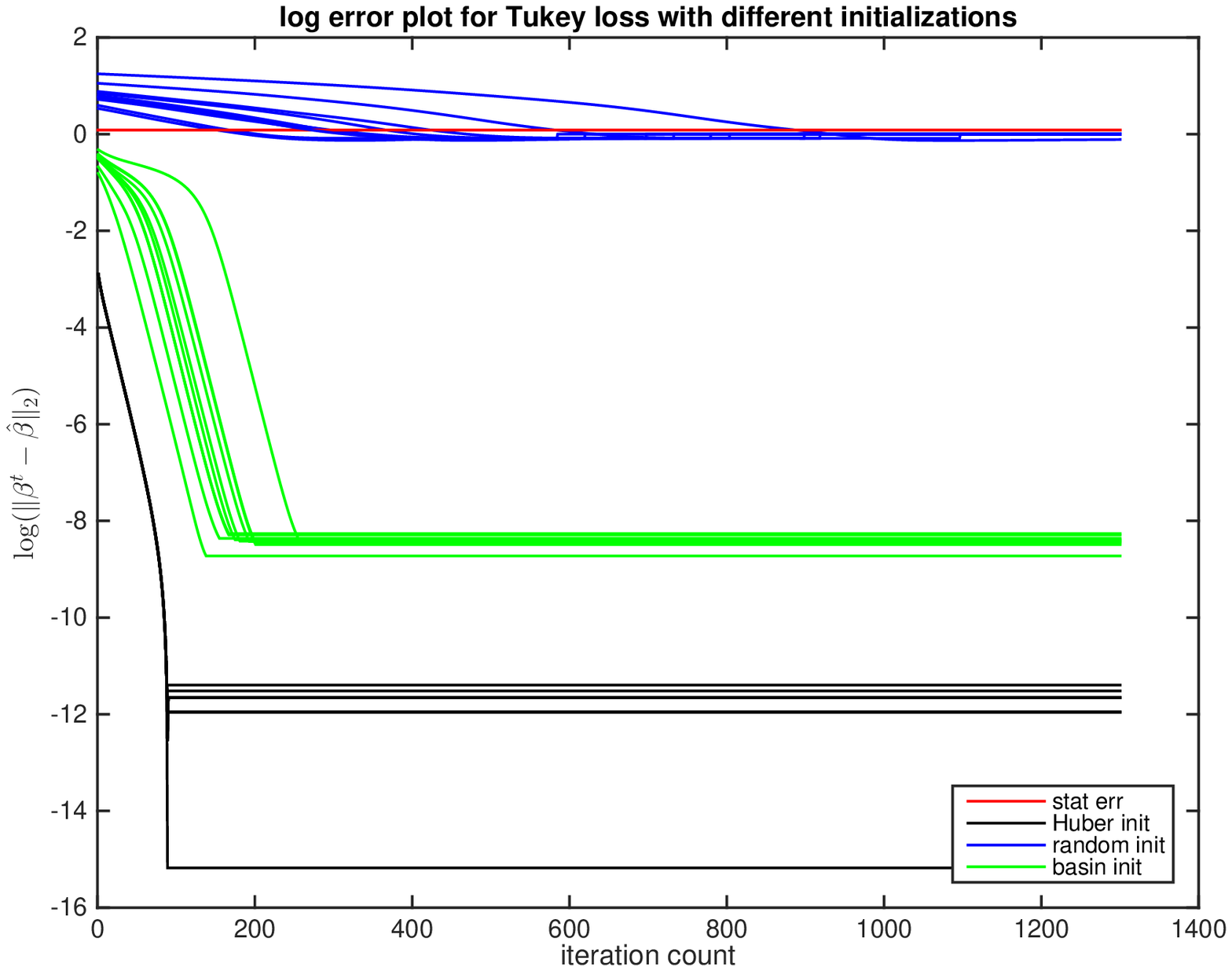} \\
(a) & (b)
\end{tabular}
\caption{Plots showing optimization trajectories for the composite gradient descent algorithm applied to various high-dimensional robust regression functions. The log of the $\ell_2$-error is plotted against the iteration number for a fixed instantiation of the data using the Huber and Tukey loss. The errors are generated from a heavy-tailed Cauchy distribution. Solution paths are shown in blue and measured with respect to $\betastar$; the statistical error is shown for reference and plotted in red. (a) Solution paths for the $\ell_1$-penalized convex Huber loss with 10 random initializations. All iterates converge to a unique optimum $\betatil$. Theorem~\ref{ThmOptimization} guarantees a rate of convergence that is linear on a log scale, once the iterates enter the region where the function satisfies restricted strong convexity. (b) Solution paths for the $\ell_1$-penalized nonconvex Tukey loss with 10 random initializations from the $\ell_1$-penalized Huber output (black); slight perturbations of $\betastar$ within the local basin where the loss function satisfies restricted strong convexity (green); and random initializations (blue). The black and green trajectories all converge at a linear rate to a unique stationary point in the local region, as predicted by Theorem~\ref{ThmOptimization}. The blue iterates converge at a slower rate to an entirely different stationary point. This figure emphasizes the need for proper initialization of the composite gradient algorithm in order to locate a statistically consistent stationary point.}
\label{FigSolnPaths}
\end{center}
\end{figure}

Figure~\ref{FigSolnPaths}(b) shows solution paths using the $\ell_1$-penalized Tukey loss. We plot the composite gradient descent iterates for 10 different starting points chosen by the output of composite gradient descent applied to the $\ell_1$-penalized Huber loss (black) with random initializations; 10 randomly chosen starting points given by $\betastar$ plus a $N(0, (0.1)^2 I_p)$ perturbation (green); and 10 randomly chosen starting points drawn from a $N(0, 3^2 I_p)$ distribution (blue). The simulation results reveal a linear rate of convergence for composite gradient descent iterates in the first two cases, as predicted by Theorem~\ref{ThmOptimization}, since the initial iterates lie within the local region around $\betastar$ where the Tukey loss satisfies the RSC condition. All of the black and green trajectories converge to the same unique stationary point in the local region. In the third case, however, the rate of convergence of composite gradient descent iterates is slower, and the iterates actually converge to a different stationary point further away from $\betastar$. This emphasizes the cautionary message that stationary points may indeed exist for nonconvex robust regression functions that are \emph{not} consistent for the true regression vector, and first-order optimization algorithms may converge to these undesirable stationary points if initialized improperly.

%%%%%

\subsection{Nonconvex regularization}

Finally, we ran simulations to verify the oracle results described in Section~\ref{SecVariance}. Figure~\ref{FigSCAD} shows side-by-side comparisons for robust regression using the Huber and Cauchy loss functions with the SCAD penalty, with parameter $a = 2.5$. We ran simulations for $p = 128, 256$, and $512$, with $k \approx \sqrt{p}$ and $\betastar = \left(\frac{1}{\sqrt{k}}, \dots, \frac{1}{\sqrt{k}}, 0, \dots, 0\right)$. The $\epsilon_i$'s were drawn from a Cauchy distribution with scale parameter $0.1$, and the $x_i$'s were drawn from a standard normal distribution. The $\ell_1$-penalized Huber loss was used to select an initial point for the composite gradient descent algorithm, as prescribed by the two-step algorithm; in all cases, we set the regularization parameters to be $\lambda = \sqrt{\frac{\log p}{n}}$ and $R = 1.1 \; \|\betastar\|_1$. Panel (a) plots the $\ell_2$-error versus the rescaled sample size $\frac{n}{k \log p}$, from which we see that both SCAD-penalized objective functions yield statistically consistent estimators. Panel (b) plots the fraction of trials (out of 50) for which the recovered support of the estimator agrees with the true support of $\betastar$. As we see, the families of curves for different loss functions stack up when the horizontal axis is rescaled according to $\frac{n}{k \log p}$. Furthermore, the probability of correct support recovery transitions sharply from 0 to 1 in panel (b), as predicted by Theorem~\ref{ThmOracle}. Note that the transition point for the Cauchy loss in panel (b), which happens for $\frac{n}{k \log p} \approx 8$, also corresponds to a sharp drop in the $\ell_2$-error in panel (a), since $\betatil$ is then equal to the low-dimensional oracle estimator. Panel (c) plots the empirical variance of $\sqrt{n} \cdot e_1^T (\betatil - \betastar)$, the first component of the error vector rescaled by $\sqrt{n}$. We see that the variance for the Cauchy loss is uniformly smaller than the variance for the Huber loss---indeed, the Cauchy loss corresponds to the MLE of the error distribution. Furthermore, the curves for each loss function roughly align for different problem sizes, and the variance is roughly constant for increasing $n$, as predicted by Corollary~\ref{CorDist}. Note that Corollary~\ref{CorDist} requires third-order differentiability, so it does not directly address the Huber loss. However, the empirical variance of the Huber estimators is also roughly constant, suggesting that a version of Corollary~\ref{CorDist} applicable to the Huber loss might be derived from the oracle results of Theorem~\ref{ThmOracle}.

\begin{figure}
\begin{center}
\begin{tabular}{cc}
\includegraphics[width=0.5\textwidth]{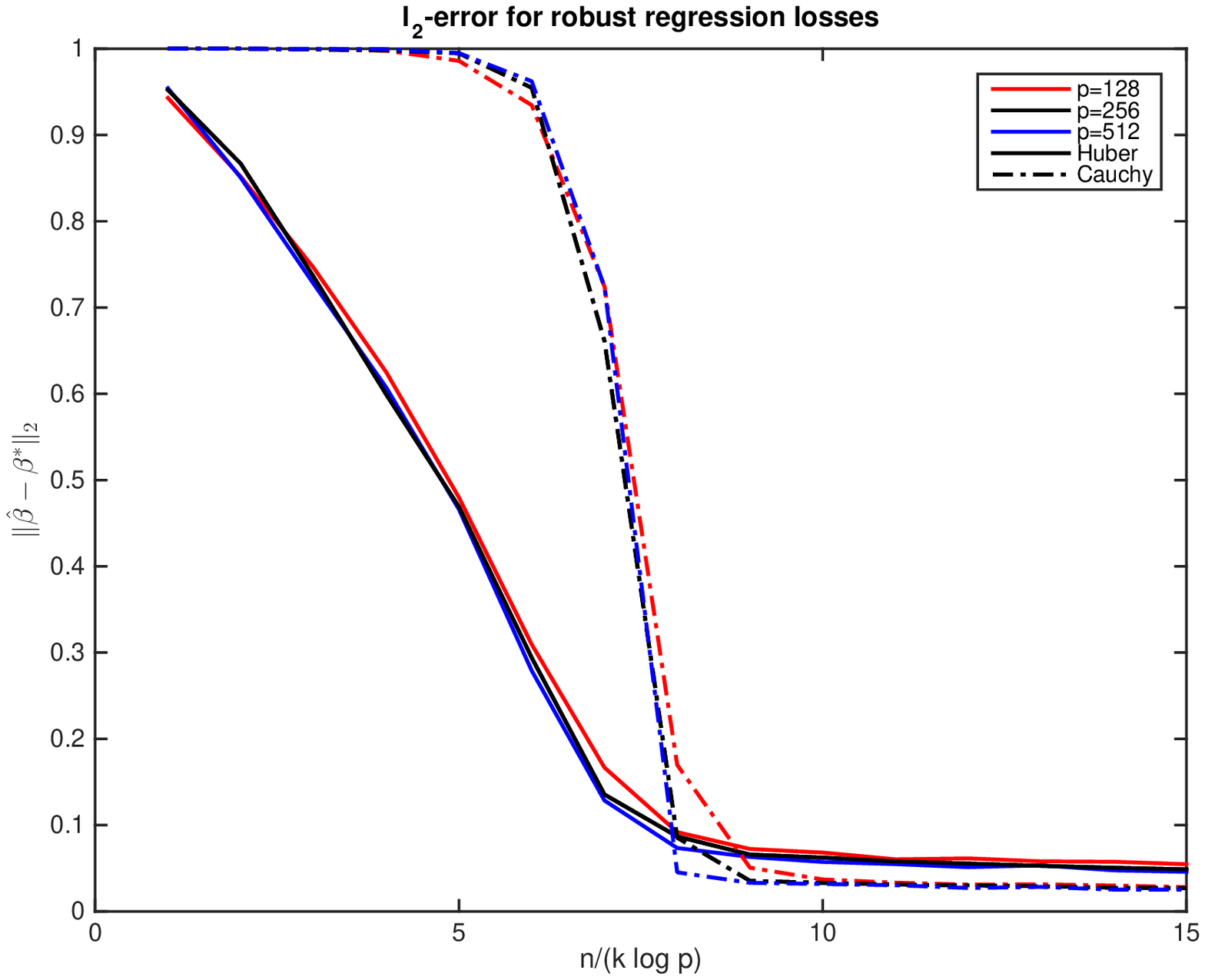} & \includegraphics[width=0.5\textwidth]{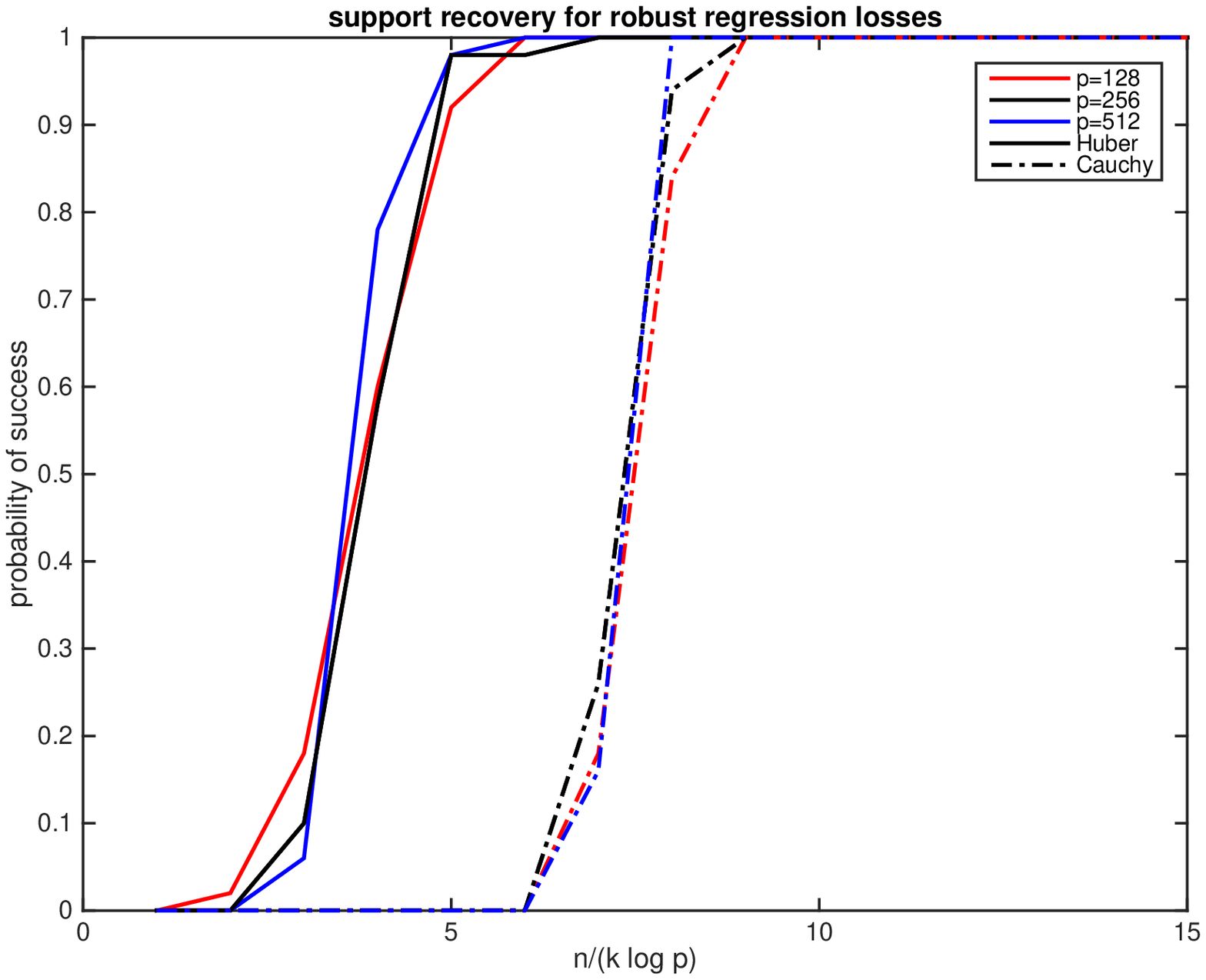} \\
(a) & (b)
\end{tabular}
\begin{tabular}{c}
\includegraphics[width=0.5\textwidth]{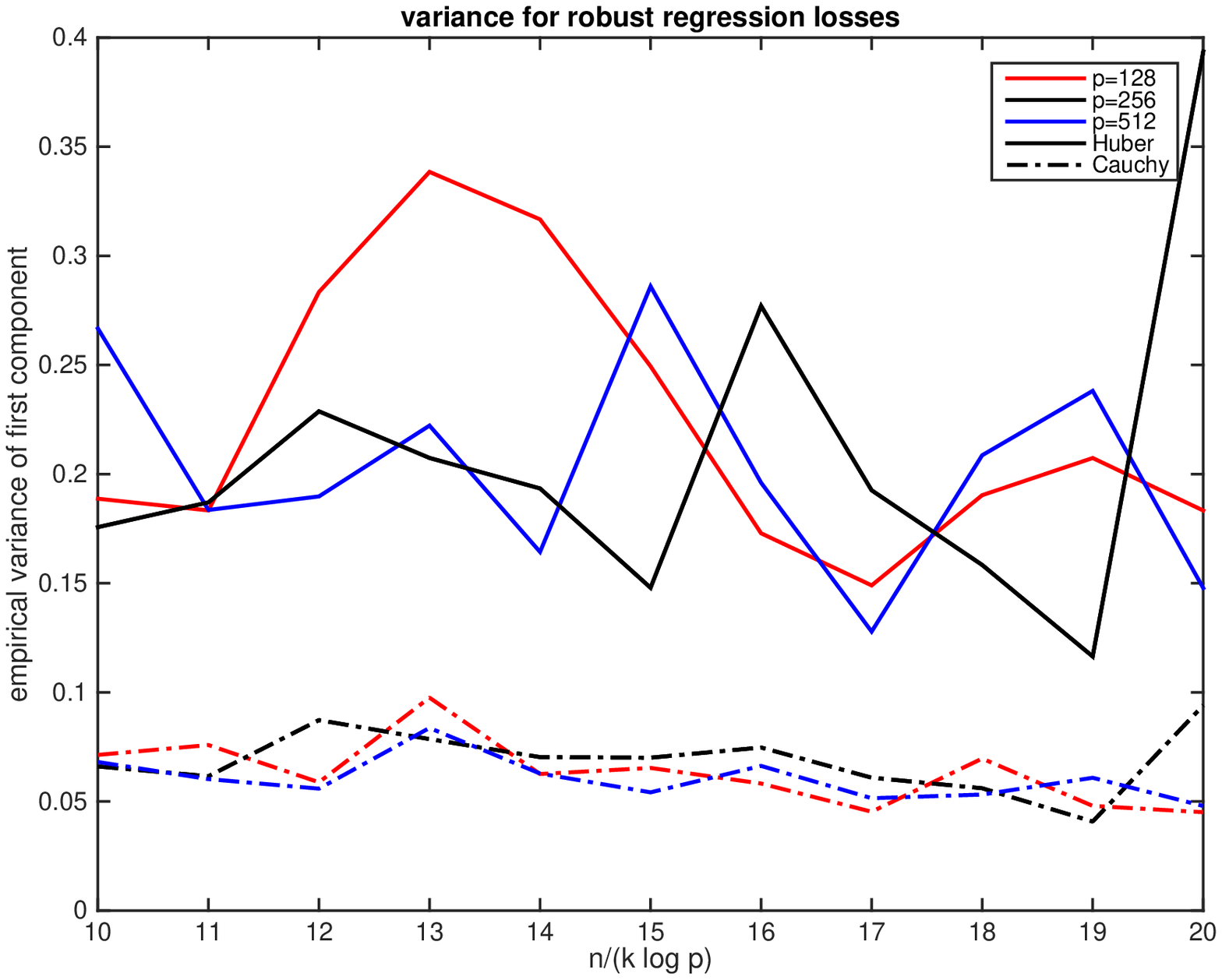} \\
(c)
\end{tabular}
\caption{Plots showing simulation results for robust regression with a nonconvex SCAD regularizer, using a Huber loss (solid lines) and Cauchy loss (dashed lines), for three problem sizes: $p = 128$ (red), $p = 256$ (black), and $p = 512$ (blue). Each point represents an average over 50 trials. (a) Plot showing $\ell_2$-error as a function of the rescaled sample size $\frac{n}{k \log p}$. Both regularizers yield statistically consistent estimators, as predicted by Theorem~\ref{ThmStationary}. (b) Plot showing variable selection consistency. The probability of success in recovering the support transitions sharply from 0 to 1 as a function of the sample size, agreeing with the theoretical predictions of Theorem~\ref{ThmOracle}. The transition threshold corresponds with the sharp drop in $\ell_2$-error seen in panel (a), since $\betatil$ agrees with the oracle result. (c) Plot showing the empirical variance of $\sqrt{n} \cdot e_1^T (\betatil - \betastar)$, the rescaled first component in the error vector. As predicted by the asymptotic normality result of Corollary~\ref{CorDist}, the empirical variance remains roughly constant for sufficiently large sample sizes.}
\label{FigSCAD}
\end{center}
\end{figure}

%%%%%

\section{Discussion}

We have studied penalized high-dimensional robust estimators for linear regression. Our results show that under a local RSC condition satisfied by many robust regression $M$-estimators, stationary points within the region of restricted curvature are actually statistically consistent estimators of the true regression vector, and even under heavy-tailed errors or outlier contamination, these estimators enjoy the same convergence rate as $\ell_1$-penalized least squares regression with sub-Gaussian errors. Furthermore, we show that when the penalty is chosen from an appropriate family of nonconvex, amenable regularizers, the stationary point within the local RSC region is unique and agrees with the local oracle solution. This allows us to establish asymptotic normality of local stationary points under appropriate regularity conditions, and in some cases conclude that the regularized $M$-estimator is asymptotically efficient. Finally, we propose a two-step $M$-estimation procedure for obtaining local stationary points when the $M$-estimator is nonconvex, where the first step consists of optimizing a convex problem to obtain a sufficiently close initialization for a final run of composite gradient descent in the second step.

Several open questions remain that provide interesting avenues for future work. First, although the side constraint $\|\beta\|_1 \le R$ in the regularized $M$-estimation program~\eqref{EqnGeneral} is required in our proofs to ensure that stationary points obey a cone condition, it is unclear whether this side condition is necessary. Indeed, since we are only concerned with stationary points within a small radius $r$ of $\betastar$, the additional $\ell_1$-constraint may be redundant. It would be useful to remove the appearance of $R$ for practical problems, since we would then only need to tune the parameter $\lambda$. Second, as a consequence of the oracle result in Theorem~\ref{ThmOracle}, local stationary points inherit other properties of the oracle solution $\boracle_S$ in addition to asymptotic normality, such as breakdown behavior and properties of the influence function. It would be interesting to explore these properties for robust $M$-estimators with a diverging number of parameters. A potentially harder problem would be to derive bounds on the measures of robustness for stationary points of regularized robust estimators when the oracle result does not hold (i.e., for $\ell_1$-penalized robust $M$-estimators). Lastly, whereas our results on asymptotic normality allow us to draw conclusions regarding the asymptotic variance of the local oracle solution, it would be valuable to derive nonasymptotic bounds on the variance of high-dimensional robust $M$-estimators. By trading off the nonasymptotic bias and variance, one could then determine the form of a robust regression function that is optimal in some sense.

%\section*{Acknowledgments}

%%%%%%%%%%

\appendix

\section{Measures of robustness}
\label{AppRobust}

Various methods exist in the classical literature for quantifying the robustness of statistical estimation procedures. In this section, we provide a review of breakdown points, influence functions, and asymptotic variance of robust estimators, and cite relevant literature.

The \emph{finite-sample breakdown point} of an estimator $T_n$ on the sample $\{x_i\}_{i=1}^n$ is defined by
\begin{equation*}
FBP_n(T; x_1, \dots, x_n) \defn \frac{1}{n} \cdot \min\left\{m: \max_{i_1, \dots, i_m} \sup_{y_1, \dots, y_m} |T_n(z_1, \dots, z_n)| = \infty\right\},
\end{equation*}
where $(z_1, \dots, z_n)$ is the sample obtained from $(x_1, \dots, x_n)$ by replacing the data points $(x_{i_1}, \dots, x_{i_m})$ by $(y_1, \dots, y_m)$~\cite{HamEtal11, DonHub83}. One may verify that the finite-sample breakdown point is $\frac{1}{n}$ for $M$-estimators of the type defined in Section~\ref{SecMEst} when $\ell$ is convex~\cite{AlfEtal13}. This provides another reason to use nonconvex loss functions in order to obtain a robust estimator. Although the breakdown behavior of $M$-estimators is much harder to characterize when the loss function is nonconvex, Maronna et al.~\cite{MarEtal79} derived theoretical results showing the breakdown point decays as $\order(p^{-1/2})$ when the $x_i$'s are Gaussian. More recently, Wang et al.~\cite{WangEtal13} analyzed the breakdown point of a certain nonconvex penalized $M$-estimator, but their analysis is again very specific to the precise form of the estimator and requires careful data-dependent tuning of the scale parameter used in the objective function. Under suitable regularity conditions, taking the limit of the finite breakdown point as $n \rightarrow \infty$ yields the \emph{asymptotic breakdown point}, but the latter concept is more technical and we do not discuss it here.

A second measure of robustness is given by the \emph{influence function}. At the population level, the influence function of an estimator $T$ on a distribution $F$ with respect to a point $(x,y)$ is defined by
\begin{equation*}
	IF((x,y); T, F) = \lim_{t \rightarrow 0^+} \frac{T((1-t)F + t \delta_{(x,y)}) - T(F)}{t},
\end{equation*}
where $\delta_{(x,y)}$ is a point mass at $(x,y)$. The \emph{gross error sensitivity} is defined in terms of the influence function as
\begin{equation*}
GES(T,F) \defn \sup_{(x,y)} \big| IF((x,y); T,F) \big|,
\end{equation*}
and the estimator $T$ is \emph{$B$-robust} if $GES(T,F) < \infty$~\cite{Rou81}. In the linear regression case, let $F_\beta$ denote the distribution on $(x_i, y_i, \epsilon_i)$ parametrized by $\beta$. If $T_\ell$ minimizes the $M$-estimator defined in equation~\eqref{EqnLoss} and $\ell$ is twice differentiable, the influence function takes the form
\begin{equation}
\label{EqnPuran}
IF\left((x,y); T_\ell, F_\beta\right) = \ell'(x^T \beta - y) \cdot \left(\E\left[\ell''(x_i^T \beta - y_i) \cdot x_i x_i^T\right]\right)^{-1} x,
\end{equation}
where the expectation is taken with respect to $F_\beta$~\cite[Section 6.3]{HamEtal11}. In particular, if the $x_i$'s are fixed and contamination is only allowed in the $y_i$'s, the influence function in equation~\eqref{EqnPuran} is bounded as a function of $y$, provided $\ell'$ is bounded. For a generalized $M$-estimator $T_\eta$ defined by equation~\eqref{EqnGeneralizedM}, the influence function is given by
\begin{equation}
\label{EqnDosa}
IF\left((x,y); T_\eta, F_\beta\right) = \eta(x, x^T \beta - y) \cdot \left(\E\left[\left(\frac{\partial \eta(x,r)}{\partial r}\right) \Bigg |_{(x_i, y_i)} \cdot x_i x_i^T \right]\right)^{-1} x,
\end{equation}
where the expectation is taken with respect to $F_\beta$~\cite{HamEtal11}. In particular, if $\eta$ takes the form in equation~\eqref{EqnDecompGM}, then equation~\eqref{EqnDosa} simplifies to
\begin{equation}
\label{EqnChapati}
IF((x,y); T_\eta, F_\beta) = w(x) \; \ell'\left((x^T \beta - y)v(x)\right) \cdot \left(\E\left[w(x_i)v(x_i) \cdot \ell''\left(r_i v(x_i)\right) \cdot x_i x_i^T\right]\right)^{-1} x,
\end{equation}
and we see that the overall influence function is bounded whenever $\ell'$ is bounded and $w$ is defined in such a way that $\|w(x)x\|_2$ is bounded.

A finite-sample version of the influence function is known as the \emph{sensitivity curve}, and under suitable regularity conditions, the sensitivity curve converges to the influence function as $n \rightarrow \infty$~\cite{HamEtal11}. The literature concerning influence functions for high-dimensional estimators is again rather sparse, but has been a topic of recent interest~\cite{MedMar14, OllEtal14}.

Finally, we turn to second-order considerations. In the classical low-dimensional setting when $p$ is fixed and $n \rightarrow \infty$, Maronna and Yohai~\cite{MarYoh81} show that under appropriate regularity conditions, the asymptotic variance of an $M$-estimator is given by
\begin{equation*}
V(T,F) = \int IF((x,y); T,F) \cdot IF((x,y); T,F) \; dF(x,y).
\end{equation*}
By the celebrated Cram\'{e}r-Rao bound~\cite{LehCas98}, when the $x_i$'s are fixed and the $\epsilon_i$'s are i.i.d., the asymptotic variance $V(T,F)$ of any unbiased estimator is bounded below by the inverse of the Fisher information of the underlying distribution. Furthermore, this lower bound is achieved when $T$ is the MLE, in which case $T$ is also asymptotically normally distributed~\cite{God60, LehCas98}. As pointed out in the previous paragraph, however, the influence function of the MLE may not be bounded, leading to a critical tradeoff in designing robust $M$-estimators. In addition, the behavior of the asymptotic variance is much harder to analyze when both $n$ and $p$ are allowed to grow. Several recent papers~\cite{ElKEtal13, BeaEtal13, DonMon13} examine the setting where $\frac{n}{p} \rightarrow \delta \in (1, \delta)$, and show that the asymptotic variance of the (unregularized) $M$-estimator coming from a convex loss function includes an additional term not present in the classical fixed-$p$ case. In contrast, we show that with the proper choice of nonconvex penalty, local solutions of nonconvex regularized $M$-estimators coincide with the oracle solution, so they inherit certain optimality properties from classical robust estimation theory. It is these higher-order considerations that reveal the true advantage of using nonconvex loss functions for robust $M$-estimation; although estimators such as the LAD Lasso~\cite{WanEtal07, Wan13} may also be shown to be statistically consistent under reasonable assumptions, the LAD loss is a suboptimal choice from the viewpoint of asymptotic efficiency, under the high-dimensional scaling $n \ge Ck\log p$ and oracle conditions, unless the additive errors follow a double exponential distribution

%%%%%

\section{Proofs of main theorems}
\label{AppProofs}

In this Appendix, we provide the proofs of the main theorems stated in the text of the paper.

\subsection{Proof of Theorem~\ref{ThmStationary}}
\label{SecThmStationary}

We first suppose the existence of stationary points in the local region; we will establish that fact at the end of the proof. Suppose $\betatil$ is a stationary point such that $\|\betatil - \betastar\|_2 \le r$. Since $\betatil$ is a stationary point and $\betastar$ is feasible, we have the inequality
\begin{equation}
\label{EqnKanga}
\inprod{\nabla \Loss_n(\betatil) - \nabla q_\lambda(\betatil) + \lambda \sign(\betatil)}{\betastar - \betatil} \ge 0.
\end{equation}
By the convexity of $\frac{\mu}{2} \|\beta\|_2^2 - q_\lambda(\beta)$, we have
\begin{equation}
\label{EqnOtter}
\inprod{\nabla q_\lambda(\betatil)}{\betastar - \betatil} \ge q_\lambda(\betastar) - q_\lambda(\betatil) - \frac{\mu}{2} \|\betatil - \betastar\|_2^2,
\end{equation}
so together with inequality~\eqref{EqnKanga}, we have
\begin{equation*}
\inprod{\nabla \Loss_n(\betatil) + \lambda \sign(\betatil)}{\betastar - \betatil} \ge q_\lambda(\betastar) - q_\lambda(\betatil) - \frac{\mu}{2} \|\betatil - \betastar\|_2^2.
\end{equation*}
Since $\inprod{\sign(\betatil)}{\betastar - \betatil} \le \|\betastar\|_1 - \|\betatil\|_1$, this means
\begin{equation}
\label{EqnAllspice}
\inprod{\nabla \Loss_n(\betatil)}{\betastar - \betatil} \ge \rho_\lambda(\betatil) - \rho_\lambda(\betastar) - \frac{\mu}{2} \|\betatil - \betastar\|_2^2.
\end{equation}

Now denote $\nutil \defn \betatil - \betastar$. From the RSC inequality~\eqref{EqnLocalRSC}, we have
\begin{equation}
\label{EqnSpicy}
\inprod{\nabla \Loss_n(\betatil) - \nabla \Loss_n(\betastar)}{\betatil - \betastar} \ge \alpha \|\nutil\|_2^2 - \tau \frac{\log p}{n} \|\nutil\|_1^2.
\end{equation}
Combining inequality~\eqref{EqnSpicy} with inequality~\eqref{EqnAllspice}, we then have
\begin{equation}
\label{EqnClove}
\left(\alpha - \frac{\mu}{2}\right) \|\nutil\|_2^2 - \tau \frac{\log p}{n} \|\nutil\|_1^2 + \left(\rho_\lambda(\betatil) - \rho_\lambda(\betastar)\right) \le \inprod{\nabla \Loss_n(\betastar)}{\betastar - \betatil},
\end{equation}
so by H\"{o}lder's inequality, we conclude that
\begin{equation}
\label{EqnRock}
\left(\alpha - \frac{\mu}{2}\right) \|\nutil\|_2^2 - \tau \frac{\log p}{n} \|\nutil\|_1^2 + \left(\rho_\lambda(\betatil) - \rho_\lambda(\betastar)\right) \le \|\nabla \Loss_n(\betastar)\|_\infty \|\nutil\|_1.
\end{equation}
In particular, under the assumed scaling $\lambda \ge 4 \|\nabla \Loss_n(\betastar)\|_\infty$ and $\lambda \ge 8\tau R \frac{\log p}{n}$, we have
\begin{align*}
\left(\alpha - \frac{\mu}{2}\right) \|\nutil\|_2^2 & \le \left(\rho_\lambda(\betastar) - \rho_\lambda(\betatil)\right) + \left(2R\tau \frac{\log p}{n} + \|\nabla \Loss_n(\betastar)\|_\infty\right) \|\nutil\|_1 \\
& \le \left(\rho_\lambda(\betastar) - \rho_\lambda(\betatil)\right) + \frac{\lambda}{2} \|\nutil\|_1 \\
& \le \left(\rho_\lambda(\betastar) - \rho_\lambda(\betatil)\right) + \frac{1}{2} \left(\rho_\lambda(\nutil) + \frac{\mu}{2} \|\nutil\|_2^2\right),
\end{align*}
implying that
\begin{equation}
\label{EqnRednose}
0 \le \left(\alpha - \frac{3\mu}{4}\right) \|\nutil\|_2^2 \le \frac{3}{2} \rho_\lambda(\betastar) - \frac{1}{2} \rho_\lambda(\betatil).
\end{equation}
By Lemma 5 in Loh and Wainwright~\cite{LohWai13}, we then have
\begin{equation}
\label{EqnRoll}
0 \le 3\rho_\lambda(\betastar) - \rho_\lambda(\betatil) \le \lambda \big(3\|\nutil_A\|_1 - \|\nutil_{A^c}\|_1\big),
\end{equation}
where $A$ is the index set of the $k$ largest elements of $\nutil$ in magnitude. Hence,
\begin{equation*}
\|\nutil_{A^c}\|_1 \le 3 \|\nutil_A\|_1,
\end{equation*}
implying that
\begin{equation}
\label{EqnCone}
\|\nutil\|_1 = \|\nutil_A\|_1 + \|\nutil_{A^c}\|_1 \le 4 \|\nutil_A\|_1 \le 4 \sqrt{k} \|\nutil\|_2.
\end{equation}
Combining inequalities~\eqref{EqnRednose} and~\eqref{EqnRoll} then gives
\begin{equation*}
\left(\alpha - \frac{3\mu}{4}\right) \|\nutil\|_2^2 \le \frac{3\lambda}{2} \|\nutil_A\|_1 - \frac{\lambda}{2} \|\nutil_{A^c}\|_1 \le \frac{3\lambda}{2} \|\nutil_A\|_1 \le 6\lambda\sqrt{k} \|\nutil\|_2,
\end{equation*}
from which we conclude that
\begin{equation}
\label{EqnL2Bd}
\|\nutil\|_2 \le \frac{24 \lambda \sqrt{k}}{4\alpha - 3\mu},
\end{equation}
as wanted. Combining the $\ell_2$-bound with inequality~\eqref{EqnCone} then yields the $\ell_1$-bound.

Finally, in order to establish the existence of stationary points, we simply define $\betahat \in \real^p$ such that
\begin{equation}
\label{EqnElmo}
\betahat \in \arg\min_{\|\betahat - \betastar\|_2 \le r, \; \|\betahat\|_1 \le R} \left\{\Loss_n(\beta) + \rho_\lambda(\beta)\right\}.
\end{equation}
Then $\betahat$ is a stationary point of the program~\eqref{EqnElmo}, so by the argument just provided, we have
\begin{equation*}
\|\betahat - \betastar\|_2 \le c \sqrt{\frac{k \log p}{n}}.
\end{equation*}
Provided $n \ge Cr^2 \cdot k \log p$, the point $\betahat$ will lie in the interior of the sphere of radius $r$ around $\betastar$. Hence, $\betahat$ is also a stationary point of the original program~\eqref{EqnGeneral}, guaranteeing the existence of such local stationary points.

%%%%%%%%%%%

\subsection{Proof of Theorem~\ref{ThmOracle}}
\label{SecThmOracle}

This argument is an adaptation of the proofs of Theorems 1 and 2 in the recent paper~\cite{LohWai14}. We follow the primal-dual witness construction introduced there:
\begin{itemize}
\item[(i)] Optimize the restricted program
\begin{equation}
\label{EqnRestricted}
\betahat_S \in \arg\min_{\beta \in \real^S: \; \|\beta\|_1 \le R} \left\{\Loss_n(\beta) + \rho_\lambda(\beta)\right\},
\end{equation}
and establish that $\|\betahat_S\|_1 < R$.
\item[(ii)] Define $\zhat_S \in \partial \|\betahat_S\|_1$, and choose $\zhat_{S^c}$ to satisfy the zero-subgradient condition
\begin{equation}
\label{EqnZeroGrad}
\nabla \Loss_n(\betahat) - \nabla q_\lambda(\betahat) + \lambda \zhat = 0,
\end{equation}
where $\zhat = (\zhat_S, \zhat_{S^c})$ and $\betahat \defn (\betahat_S, 0_{S^c})$. Show that $\betahat_S = \boracle_S$ and establish strict dual feasibility: $\|\zhat_{S^c}\|_\infty < 1$.
\item[(iii)] Verify via second-order conditions that $\betahat$ is a local minimum of the program~\eqref{EqnGeneral} and conclude that all stationary points $\betatil$ satisfying $\|\betatil - \betastar\|_2 \le r$ are supported on $S$.
\end{itemize}

\paragraph{Step (i):}

By Theorem~\ref{ThmStationary} applied to the restricted program~\eqref{EqnRestricted}, we have
\begin{equation*}
\|\betahat_S - \betastar_S\|_1 \le \frac{80 \lambda k}{2\alpha - \mu},
\end{equation*}
so
\begin{equation*}
\|\betahat_S\|_1 \le \|\betastar\|_1 + \|\betahat_S - \betastar_S\|_1 \le \frac{R}{2} + \frac{80 \lambda k}{2 \alpha - \mu} < R,
\end{equation*}
using the assumptions of the theorem. This establishes step (i) of the PDW construction. \\

\paragraph{Step (ii):}

Since $\betahat_S$ is an interior point of the restricted program~\eqref{EqnRestricted}, it must satisfy a zero-subgradient condition on the restricted program, implying that we may define $\zhat_{S^c}$ to satisfy equation~\eqref{EqnZeroGrad}. We rewrite the zero-subgradient condition~\eqref{EqnZeroGrad} as
\begin{equation*}
\left(\nabla \Loss_n(\betahat) - \nabla \Loss_n(\betastar)\right) +
\left(\nabla \Loss_n(\betastar) - \nabla q_\lambda(\betahat)\right) +
\lambda \zhat = 0,
\end{equation*}
and by the fundamental theorem of calculus,
\begin{equation*}
\Qhat (\betahat - \betastar) + \left(\nabla
\Loss_n(\betastar) - \nabla q_\lambda(\betahat)\right) + \lambda \zhat
= 0,
\end{equation*}
where $\Qhat \defn \int_0^1 \nabla^2 \Loss_n\left(\betastar + t(\betahat - \betastar)\right) dt$. In block form, this means
\begin{equation}
\label{EqnZeroSubBlock}
\begin{bmatrix}
\Qhat_{SS} & \Qhat_{SS^c} \\
\Qhat_{S^cS} & \Qhat_{S^cS^c}
\end{bmatrix}  
\begin{bmatrix}		\betahat_S - \betastar_S \\ 0
\end{bmatrix}
+
\begin{bmatrix} \nabla \Loss_n(\betastar)_S - \nabla
                q_\lambda(\betahat_S) \\ \nabla
                \Loss_n(\betastar)_{S^c} - \nabla
                q_\lambda(\betahat_{S^c})
\end{bmatrix}
+ \lambda \begin{bmatrix} \zhat_S \\ \zhat_{S^c}
        \end{bmatrix}
= 0.
\end{equation}

We now have the following lemma, concerning the oracle estimator:
\begin{lem*}
\label{LemOracle}
Under the conditions of Theorem~\ref{ThmOracle}, we have the bound
\begin{equation*}
\|\boracle_S - \betastar_S\|_\infty \le c \sqrt{\frac{\log k}{n}},
\end{equation*}
and $\betahat_S = \boracle_S$.
\end{lem*}

\begin{proof}
By the optimality of the oracle estimator, we have
\begin{equation}
\label{EqnLima}
\Loss_n(\boracle) \le \Loss_n(\betastar).
\end{equation}
Furthermore, $\Loss_n$ is strongly convex over the restricted region $S_r$ by Lemma~\ref{LemConvex}. Hence,
\begin{equation}
\label{EqnBean}
\Loss_n(\betastar) + \inprod{\nabla \Loss_n(\betastar)}{\boracle - \betastar} + \frac{\alpha}{4} \|\boracle - \betastar\|_2^2 \le \Loss_n(\boracle).
\end{equation}
Summing inequalities~\eqref{EqnLima} and~\eqref{EqnBean}, we obtain
\begin{align*}
\frac{\alpha}{4} \|\boracle - \betastar\|_2^2 & \le \inprod{\nabla \Loss_n(\betastar)}{\betastar - \boracle} \le \|\nabla \Loss_n(\betastar)\|_\infty \cdot \|\boracle - \betastar\|_1 \\
& \le \sqrt{k} \|\nabla \Loss_n(\betastar)\|_\infty \cdot \|\boracle - \betastar\|_2,
\end{align*}
implying that
\begin{equation*}
\|\boracle - \betastar\|_2 \le \frac{4\sqrt{k}}{\alpha} \|\nabla \Loss_n(\betastar)\|_\infty.
\end{equation*}
In particular, when
\begin{equation*}
\frac{4 \sqrt{k}}{\alpha} \|\nabla \Loss_n(\betastar)\|_\infty < r,
\end{equation*}
the oracle estimator $\boracle$ is in the interior point of the feasible region, implying in particular that
\begin{equation*}
\left(\nabla \Loss_n(\boracle)\right)_S = 0.
\end{equation*}
Hence, we have
\begin{equation*}
\left(\nabla \Loss_n(\boracle) - \nabla \Loss_n(\betastar)\right)_S + \left(\nabla \Loss_n(\betastar)\right)_S = 0,
\end{equation*}
or
\begin{equation*}
\Qhat^{\mathcal{O}}_{SS} (\boracle - \betastar)_S + \left(\nabla \Loss_n(\betastar)\right)_S = 0,
\end{equation*}
where
\begin{equation*}
\Qhat^{\mathcal{O}} \defn \int_0^1 \nabla^2 \Loss_n\left(\betastar + t(\boracle - \betastar)\right) dt = \left(\int_0^1 \ell''\left(x_i^T \big(\betastar + t(\boracle - \betastar)\big) - y_i\right) dt\right) \cdot x_ix_i^T.
\end{equation*}
This implies that
\begin{equation}
\label{EqnLochness}
\|\boracle_S - \betastar_S\|_\infty = \left\| (\Qhat^{\mathcal{O}})_{SS}^{-1} \left(\nabla \Loss_n(\betastar)\right)_S \right\|_\infty.
\end{equation}

Let $\Ball_2^S(1) \defn \{u: \|u\|_2 \le 1, \text{ and } \supp(u) \subseteq S\}$. For $v, w \in \Ball_2^S(1)$, consider the quantity
\begin{align*}
\Bigg|v^T \Bigg\{\Qhat^{\mathcal{O}}_{SS} - & \left(\nabla^2 \Loss_n(\betastar)\right)_{SS} \Bigg\}w\Bigg| \\
& = \left| \frac{1}{n} \sum_{i=1}^n \left\{\int_0^1 \left(\ell''\left(x_i^T \big(\betastar + t(\boracle - \betastar)\big) - y_i\right)- \ell''(x_i^T \betastar - y_i)\right) dt \right\}(x_i^T v) (x_i^T w) \right| \\
& \le \kappa_3 \cdot \frac{1}{n} \sum_{i=1}^n \int_0^1 \left(t \cdot |x_i^T (\boracle - \betastar)| \; dt \right) \cdot |x_i^T v| \cdot |x_i^T w| \\
& = \kappa_3 \cdot \frac{1}{n} \sum_{i=1}^n \int_0^1 |x_i^T (\boracle - \betastar)| \cdot |x_i^T v| \cdot |x_i^T w| \\
& \le \kappa_3 \|\boracle - \betastar\|_2 \cdot \sup_{\|u\|_2 = 1, \; \supp(u) \subseteq S} \left\{\frac{1}{n} \sum_{i=1}^n |x_i^T u| \cdot |x_i^T v| \cdot |x_i^T w| \right\},
\end{align*}
and denote $f(u, v, w) \defn \frac{1}{n} \sum_{i=1}^n |x_i^T u| \cdot |x_i^T v| \cdot |x_i^T w|$. Then
\begin{equation}
\label{EqnCroc}
\opnorm{\Qhat^{\mathcal{O}}_{SS} - \left(\nabla^2 \Loss_n(\betastar)\right)_{SS}}_2 \le \kappa_3 \|\boracle - \betastar\|_2 \cdot \sup_{u, v, w \in \Ball_2^S(1)} f(u,v,w).
\end{equation}
We now use a covering argument. Let $\scriptM$ denote a $\frac{1}{4}$-cover of $\Ball_2(1)$. By standard results on metric entropy, we may choose $\scriptM$ such that $|\scriptM| \le c^k$. For all triples $u, v, w \in \Ball_2^S(1)$, we may find $u', v', w' \in \scriptM$ such that
\begin{equation*}
\|u - u'\|_2, \; \|v - v'\|_2, \; \|w - w'\|_2 \le \frac{1}{4}.
\end{equation*}
Furthermore,
\begin{multline}
\label{EqnBrulee}
|f(u, v, w) - f(u', v', w')| \le |f(u, v, w) - f(u', v, w)| \\
+ |f(u', v, w) - f(u', v', w)| + |f(u', v', w) - f(u', v', w')|.
\end{multline}
Note that
\begin{align*}
|f(u, v, w) - f(u', v, w)| & = \left|\frac{1}{n} \sum_{i=1}^n \left(|x_i^T u| - |x_i^T u'|\right) \cdot |x_i^T v| \cdot |x_i^T w|\right| \\
& \le \frac{1}{n} \sum_{i=1}^n \bigg| |x_i^T u| - |x_i^T u'| \bigg| \cdot |x_i^T v| \cdot |x_i^T w| \\
& \le \frac{1}{n} \sum_{i=1}^n |x_i^T(u - u')| \cdot |x_i^T v| \cdot |x_i^T w| \\
& \le \|u - u'\|_2 \cdot \sup_{u, v, w \in \Ball_2^S(1)} f(u, v, w) \\
& \le \frac{1}{4} \cdot \sup_{u, v, w \in \Ball_2^S(1)} f(u, v, w),
\end{align*}
and we may bound the other two terms in the expansion~\eqref{EqnBrulee} analogously. Hence,
\begin{equation*}
\sup_{u, v, w \in \Ball_2^S(1)} f(u, v, w) \le \max_{u', v', w' \in \scriptM} f(u', v', w') + \frac{3}{4} \cdot \sup_{u, v, w \in \Ball_2^S(1)} f(u, v, w),
\end{equation*}
implying that
\begin{equation}
\label{EqnAlligator}
\sup_{u, v, w \in \Ball_2^S(1)} f(u, v, w) \le 4 \cdot \max_{u, v, w \in \scriptM} f(u, v, w),
\end{equation}
and it remains to bound the right-hand side of inequality~\eqref{EqnAlligator}. For a fixed triple $u, v, w \in \scriptM$, we apply the arithmetic mean-geometric mean inequality, to obtain
\begin{equation*}
|x_i^T u| \cdot |x_i^T v| \cdot |x_i^T w| \le \frac{1}{3} \left(|x_i^T u|^3 + |x_i^T v|^3 + |x_i^T w|^3\right).
\end{equation*}
Then
\begin{equation*}
f(u, v, w) \le \frac{1}{3} \left(\frac{1}{n} \sum_{i=1}^n |x_i^T u|^3 + \frac{1}{n} \sum_{i=1}^n |x_i^T v|^3 + \frac{1}{n} \sum_{i=1}^n |x_i^T w|^3\right).
\end{equation*}
Note that
\begin{equation*}
\E\left[f(u, v, w)\right] = \frac{1}{3} \bigg(\E\left[|x_i^T u|^3\right] + \E\left[|x_i^T v|^3\right] + \E\left[|x_i^T w|^3\right]\bigg) \le c \sigma_x^3,
\end{equation*}
using the sub-Gaussian assumption on the $x_i$'s. Finally, we invoke a concentration bound on $f(u, v, w)$. We use a result from Adamczak and Wolff~\cite{AdaWol14}. Theorem 1.4 of that paper gives a concentration result for i.i.d.\ averages of polynomials of sub-Gaussian variables, implying in particular that
\begin{equation*}
\mprob\left(\left|f(u,v,w) - \E\left[f(u,v,w)\right]\right| \ge t\right) \le c_1 \exp\left(\min\left\{\frac{nt^2}{\sigma_x^6}, \; \frac{(nt)^{2/3}}{\sigma_x^2}\right\}\right), \qquad \forall t > 0.
\end{equation*}
Setting $t = c \sigma_x^3 \sqrt{\frac{k}{n}}$ and taking a union bound over all $u, v, w \in \scriptM$, we then conclude from inequality~\eqref{EqnAlligator} that
\begin{equation*}
\sup_{u, v, w \in \Ball_2^S(1)} f(u, v, w) \le c\sigma_x^3 \sqrt{\frac{k}{n}},
\end{equation*}
with probability at least $1 - c_1' \exp(-c_2' k)$. Plugging back into inequality~\eqref{EqnCroc} then gives
\begin{equation}
\label{EqnFlour}
\opnorm{\Qhat^{\mathcal{O}}_{SS} - \left(\nabla^2 \Loss_n(\betastar)\right)_{SS}}_2 \le c \kappa_3 \sigma_x^3 \|\boracle - \betastar\|_2 \cdot \sqrt{\frac{k}{n}} \le c\kappa_3 \sigma_x^3 r \sqrt{\frac{k}{n}}.
\end{equation}
We further note that
\begin{equation*}
v^T \left\{\left(\nabla^2 \Loss_n(\betastar)\right)_{SS} \right\} w  = \frac{1}{n} \sum_{i=1}^n \ell''(x_i^T \betastar - y_i) \cdot (x_i^T v) \cdot (x_i^T w),
\end{equation*}
which is an i.i.d.\ average of products of sub-Gaussians (since $\ell''$ is bounded), so an even easier covering argument establishes concentration to $v^T \left\{\left(\nabla^2 \Loss(\betastar)\right)_{SS} \right\} w$. Hence,
\begin{equation}
\label{EqnSugar}
\opnorm{\Qhat^{\mathcal{O}}_{SS} - \left(\nabla^2 \Loss(\betastar)\right)_{SS}}_2 \le c' \sqrt{\frac{k}{n}}.
\end{equation}
Combining inequalities~\eqref{EqnFlour} and~\eqref{EqnSugar} gives
\begin{equation*}
\opnorm{\Qhat^{\mathcal{O}}_{SS} - \left(\nabla^2 \Loss(\betastar)\right)_{SS}}_2 \le c'' \sqrt{\frac{k}{n}}.
\end{equation*}
Then by a simple matrix inversion relation (cf.\ Lemma 12 in Loh and Wainwright~\cite{LohWai14}), we have
\begin{equation*}
\opnorm{(\Qhat^{\mathcal{O}}_{SS})^{-1} - \left(\nabla^2 \Loss(\betastar)\right)_{SS}^{-1}}_2 \le c'' \sqrt{\frac{k}{n}},
\end{equation*}
as well. Returning to equation~\eqref{EqnLochness}, we see that
\begin{align*}
& \|\boracle_S - \betastar_S\|_\infty \\
& \qquad \le \left\|\left\{(\Qhat^{\mathcal{O}}_{SS})^{-1} - \left(\nabla^2 \Loss(\betastar)\right)_{SS}^{-1}\right\} \left(\nabla \Loss_n(\betastar)\right)_S\right\|_\infty + \left\|\left(\nabla^2 \Loss(\betastar)\right)_{SS}^{-1} \left(\nabla \Loss_n(\betastar)\right)_S\right\|_\infty \\
& \qquad \le c'' \sqrt{\frac{k}{n}} \cdot \sqrt{k} \left\|\left(\nabla \Loss_n(\betastar)\right)_S\right\|_\infty + \left\|\left(\nabla^2 \Loss(\betastar)\right)_{SS}^{-1} \left(\nabla \Loss_n(\betastar)\right)_S\right\|_\infty \\
& \qquad \le C \sqrt{\frac{\log k}{n}},
\end{align*}
assuming $n \ge k^2$. This is the desired result.

\end{proof}

In particular, Lemma~\ref{LemOracle} implies when $\betamin \ge C \sqrt{\frac{\log k}{n}} + \gamma \lambda$, we have
\begin{equation*}
\nabla q_\lambda(\betahat_S) = \lambda \sign(\betahat_S) = \lambda \zhat_S.
\end{equation*}
Furthermore, the selection property implies $\nabla q_\lambda(\betahat_{S^c}) = 0$. Plugging these results into equation~\eqref{EqnZeroSubBlock} and performing some algebra, we conclude that
\begin{equation}
\label{EqnRabbit}
\zhat_{S^c} = \frac{1}{\lambda} \Big\{\Qhat_{S^cS}
(\Qhat_{SS})^{-1} \left(\nabla
\Loss_n(\betastar)\right)_S - \left(\nabla
\Loss_n(\betastar)\right)_{S^c} \Big\},
\end{equation}
so
\begin{equation}
\label{EqnCotton}
\|\zhat_{S^c}\|_\infty \le \frac{1}{\lambda} \left\|\Qhat_{S^cS} (\Qhat_{SS})^{-1} \left(\nabla \Loss_n(\betastar)\right)_S \right\|_\infty + \frac{1}{\lambda} \|\left(\nabla \Loss_n(\betastar)\right)_{S^c}\|_\infty.
\end{equation}
We now use similar arguments to those employed in the proof of Lemma~\ref{LemOracle} to control the terms in inequality~\eqref{EqnCotton}. Note that $\|\nabla \Loss_n(\betastar)\|_\infty \le c \sqrt{\frac{\log p}{n}}$ by assumption, so we can focus on the first term. We have
\begin{align*}
\Big| v^T \Big(\Qhat - \nabla^2 & \Loss_n(\betastar) \Big) w\Big| \\
& = \left|\frac{1}{n} \sum_{i=1}^n \int_0^1 \left(\ell''\left(x_i^T\big(\betahat + t(\betahat - \betastar)\big) - y_i\right) - \ell''(x_i^T \betastar - y_i)\right) dt \cdot (x_i^T v) (x_i^T w)\right| \\
& \le \kappa_3 \cdot \frac{1}{n} \sum_{i=1}^n \int_0^1 \left(t \cdot |x_i^T (\betahat - \betastar)| \right) dt \cdot |x_i^T v| \cdot |x_i^T w| \\
& \le \kappa_3 \|\betahat - \betastar\|_2 \cdot \sup_{u \in \Ball_2^S(1)} \left\{\frac{1}{n} \sum_{i=1}^n |x_i^T u| \cdot |x_i^T v| \cdot |x_i^T w| \right\} \\
& \le \kappa_3 r \cdot \sup_{u \in \Ball_2^S(1)} \left\{\frac{1}{n} \sum_{i=1}^n |x_i^T u| \cdot |x_i^T v| \cdot |x_i^T w|\right\}.
\end{align*}
By essentially the same bounding and covering argument as before, we conclude that
\begin{equation*}
\opnorm{\Qhat_{SS} - \left(\nabla^2 \Loss(\betastar)\right)_{SS}}_2 \le c \sqrt{\frac{k}{n}},
\end{equation*}
and
\begin{equation}
\label{EqnYellow}
\opnorm{(\Qhat_{SS})^{-1} - \left(\nabla^2 \Loss(\betastar)\right)_{SS}^{-1}}_2 \le c' \sqrt{\frac{k}{n}},
\end{equation}
with probability at least $1 - c_1 \exp(- c_2 k)$. Furthermore, we may show that
\begin{equation}
\label{EqnGreen}
\max_{j \in S^c} \left\|e_j^T \left\{\Qhat_{S^cS} - \left(\nabla^2 \Loss(\betastar)\right)_{S^cS} \right\} \right\|_2 \le c''' \sqrt{\frac{k + \log p}{n}},
\end{equation}
with probability at least $1 - c_1' \exp(-c_2' \min\{k, \log p\})$ by a similar argument, this time taking a union bound over $j \in S^c$ rather than all unit vectors for one of the coordinates in the covering. Defining
\begin{align*}
\delta_1 \defn \Qhat_{S^cS} - \left(\nabla^2 \Loss(\betastar)\right)_{S^cS}, \qquad \text{and} \qquad \delta_2 \defn (\Qhat_{SS})^{-1} - \left(\nabla^2 \Loss(\betastar)\right)_{SS}^{-1},
\end{align*}
we may conclude that
\begin{align*}
& \left\|\Qhat_{S^cS} (\Qhat_{SS})^{-1} \left(\nabla \Loss_n(\betastar)\right)_S \right\|_\infty \le \left\|\delta_1 \delta_2 \left(\nabla \Loss_n(\betastar)\right)_S\right\|_\infty + \left\|\delta_1 \left(\nabla^2 \Loss(\betastar)\right)_{SS}^{-1} \left(\nabla \Loss_n(\betastar)\right)_S\right\|_\infty \\
& \qquad \qquad + \left\|\left(\nabla^2 \Loss(\betastar)\right)_{S^cS} \delta_2 \left(\nabla \Loss_n(\betastar)\right)_S\right\|_\infty \\
& \qquad \le \max_{j \in S^c} \|e_j^T \delta_1\|_2 \cdot \opnorm{\delta_2}_2 \cdot \sqrt{k} \|\nabla \Loss_n(\betastar)\|_\infty + \max_{j \in S^c} \|e_j^T \delta_1\|_2 \cdot \sqrt{k} \left\|\left(\nabla^2 \Loss(\betastar)\right)_{SS}^{-1} \left(\nabla \Loss_n(\betastar)\right)_S\right\|_\infty \\
& \qquad \qquad + \opnorm{\left(\nabla^2 \Loss(\betastar)\right)_{S^cS}}_2 \cdot \opnorm{\delta_2}_2 \cdot \sqrt{k} \left\|\nabla \Loss_n(\betastar)\right\|_\infty \\
& \qquad \le C \sqrt{\frac{\log p}{n}},
\end{align*}
with probability at least $1 - c_1' \exp(-c_2' \min\{k, \log p\})$, assuming the scaling $n \ge \max\{k^2, k\log p\}$ and using the inqualities~\eqref{EqnYellow} and~\eqref{EqnGreen} above. In particular, for $\lambda \ge C' \sqrt{\frac{\log p}{n}}$, we conclude at last that the strict dual feasibility condition $\|\zhat_{S^c}\|_\infty < 1$ holds, completing step (ii) of the PDW construction. \\

\paragraph{Step (iii):}

Finally, we establish that $\betahat = (\betahat_S, 0_{S^c})$ is a local minimum of the full program~\eqref{EqnGeneral} and in fact, all stationary points of the program must take this form. A classical result by Fletcher and Watson~\cite{FleWat80} gives sufficient conditions for a point to be a local minimum of a norm-regularized program. Rather than repeating the details here, we refer the reader to the argument provided in the proof of Theorem 1 in Loh and Wainwright~\cite{LohWai14}, which may be applied verbatim to establish that $\betahat$ is a local minimum. Now suppose $\betatil$ is a stationary point of the program~\eqref{EqnGeneral} satisfying $\|\betatil - \betastar\|_2 \le r$. By the RSC condition~\eqref{EqnLocalRSC} applied to the pair $(\betatil, \betahat)$, we have
\begin{equation}
\label{EqnBarneys}
\alpha \|\betatil - \betahat\|_2^2 - \tau \frac{\log p}{n} \|\betatil - \betahat\|_1^2 \le \inprod{\nabla \Loss_n(\betatil) - \nabla \Loss_n(\betahat)}{\betatil - \betahat}.
\end{equation}
By the convexity of $\frac{\mu}{2} \|\beta\|_2^2 - q_\lambda(\beta)$, we also have
\begin{equation}
\label{EqnCrepevine}
\inprod{\nabla q_\lambda(\betatil) - \nabla q_\lambda(\betahat)}{\betatil - \betahat} \le \mu \|\betatil - \betahat\|_2^2.
\end{equation}
Finally, the first-order optimality condition applied to $\betatil$ gives
\begin{equation}
\label{EqnCheeseboard}
0 \le \inprod{\nabla \Loss_n(\betatil) - \nabla q_\lambda(\betatil)}{\betahat - \betatil} + \lambda \cdot \inprod{\ztil}{\betahat - \betatil},
\end{equation}
where $\ztil \in \partial \|\betatil\|_1$. Summing the inequalities~\eqref{EqnBarneys}, \eqref{EqnCrepevine}, and~\eqref{EqnCheeseboard}, we obtain
\begin{equation}
\label{EqnUdupi}
(\alpha - \mu) \|\betatil - \betahat\|_2^2 - \tau \frac{\log p}{n} \|\betatil - \betahat\|_1^2 \le \inprod{\nabla q_\lambda(\betahat) - \nabla \Loss_n(\betahat)}{\betatil - \betahat} + \lambda \cdot \inprod{\ztil}{\betahat - \betatil}.
\end{equation}
Recall that since $\betahat$ is an interior point, we have the zero-subgradient condition
\begin{equation*}
\nabla \Loss_n(\betahat) - \nabla q_\lambda(\betahat) + \lambda \zhat = 0.
\end{equation*}
Combining this with inequality~\eqref{EqnUdupi}, we obtain
\begin{align}
\label{EqnTibet}
(\alpha - \mu) \|\betatil - \betahat\|_2^2 - \tau \frac{\log p}{n} \|\betatil - \betahat\|_1^2 & \le \lambda \cdot \inprod{\zhat}{\betatil - \betahat} + \lambda \cdot \inprod{\ztil}{\betahat - \betatil} \notag \\
& = \lambda \cdot \inprod{\zhat}{\betatil} - \lambda \|\betahat\|_1 + \lambda \cdot \inprod{\ztil}{\betahat} - \lambda \|\betatil\|_1 \notag \\
& \le \lambda \cdot \inprod{\zhat}{\betatil} - \lambda \|\betatil\|_1.
\end{align}

We now show the following lemma:
\begin{lem*}
\label{LemTilCone}
Suppose $\delta > 0$ is such that $\|\zhat_{S^c}\|_\infty \le 1 - \delta$. Then for $\lambda \ge \frac{4R\tau}{\log p}{\delta n}$, we have
\begin{equation*}
\|\betatil - \betahat\|_1 \le \left(\frac{4}{\delta} + 2\right) \sqrt{k} \|\betatil - \betahat\|_2.
\end{equation*}
\end{lem*}

\begin{proof}
This is identical to the proof of Lemma 7 in Loh and Wainwright~\cite{LohWai14}.
\end{proof}

Using Lemma~\ref{LemTilCone} to bound the left-hand side of inequality~\eqref{EqnTibet}, we then obtain
\begin{equation*}
\left(\alpha - \mu - \tau \frac{k \log p}{n} \left(\frac{4}{\delta} + 2\right)^2\right) \|\betatil - \betahat\|_2^2 \le \lambda \cdot \inprod{\zhat}{\betatil} - \lambda \|\betatil\|_1,
\end{equation*}
so if $n \ge \frac{2 \tau}{\alpha - \mu} \left(\frac{4}{\delta} + 2\right)^2 k \log p$, this implies
\begin{equation*}
0 \le \lambda \cdot \inprod{\zhat}{\betatil} - \lambda \|\betatil\|_1.
\end{equation*}
At the same time, H\"{o}lder's inequality gives
\begin{equation*}
\lambda \cdot \inprod{\zhat}{\betatil} - \lambda \|\betatil\|_1 \le \lambda \cdot \|\zhat\|_\infty \|\betatil\|_1 - \lambda \|\betatil\|_1 \le \lambda \|\betatil\|_1 - \lambda \|\betatil\|_1 = 0.
\end{equation*}
Hence, we must have $\inprod{\zhat}{\betatil} = \|\betatil\|_1$. Since $\|\zhat_{S^c}\|_\infty < 1$ by assumption, this means that $\supp(\betatil) \subseteq S$, as wanted.

%%%%%%%%%

\subsection{Proof of Theorem~\ref{ThmOptimization}}
\label{SecThmOptimization}

We derive the following variants of Lemmas 1 and 2 in Loh and Wainwright~\cite{LohWai13}; the remainder of the argument is exactly the same as in that paper, so we do not repeat it here. The reason why we need to revise the two lemmas is that the proofs in Loh and Wainwright~\cite{LohWai13} require the statement of the RSC condition in that paper, which also provides control on the behavior of $\Loss_n$ outside the local region.

\begin{lem*}
\label{LemBasin}
Under the conditions of the theorem, we have
\begin{equation*}
\|\beta^t - \betahat\|_2 \le \frac{r}{2}, \qquad \forall t \ge 0.
\end{equation*}
\end{lem*}

\begin{proof}

We induct on the iteration number $t$. Note that the base case, $t = 0$, holds by assumption. Suppose $t \ge 0$ is such that $\|\beta^t - \betahat\|_2 \le \frac{r}{2}$; we will show that $\|\beta^{t+1} - \betahat\|_2 \le \frac{r}{2}$, as well.

By the RSC condition~\eqref{EqnRSC2}, we have
\begin{equation}
\label{EqnCrumbCake}
\alpha' \|\beta^t - \betahat\|_2^2 - \tau' \frac{\log p}{n} \|\beta^t - \betahat\|_1^2 \le \Loss_n(\betahat) - \Loss_n(\beta^t) - \inprod{\nabla \Loss_n(\beta^t)}{\betahat - \beta^t}.
\end{equation}
Furthermore, since $\frac{\mu}{2} \|\beta\|_2^2 - q_\lambda(\beta)$ is convex by the $\mu$-amenability of $\rho_\lambda$, combining inequality~\eqref{EqnQConcave} with $(\beta_1, \beta_2) = (\beta^t, \betahat)$ and inequality~\eqref{EqnCrumbCake} and the inequality
\begin{equation*}
\|\beta^{t+1}\|_1 + \inprod{\sign(\beta^{t+1})}{\betahat - \beta^{t+1}} \le \|\betahat\|_1
\end{equation*}
implies that
\begin{multline*}
\Loss_n(\beta^t) + \inprod{\nabla \Loss_n(\beta^t) - \nabla q_\lambda(\beta^t)}{\betahat - \beta^t} - q_\lambda(\beta^t) + \lambda \|\beta^{t+1}\|_1 + \lambda \inprod{\sign(\beta^{t+1})}{\betahat - \beta^{t+1}} \\
+ \left(\alpha' - \frac{\mu}{2}\right) \|\beta^t - \betahat\|_2^2 - \tau' \frac{\log p}{n} \|\beta^t - \betahat\|_1^2 \le \Loss_n(\betahat) - q_\lambda(\betahat) + \lambda \|\betahat\|_1,
\end{multline*}
so
\begin{multline}
\label{EqnCaramel}
\Lossbar_n(\beta^t) + \inprod{\nabla \Lossbar_n(\beta^t)}{\betahat - \beta^t} + \lambda \|\beta^{t+1}\|_1 + \lambda \inprod{\sign(\beta^{t+1})}{\betahat - \beta^{t+1}} - \tau' \frac{\log p}{n} \|\beta^t - \betahat\|_1^2 \\
\le \Lossbar_n(\betahat) + \lambda \|\betahat\|_1.
\end{multline}
By the RSM condition~\eqref{EqnRSM}, we have
\begin{equation*}
\Loss_n(\beta^{t+1}) - \Loss_n(\beta^t) - \inprod{\nabla \Loss_n(\beta^t)}{\beta^{t+1} - \beta^t} \le \alpha'' \|\beta^{t+1} - \beta^t\|_2^2 + \tau'' \frac{\log p}{n} \|\beta^{t+1} - \beta^t\|_1^2,
\end{equation*}
and combined with the convexity of $q_\lambda$, we have
\begin{equation}
\label{EqnSalty}
\Lossbar_n(\beta^{t+1}) - \Lossbar_n(\beta^t) - \inprod{\nabla \Lossbar_n(\beta^t)}{\beta^{t+1} - \beta^t} \le \alpha'' \|\beta^{t+1} - \beta^t\|_2^2 + \tau'' \frac{\log p}{n} \|\beta^{t+1} - \beta^t\|_1^2.
\end{equation}
Combining inequalities~\eqref{EqnCaramel} and~\eqref{EqnSalty} then gives
\begin{multline}
\label{EqnBrookie}
\left(\Lossbar_n(\beta^{t+1}) + \lambda \|\beta^{t+1}\|_1\right) - \left(\Lossbar_n(\betahat) + \lambda \|\betahat\|_1\right) \\
\le \inprod{\nabla \Lossbar_n(\beta^t)}{\beta^{t+1} - \betahat} + \lambda \inprod{\sign(\beta^{t+1})}{\beta^{t+1} - \betahat} + \alpha'' \|\beta^{t+1} - \beta^t\|_2^2 + 4R^2 (\tau' + \tau'') \frac{\log p}{n},
\end{multline}
using the fact that $\|\beta^{t+1} - \beta^t\|_1, \; \|\beta^t - \betahat\|_1 \le 2R$, by feasibility of each point. Note that the left-hand side of inequality~\eqref{EqnBrookie} is lower-bounded by 0, since $\betahat$ is a global optimum. Finally, note that from the first-order optimality condition on equation~\eqref{EqnIterates}, we have
\begin{equation}
\label{EqnFOC}
\inprod{\nabla \Lossbar_n(\beta^t) + \eta (\beta^{t+1} - \beta^t) + \lambda \sign(\beta^{t+1})}{\beta^{t+1} - \betahat} \le 0.
\end{equation}
Combining inequality~\eqref{EqnFOC} with inequality~\eqref{EqnBrookie} then gives
\begin{align}
\label{EqnIceCream}
0 & \le \left(\Lossbar_n(\beta^{t+1}) + \lambda \|\beta^{t+1}\|_1\right) - \left(\Lossbar_n(\betahat) + \lambda \|\betahat\|_1\right) \notag \\
& \le \alpha'' \|\beta^{t+1} - \beta^t\|_2^2 + (\tau' + \tau'') \frac{4R^2 \log p}{n} - \eta \inprod{\beta^{t+1} - \beta^t}{\beta^{t+1} - \betahat} \notag \\
& = \left(\alpha'' - \frac{\eta}{2}\right) \|\beta^{t+1} - \beta^t\|_2^2 - \frac{\eta}{2} \|\beta^{t+1} - \betahat\|_2^2 + \frac{\eta}{2} \|\beta^t - \betahat\|_2^2 + (\tau' + \tau'') \frac{4R^2 \log p}{n} \notag \\
& \le \frac{\eta}{2} \|\beta^t - \betahat\|_2^2 - \frac{\eta}{2} \|\beta^{t+1} - \betahat\|_2^2 + (\tau' + \tau'') \frac{4R^2 \log p}{n},
\end{align}
using the assumption that $\eta \ge 2 \alpha''$. Hence,
\begin{equation*}
\|\beta^{t+1} - \betahat\|_2^2 \le \|\beta^t - \betahat\|_2^2 + \frac{8(\tau' + \tau'')}{\eta} \cdot \frac{R^2 \log p}{n}.
\end{equation*}
Using the inductive hypothesis and the assumption that $n \ge \frac{32(\tau' + \tau'')R^2}{\eta r^2} \log p$, we then have
\begin{equation*}
\|\beta^{t+1} - \betahat\|_2^2 \le \frac{r^2}{4} + \frac{r^2}{4} \le r^2.
\end{equation*}
In particular, we may apply the RSC condition~\eqref{EqnRSC2} to the pair $(\beta^{t+1}, \betahat)$ to obtain
\begin{equation*}
\alpha' \|\beta^{t+1} - \betahat\|_2^2 - \tau'  \frac{\log p}{n} \|\beta^{t+1} - \betahat\|_1^2 \le \Loss_n(\beta^{t+1}) - \Loss_n(\betahat) - \inprod{\nabla \Loss_n(\betahat)}{\beta^{t+1} - \betahat}.
\end{equation*}
By the convexity of $\frac{\mu}{2} \|\beta\|_2^2 - q_\lambda(\beta)$, we have
\begin{equation}
\label{EqnQConcave}
\inprod{\nabla q_\lambda(\betahat)}{\beta^{t+1} - \betahat} \ge q_\lambda(\beta^{t+1}) - q_\lambda(\betahat) - \frac{\mu}{2} \|\betahat - \beta^{t+1}\|_2^2.
\end{equation}
Together with the inequality
\begin{equation*}
\|\betahat\|_1 + \inprod{\sign(\betahat)}{\beta^{t+1} - \betahat} \le \|\beta^{t+1}\|_1,
\end{equation*}
we then have
\begin{multline}
\label{EqnCrumble}
\left(\alpha' - \frac{\mu}{2}\right)  \|\beta^{t+1} - \betahat\|_2^2 - \tau' \frac{\log p}{n} \|\beta^{t+1} - \betahat\|_1^2 \\
\le \left(\Lossbar_n(\beta^{t+1}) + \lambda \|\beta^{t+1}\|_1\right) - \left(\Lossbar_n(\betahat) + \lambda \|\betahat\|_1\right) - \inprod{\nabla \Lossbar_n(\betahat) + \lambda \sign(\betahat)}{\beta^{t+1} - \betahat}.
\end{multline}
Finally, the first-order optimality condition on $\betahat$ gives
\begin{equation*}
\inprod{\nabla \Lossbar_n(\betahat) + \lambda \sign(\betahat)}{\beta^{t+1} - \betahat} \ge 0.
\end{equation*}
Combined with inequality~\eqref{EqnCrumble}, we conclude that
\begin{equation}
\label{EqnBlondie}
\left(\alpha' - \frac{\mu}{2}\right) \|\beta^{t+1} - \betahat\|_2^2 - \tau' \frac{\log p}{n} \|\beta^{t+1} - \betahat\|_1^2 \le \left(\Lossbar_n(\beta^{t+1}) + \lambda \|\beta^{t+1}\|_1\right) - \left(\Lossbar_n(\betahat) + \lambda \|\betahat\|_1\right).
\end{equation}
Inequality~\eqref{EqnIceCream} gives an upper bound on the right-hand side of inequality~\eqref{EqnBlondie}. Combining the two inequalities, we then have
\begin{equation*}
\left(\alpha' - \frac{\mu}{2}\right) \|\beta^{t+1} - \betahat\|_2^2 - \tau' \frac{\log p}{n} \|\beta^{t+1} - \betahat\|_1^2 \le \frac{\eta}{2} \|\beta^t - \betahat\|_2^2 - \frac{\eta}{2} \|\beta^{t+1} - \betahat\|_2^2 + (\tau' + \tau'') \frac{4R^2 \log p}{n}.
\end{equation*}
Hence,
\begin{align*}
\|\beta^{t+1} - \betahat\|_2^2 & \le \frac{\eta/2}{\alpha' - \mu/2 + \eta/2} \|\beta^t - \betahat\|_2^2 \\
& \qquad + \frac{1}{\alpha' - \mu/2 + \eta/2} \left((\tau' + \tau'') \frac{4R^2\log p}{n} + \tau' \frac{\log p}{n} \|\beta^{t+1} - \betahat\|_1^2\right) \\
& \le \frac{\eta/2}{\alpha' - \mu/2 + \eta/2} \|\beta^t - \betahat\|_2^2 + \frac{4(2\tau' + \tau'')}{\alpha' - \mu/2 + \eta/2} \cdot \frac{R^2 \log p}{n}.
\end{align*}
Using the inductive hypothesis one more time and the scaling assumption~\eqref{EqnScaling},
we conclude that
\begin{equation*}
\|\beta^{t+1} - \betahat\|_2^2 \le \frac{\eta/2}{\alpha' - \mu/2 + \eta/2} \cdot \frac{r^2}{4} + \frac{4(2\tau' + \tau'')}{\alpha' - \mu/2 + \eta/2} \cdot \frac{R^2 \log p}{n} \le \frac{r^2}{4},
\end{equation*}
completing the induction.

\end{proof}

\begin{lem*}
\label{LemCone}
Under the conditions of the theorem, suppose there exists a pair $(\etabar, T)$ such that
\begin{equation*}
\phi(\beta^t) - \phi(\betahat) \le \etabar, \qquad \forall t \ge T.
\end{equation*}
Then for any iteration $t \ge T$, we have
\begin{equation*}
\|\beta^t - \betahat\|_1 \le 8 \sqrt{k} \|\beta^t - \betahat\|_2 + 16 \sqrt{k} \|\betahat - \betastar\|_2 + 2 \cdot \min\left(\frac{2\etabar}{\lambda}, R\right).
\end{equation*}
\end{lem*}

\begin{proof}

This proof is in fact a simplification of the argument used to prove Lemma 1 in Loh and Wainwright~\cite{LohWai13}, since by Lemma~\ref{LemBasin} and the assumption, we are guaranteed that
\begin{equation*}
\|\beta^t - \betastar\|_2 \le \|\beta^t - \betahat\|_2 + \|\betahat - \betastar\|_2 \le \frac{r}{2} + \frac{r}{2} = r,
\end{equation*}
so we may apply the RSC condition~\eqref{EqnRSC2} directly. Denoting $\Delta \defn \beta^t - \betastar$, we then have
\begin{align}
\label{EqnMuffin}
\alpha' \|\Delta\|_2^2 - \tau' \frac{\log p}{n} \|\Delta\|_1^2 & \le \Loss_n(\beta^t) - \Loss_n(\betastar) - \inprod{\nabla \Loss_n(\betastar)}{\Delta} \notag \\
& \le \Loss_n(\beta^t) - \Loss_n(\betastar) + \|\nabla \Loss_n(\betastar)\|_\infty \cdot \|\Delta\|_1 \notag \\
& \le \Loss_n(\beta^t) - \Loss_n(\betastar) + \frac{\lambda}{8} \|\Delta\|_1.
\end{align}

Furthermore, by assumption, we have
\begin{equation}
\label{EqnWhite}
\Loss_n(\beta^t) - \Loss_n(\betastar) + \rho_\lambda(\beta^t) - \rho_\lambda(\betastar) \le \etabar,
\end{equation}
which combined with inequality~\eqref{EqnMuffin} implies that
\begin{equation}
\label{EqnLemon}
\alpha' \|\Delta\|_2^2 - \tau' \frac{\log p}{n}  \|\Delta\|_1^2 \le \rho_\lambda(\betastar) - \rho_\lambda(\beta^t) + \etabar + \frac{\lambda}{8} \|\Delta\|_1.
\end{equation}
Note that if $\etabar \ge \frac{\lambda}{4} \|\Delta\|_1$, the desired inequality is trivial. Hence, we assume that $\etabar \le \frac{\lambda}{4} \|\Delta\|_1$. In particular, inequality~\eqref{EqnLemon} implies that
\begin{align*}
\alpha' \|\Delta\|_2^2 & \le \rho_\lambda(\betastar) - \rho_\lambda(\beta^t) + \tau' \frac{\log p}{n} \|\Delta\|_1^2 + \frac{3\lambda}{8} \|\Delta\|_1 \\
& \le \rho_\lambda(\betastar) - \rho_\lambda(\beta^t) + 2\tau'R\frac{\log p}{n} \|\Delta\|_1 + \frac{3\lambda}{8} \|\Delta\|_1 \\
& \le \rho_\lambda(\betastar) - \rho_\lambda(\beta^t) + \frac{\lambda}{2} \|\Delta\|_1 \\
& \le \rho_\lambda(\betastar) - \rho_\lambda(\beta^t) + \frac{\rho_\lambda(\Delta)}{2} + \frac{\mu}{4} \|\Delta\|_2^2 \\
& \le \rho_\lambda(\betastar) - \rho_\lambda(\beta^t) + \frac{\rho_\lambda(\betastar) + \rho_\lambda(\beta^t)}{2} + \frac{\mu}{4} \|\Delta\|_2^2 \\
& \le \frac{3}{2} \rho_\lambda(\betastar) - \frac{1}{2} \rho_\lambda(\beta^t) + \frac{\mu}{4} \|\Delta\|_2^2.
\end{align*} 
Hence, we have
\begin{equation*}
0 \le \left(\alpha' - \frac{\mu}{4}\right) \|\Delta\|_2^2 \le \frac{3}{2} \rho_\lambda(\betastar) - \frac{1}{2} \rho_\lambda(\beta^t).
\end{equation*}
By Lemma 5 in Loh and Wainwright~\cite{LohWai13}, we then have
\begin{equation}
\label{EqnOrange}
\rho_\lambda(\betastar) - \rho_\lambda(\beta^t) \le 3\rho_\lambda(\betastar) - \rho_\lambda(\beta^t) \le \lambda \big(3\|\Delta_A\|_1 - \|\Delta_{A^c}\|_1\big),
\end{equation}
where $A$ indexes the top $k$ components of $\Delta$ in magnitude. Combining inequalities~\eqref{EqnLemon} and~\eqref{EqnOrange}, we then have
\begin{equation*}
\alpha' \|\Delta\|_2^2 - \tau' \frac{\log p}{n} \|\Delta\|_1^2 \le 3\lambda \|\Delta_A\|_1 - \lambda \|\Delta_{A^c}\|_1 + \etabar + \frac{\lambda}{8} \|\Delta\|_1,
\end{equation*}
so
\begin{align*}
0 \le \alpha' \|\Delta\|_2^2 & \le 3\lambda \|\Delta_A\|_1 - \lambda \|\Delta_{A^c}\|_1 + \etabar + \frac{\lambda}{8} \|\Delta\|_1 + 2\tau'R \frac{\log p}{n} \|\Delta\|_1 \\
& \le 3 \lambda \|\Delta_A\|_1 - \lambda \|\Delta_{A^c}\|_1 + \etabar + \frac{\lambda}{2} \|\Delta\|_1 \\
& \le \frac{7\lambda}{2} \|\Delta_A\|_1 - \frac{\lambda}{2} \|\Delta_{A^c}\|_1 + \etabar.
\end{align*}
Hence,
\begin{equation*}
\|\Delta_{A^c}\|_1 \le 7 \|\Delta_A\|_1 + \frac{2\etabar}{\lambda},
\end{equation*}
so
\begin{equation*}
\|\Delta\|_1 \le \|\Delta_A\|_1 + \|\Delta_{A^c}\|_1 \le 8 \|\Delta_A\|_1 + \frac{2\etabar}{\lambda} \le 8 \sqrt{k} \|\Delta\|_2 + \frac{2\etabar}{\lambda}.
\end{equation*}
Also, $\|\Delta\|_1 \le 2R$, so we clearly have
\begin{equation}
\label{EqnVanilla}
\|\Delta\|_1 \le 8\sqrt{k} \|\Delta\|_2 + 2 \cdot \min\left(\frac{2\etabar}{\lambda}, \; R\right).
\end{equation}
Further note that by essentially the same argument, with inequality~\eqref{EqnWhite} replaced by
\begin{equation*}
\Loss_n(\betahat) - \Loss_n(\betastar) + \rho_\lambda(\betahat) - \rho_\lambda(\betastar) \le 0,
\end{equation*}
we have the inequality
\begin{equation}
\label{EqnButterscotch}
\|\betahat - \betastar\|_1 \le 8 \sqrt{k} \|\betahat - \betastar\|_2.
\end{equation}
Combining inequalities~\eqref{EqnVanilla} and~\eqref{EqnButterscotch} and using the triangle inequality then yields
\begin{align*}
\|\beta^t - \betahat\|_1 & \le \|\betahat - \betastar\|_1 + \|\beta^t - \betastar\|_1 \\
& \le 8 \sqrt{k} \left(\|\betahat - \betastar\|_2 + \|\beta^t - \betastar\|_2\right) + 2 \cdot \min\left(\frac{2\etabar}{\lambda}, \; R\right) \\
& \le 8 \sqrt{k} \left(\|\betahat - \betastar\|_2 + \|\beta^t - \betahat\|_2 + \|\betahat - \betastar\|_2 \right) + 2 \cdot \min \left(\frac{2\etabar}{\lambda}, \; R \right) \\
& = 8 \sqrt{k} \left(\|\beta^t - \betahat\|_2 + 2 \|\betahat - \betastar\|_2\right) + 2 \cdot \min\left(\frac{2\etabar}{\lambda}, \; R\right),
\end{align*}
completing the proof.

\end{proof}

%%%%%%%%%%

\section{Proofs of propositions in Section~\ref{SecRSC}}
\label{AppSecRSC}

In this Appendix, we provide the proofs of the technical propositions establishing sufficient conditions for statistical consistency of stationary points in Section~\ref{SecRSC}.

\subsection{Proof of Proposition~\ref{PropGradBound}}
\label{AppPropGradBound}

We have
\begin{equation*}
\|\nabla \Loss_n(\betastar)\|_\infty = \left\|\frac{1}{n} \sum_{i=1}^n w(x_i) x_i \cdot \ell'\left(\epsilon_i \cdot v(x_i)\right)\right\|_\infty.
\end{equation*}
Since $x_i \condind \epsilon_i$ by assumption, the tower property of conditional expectation gives
\begin{equation}
\label{EqnCandyCane}
\E\left[w(x_i) x_i \cdot \ell'(\epsilon_i \cdot v(x_i))\right] = \E\Bigg[\E\left[\ell'(\epsilon_i \cdot v(x_i)) \mid x_i\right] \cdot w(x_i) x_i\Bigg]
\end{equation}
Under condition (2a), the right-hand expression of equation~\eqref{EqnCandyCane} may be written as
\begin{equation*}
\E\Bigg[\E\left[\ell'(\epsilon_i) \mid x_i\right] \cdot w(x_i) x_i\Bigg] = \E\Big[\E[\ell'(\epsilon_i)] \cdot w(x_i) x_i\Big] = \E[\ell'(\epsilon_i)] \cdot \E[w(x_i) x_i] = 0.
\end{equation*}
If instead condition (2b) holds, the right-hand expression of equation~\eqref{EqnCandyCane} is clearly also equal to 0.

Finally, note that since $\ell'$ is bounded, the variables $\ell'(\epsilon_i \cdot v(x_i))$ are i.i.d.\ sub-Gaussian with parameter scaling with $\kappa_1$. By condition (1), the variables $w(x_i) x_i$ are also sub-Gaussian. Hence, the desired bound holds by using standard concentration results for i.i.d.\ sums of products of sub-Gaussian variables.

%%%%%

\subsection{Proof of Proposition~\ref{PropMRSC}}
\label{AppPropMRSC}

We begin with the outline of the main argument, with the proofs of supporting lemmas provided in subsequent subsections. The same general argument is used in the proofs of Propositions~\ref{PropMallows} and~\ref{PropHillRyan}, as well.

\subsubsection{Main argument}
\label{AppMain}

We have
\begin{align}
\label{EqnSquirrel}
\T(\beta_1, \beta_2) & \defn \inprod{\nabla \Loss_n(\beta_1) - \nabla \Loss_n(\beta_2)}{\beta_1 - \beta_2} \notag \\
& = \frac{1}{n} \sum_{i=1}^n (\ell'(x_i^T \beta_1 - y_i) - \ell'(x_i^T \beta_2 - y_i)) x_i^T (\beta_1 - \beta_2).
\end{align}
Under the assumptions, equation~\eqref{EqnSquirrel} implies that
\begin{multline}
\label{EqnFoxie}
\T(\beta_1, \beta_2) \ge \frac{1}{n} \sum_{i=1}^n (\ell'(x_i^T \beta_1 - y_i) - \ell'(x_i^T \beta_2 - y_i)) x_i^T (\beta_1 - \beta_2) 1_{A_i} - \kappa_2 \cdot \frac{1}{n} \sum_{i=1}^n \left(x_i^T (\beta_1 - \beta_2)\right)^2 1_{A_i^c},
\end{multline}
where we set $\kappa_2 = 0$ in the case when $\ell$ is convex (but $\ell''$ does not necessarily exist everywhere), and the event $A_i$ is defined according to
\begin{equation}
\label{EqnADefn}
A_i \defn \left\{|\epsilon_i| \le \frac{T}{2}\right\} \cap \left\{|x_i^T (\beta_1 - \beta_2)| \le \frac{T}{8r} \|\beta_1 - \beta_2\|_2\right\} \cap \left\{|x_i^T (\beta_2 - \betastar)| \le \frac{T}{4} \right\},
\end{equation}
for a parameter $T > 0$, using the definition~\eqref{EqnAlphaTau}. Inequality~\eqref{EqnFoxie} holds because when $\ell$ is convex, each summand in inequality~\eqref{EqnSquirrel} is always bounded below by 0; and when $\ell''$ exists and satisfies the bound~\eqref{EqnKappa2}, the mean value theorem gives
\begin{equation*}
(\ell'(x_i^T \beta_1 - y_i) - \ell'(x_i^T \beta_2 - y_i)) x_i^T (\beta_1 - \beta_2) = \ell''(u_i) \left(x_i^T (\beta_1 - \beta_2)\right)^2 \ge - \kappa_2 \left(x_i^T (\beta_1 - \beta_2)\right)^2,
\end{equation*}
where $u_i$ is a point lying between $x_i^T \beta_1 - y_i$ and $x_i^T \beta_2 - y_i$.

Note that on $A_i$ and for $\|\beta_1 - \betastar\|_2, \|\beta_2 - \betastar\|_2 \le r$, the triangle inequality gives
\begin{equation*}
|x_i^T \beta_2 - y_i| \le |x_i^T (\beta_2 - \betastar)| + |\epsilon_i| \le T,
\end{equation*}
and
\begin{equation*}
|x_i^T \beta_1 - y_i| \le |x_i^T (\beta_1 - \beta_2)| + |x_i^T (\beta_2 - \betastar)| + |\epsilon_i| \le \frac{T}{4} + \frac{T}{4} + \frac{T}{2} = T.
\end{equation*} 
Hence, the mean value theorem implies that
\begin{equation*}
\ell'(x_i^T \beta_1 - y_i) - \ell'(x_i^T \beta_2 - y_i) = \ell''(u_i) x_i^T (\beta_1 - \beta_2),
\end{equation*}
for some $u_i$ with $|u_i| \le r$. We then deduce from inequality~\eqref{EqnSquirrel} that
\begin{align}
\label{EqnCricket}
\T(\beta_1, \beta_2) & \ge \alpha_T \cdot \frac{1}{n} \sum_{i=1}^n \left(x_i^T (\beta_1 - \beta_2)\right)^2 1_{A_i} - \kappa_2 \cdot \frac{1}{n} \sum_{i=1}^n \left(x_i^T (\beta_1 - \beta_2)\right)^2 1_{A_i^c} \notag \\
& = (\alpha_T + \kappa_2) \cdot \frac{1}{n} \sum_{i=1}^n \left(x_i^T (\beta_1 - \beta_2)\right)^2 1_{A_i} - \kappa_2 \cdot \frac{1}{n} \sum_{i=1}^n \left(x_i^T (\beta_1 - \beta_2)\right)^2 \notag \\
& \ge (\alpha_T + \kappa_2) \cdot \frac{1}{n} \sum_{i=1}^n \varphi_{T \|\beta_1 - \beta_2\|_2/8r} \left(x_i^T (\beta_1 - \beta_2)\right) \cdot \psi_{T/2} (\epsilon_i) \cdot \psi_{T/4} \left(x_i^T (\beta_2 - \betastar)\right) \notag \\
& \qquad \qquad - \kappa_2 \cdot \frac{1}{n} \sum_{i=1}^n \left(x_i^T (\beta_1 - \beta_2)\right)^2 \notag \\
& \defn (\alpha_T + \kappa_2) \cdot f(\beta_1, \beta_2) - \kappa_2 \cdot \ftil(\beta_1, \beta_2).
\end{align}
Here, we have defined the truncation functions
\begin{equation}
\label{EqnTrunc}
\varphi_t(u) =
\begin{cases}
u^2, & \text{if } |u| \le \frac{t}{2}, \\
(t - u)^2, & \text{if } \frac{t}{2} \le |u| \le t, \\
0, & \text{if } |u| \ge t,
\end{cases}
\qquad \text{and} \qquad
\psi_t(u) =
\begin{cases}
1, & \text{if } |u| \le \frac{t}{2}, \\
2 - \frac{2}{t} |u|, & \text{if } \frac{t}{2} \le |u| \le t, \\
0, & \text{if } |u| \ge t,
\end{cases}
\end{equation}
as well as the functions
\begin{align*}
f(\beta_1, \beta_2) & \defn \frac{1}{n} \sum_{i=1}^n \varphi_{T \|\beta_1 - \beta_2\|_2/8r} \left(x_i^T (\beta_1 - \beta_2)\right) \cdot \psi_{T/2}(\epsilon_i) \cdot \psi_{T/4}\left(x_i^T (\beta_2 - \betastar)\right), \\
\ftil(\beta_1, \beta_2) & \defn \frac{1}{n} \sum_{i=1}^n \left(x_i^T (\beta_1 - \beta_2)\right)^2.
\end{align*}
Note in particular that $\varphi_t$ and $\psi_t$ are $t$-Lipschitz and $\frac{2}{t}$-Lipschitz, respectively, and the truncation functions satisfy the bounds
\begin{equation*}
\varphi_t(u) \le u^2 \cdot 1\{|u| \le t\}, \qquad \mbox{and} \qquad \psi_t(u) \le 1\{|u| \le t\}.
\end{equation*}
Note also that inequality~\eqref{EqnCricket} also implies the simple bound
\begin{equation}
\label{EqnQuadBd}
\inprod{\nabla \Loss_n(\beta_1) - \nabla \Loss_n(\beta_2)}{\beta_1 - \beta_2} \ge -\kappa_2 \cdot \ftil(\beta_1, \beta_2).
\end{equation}

We now define the sets
\begin{equation*}
B_\delta \defn \left\{(\beta_1, \beta_2): \|\beta_1 - \betastar\|_2, \|\beta_2 - \betastar\|_2 \le r, \; \frac{\|\beta_1 - \beta_2\|_1}{\|\beta_1 - \beta_2\|_2} \le \delta, \; \beta_1, \beta_2 \in \Ball_1(R)\right\},
\end{equation*}
for a parameter $1 \le \delta \le c \sqrt{\frac{n}{\log p}}$. Let
\begin{equation*}
Z(\delta) \defn \sup_{(\beta_1, \beta_2) \in B_\delta} \left\{\frac{1}{\|\beta_1 - \beta_2\|_2^2} \left|f(\beta_1, \beta_2) - \E[f(\beta_1, \beta_2)]\right|\right\},
\end{equation*}
and
\begin{equation*}
\Ztil(\delta) \defn \sup_{(\beta_1, \beta_2) \in B_\delta} \left\{\frac{1}{\|\beta_1 - \beta_2\|_2^2} \left|\ftil(\beta_1, \beta_2)  -\E\left[\ftil(\beta_1, \beta_2)\right] \right|\right\}.
\end{equation*}
With this notation, inequality~\eqref{EqnCricket} implies that for all $(\beta_1, \beta_2) \in B_\delta$, we have
\begin{align}
\label{EqnMatcha}
& \frac{\T(\beta_1, \beta_2)}{\|\beta_1 - \beta_2\|_2^2} \ge (\alpha_T + \kappa_2) \cdot \frac{\E[f(\beta_1, \beta_2)]}{\|\beta_1 - \beta_2\|_2^2} - \kappa_2 \cdot \frac{\E\left[\ftil(\beta_1, \beta_2)\right]}{\|\beta_1 - \beta_2\|_2^2} - (\alpha_T + \kappa_2) Z(\delta) - \kappa_2 \Ztil(\delta) \notag \\
& \qquad = \alpha_T \cdot \frac{\E[\ftil(\beta_1, \beta_2)]}{\|\beta_1 - \beta_2\|_2^2} - \frac{(\alpha_T +\kappa_2) \left(\E[\ftil(\beta_1, \beta_2)] - \E[f(\beta_1, \beta_2)]\right)}{\|\beta_1 - \beta_2\|_2^2} - (\alpha_T + \kappa_2) Z(\delta) - \Ztil(\delta).
\end{align}
The following lemma bounds the difference in expectations as a function of the truncation parameters. The proof is provided in Appendix~\ref{AppLemButterfly}.
\begin{lem*}
\label{LemButterfly}
We have the bound
\begin{align*}
\E\left[\ftil(\beta_1, \beta_2) \right] - \E[f(\beta_1, \beta_2)] \le c \sigma_x^2 \|\beta_1 - \beta_2\|_2^2 \left(\epsilon_T^{1/2} + \exp\left(-\frac{c'T^2}{\sigma_x^2 r^2}\right)\right).
\end{align*}
\end{lem*}

In particular, Lemma~\ref{LemButterfly} implies that when inequality~\eqref{EqnTricky} holds, we have
\begin{equation*}
(\alpha_T + \kappa_2) \left(\E[\ftil(\beta_1, \beta_2)] - \E[f(\beta_1, \beta_2)]\right) \le \frac{\alpha_T}{2} \cdot \E[\ftil(\beta_1, \beta_2)],
\end{equation*}
since
\begin{equation*}
\E[\ftil(\beta_1, \beta_2)] \ge \lambda_{\min}(\Sigma_x) \cdot \|\beta_1 - \beta_2\|_2^2.
\end{equation*}
Then inequality~\eqref{EqnMatcha} implies that
\begin{align}
\label{EqnEgg}
\frac{\T(\beta_1, \beta_2)}{\|\beta_1 - \beta_2\|_2^2} & \ge \frac{\alpha_T}{2} \cdot \frac{\E[\ftil(\beta_1, \beta_2)]}{\|\beta_1 - \beta_2\|_2^2} - (\alpha_T + \kappa_2) Z(\delta) - \kappa_2 \Ztil(\delta).
\end{align}
We now focus on the terms $Z(\delta)$ and $\Ztil(\delta)$. Note that $\ftil(\beta_1, \beta_2)$ is a quadratic form in $\beta_1 - \beta_2$, and for each unit vector $v \in \real^p$, the quantity $\frac{1}{n} \sum_{i=1}^n (x_i^T v)^2$ is an i.i.d.\ average of sub-exponential variables with parameter proportional to $\sigma_x^2$. Then by Lemmas 11 and 12 in Loh and Wainwright~\cite{LohWai11a}, we have the bound
\begin{equation}
\label{EqnRoti}
\left|\ftil(\beta_1, \beta_2) - \E[\ftil(\beta_1, \beta_2)]\right| \le t \sigma_x^2 \|\beta_1 - \beta_2\|_2^2 + t \sigma_x^2 \frac{\log p}{n} \|\beta_1 - \beta_2\|_1^2, \qquad \forall \beta_1, \beta_2 \in \real^p
\end{equation}
with probability at least $1 - c_1 \exp(-c_2 nt^2 + c_3 k \log p)$. In particular, since $\delta \le c \sqrt{\frac{n}{\log p}}$, we may guarantee that
\begin{equation}
\label{EqnWaffle}
\kappa_2 \Ztil(\delta) \le \frac{\alpha_T}{4} \cdot \frac{\E[\ftil(\beta_1, \beta_2)]}{\|\beta_1 - \beta_2\|_2^2},
\end{equation}
w.h.p. Turning to $Z(\delta)$, we have the following lemma, proved in Appendix~\ref{AppLemExpZ}:
\begin{lem*}
\label{LemExpZ}
For some constants $c, c'$, and $c''$, we have
\begin{equation}
\label{EqnCake}
\mprob\left(Z(\delta) \ge c'' \sigma_x \left(\frac{RT}{r^2} + \frac{\delta T}{r}\right) \sqrt{\frac{\log p}{n}}\right) \le c \exp(-c' \log p).
\end{equation}
\end{lem*}

Combining inequalities~\eqref{EqnWaffle} and~\eqref{EqnCake} with inequality~\eqref{EqnEgg}, we then have
\begin{align}
\label{EqnFrog}
\frac{\T(\beta_1, \beta_2)}{\|\beta_1 - \beta_2\|_2^2} & \ge \frac{\alpha_T}{4} \cdot \frac{\E[\ftil(\beta_1, \beta_2)]}{\|\beta_1 - \beta_2\|_2^2} - (\alpha_T + \kappa_2) c'' \sigma_x \left(\frac{RT}{r^2} + \frac{\delta T}{r} \right) \sqrt{\frac{\log p}{n}},
\end{align}
with probability at least $1 - c_1 \exp(-c_2 \log p)$. Let $n \succsim R^2 \log p$ be chosen such that
\begin{equation*}
(\alpha_T + \kappa_2) c'' \sigma_x \cdot \frac{RT}{r^2} \sqrt{\frac{\log p}{n}} \le \frac{\alpha_T}{8} \cdot \frac{\E[\ftil(\beta_1, \beta_2)]}{\|\beta_1 - \beta_2\|_2^2}.
\end{equation*}
Then inequality~\eqref{EqnFrog} implies that
\begin{equation}
\label{EqnToad}
\frac{\T(\beta_1, \beta_2)}{\|\beta_1 - \beta_2\|_2^2} \ge \frac{\alpha_T}{8} \cdot \frac{\E[\ftil(\beta_1, \beta_2)]}{\|\beta_1 - \beta_2\|_2^2} - \frac{c'' (\alpha_T + \kappa_2) \sigma_x T}{r} \delta \sqrt{\frac{\log p}{n}}.
\end{equation}
We now extend inequality~\eqref{EqnToad} to a bound that holds uniformly over the domain, with $\delta$ replaced by $\frac{\|\beta_1 - \beta_2\|_1}{\|\beta_1 - \beta_2\|_2}$. This is accomplished via a peeling argument in the proof of the following lemma:
\begin{lem*}
\label{LemPeel}
Fix $c_0 > 0$, and let
\begin{equation*}
\scriptD \defn \left\{\beta_1, \beta_2 \in \Ball_1(R): \|\beta_1 - \betastar\|_2, \|\beta_2 - \betastar\|_2 \le r \text{ and } \frac{\|\beta_1 - \beta_2\|_1}{\|\beta_1 - \beta_2\|_2} \le \frac{c_0 \alpha_T \lambda_{\min}(\Sigma_x)r}{\sigma_x (\alpha_T + \kappa_2) T} \sqrt{\frac{n}{\log p}} \right\}.
\end{equation*}
With probability at least $1 - c'_1 \exp(-c'_2 \log p)$, the following inequality holds uniformly over all $\beta_1, \beta_2 \in \scriptD$:
\begin{align}
\label{EqnPeelBd}
\frac{\T(\beta_1, \beta_2)}{\|\beta_1 - \beta_2\|_2^2} & \ge \alpha_T \cdot \frac{\lambda_{\min}(\Sigma_x)}{8} - \frac{c''(\alpha_T + \kappa_2)\sigma_x T}{c_0r} \frac{\|\beta_1 - \beta_2\|_1}{\|\beta_1 - \beta_2\|_2} \sqrt{\frac{\log p}{n}} \\
\label{EqnPeelBd2}
& \ge \alpha_T \cdot \frac{\lambda_{\min}(\Sigma_x)}{16} - \frac{c''' (\alpha_T + \kappa_2)^2 \sigma_x^2 T^2}{c_0^2 r^2} \frac{\log p}{n} \frac{\|\beta_1 - \beta_2\|_1^2}{\|\beta_1 - \beta_2\|_2^2}.
\end{align}
\end{lem*}
The proof of Lemma~\ref{LemPeel} is provided in Appendix~\ref{AppLemPeel}.

Finally, note that inequality~\eqref{EqnRoti} implies the bound
\begin{equation*}
\ftil(\beta_1, \beta_2) \le \alpha' \|\beta_1 - \beta_2\|_2^2 + \tau' \frac{\log p}{n} \|\beta_1 - \beta_2\|_1^2, \qquad \forall \beta_1, \beta_2 \in \real^p.
\end{equation*}
Together with inequality~\eqref{EqnQuadBd}, we then see that for a proper choice of the constant $c_0$, we have 
\begin{align}
\label{EqnPepper}
\T(\beta_1, \beta_2) & \ge -\kappa_2 \left(\alpha' \|\beta_1 - \beta_2\|_2^2 - \tau' \frac{\log p}{n} \|\beta_1 - \beta_2\|_1^2\right) \notag \\
& \ge \alpha_T \cdot \frac{\lambda_{\min}(\Sigma_x)}{16} \|\beta_1 - \beta_2\|_2^2 - \frac{c''' (\alpha_T + \kappa_2)^2 \sigma_x^2 T^2}{c_0^2 r^2} \frac{\log p}{n} \|\beta_1 - \beta_2\|_1^2
\end{align}
whenever $\frac{\|\beta_1 - \beta_2\|_1}{\|\beta_1 - \beta_2\|_2} > \frac{c_0r}{c'''(\alpha_T + \kappa_2) \sigma_x \tau} \sqrt{\frac{n}{\log p}}$. Combined with Lemma~\ref{LemPeel}, inequality~\eqref{EqnPepper} implies the RSC condition~\eqref{EqnLocalRSC}.

%%%%%%

\subsubsection{Proof of Lemma~\ref{LemButterfly}}
\label{AppLemButterfly}

Note that
\begin{align}
\label{EqnButterfly}
\E\left[\left(x_i^T (\beta_1 - \beta_2)\right)^2\right] - \E[f(\beta_1, \beta_2)] & \le \E\left[\left(x_i^T (\beta_1 - \beta_2)\right)^2 1\left\{|x_i^T (\beta_1 - \beta_2)| \ge \frac{T}{8r} \|\beta_1 - \beta_2\|_2\right\}\right] \notag \\
& \qquad + \E\left[ \left(x_i^T (\beta_1 - \beta_2)\right)^2 1\left\{|\epsilon_i| \ge \frac{T}{2}\right\}\right] \notag \\
& \qquad + \E\left[\left(x_i^T (\beta_1 - \beta_2)\right)^2 1\left\{|x_i^T (\beta_2 - \betastar)| \ge \frac{T}{4} \right\}\right].
\end{align}
Applying the Cauchy-Schwarz inequality, we have bounds of the form
\begin{equation*}
\E\left[\left(x_i^T (\beta_1 - \beta_2)\right)^2 1_{E_i}\right] \le \E\left[\left(x_i^T (\beta_1 - \beta_2)\right)^4\right]^{1/2} \cdot \E\left[1_{E_i}\right]^{1/2} \le c \sigma_x^2 \|\beta_1 - \beta_2\|_2^2 \cdot \left(\mprob(E_i)\right)^{1/2},
\end{equation*}
where the second inequality holds because of the assumption that $x_i$ is sub-Gaussian with parameter $\sigma_x^2$.

Furthermore, note that
\begin{equation*}
\mprob\left(|x_i^T (\beta_2 - \betastar)| \ge \frac{T}{4}\right) \le c \exp\left(-\frac{c'T^2}{\sigma_x^2 r^2}\right),
\end{equation*}
since $x_i$ is sub-Gaussian and $\|\beta_2 - \betastar\|_2 \le r$ by assumption. Finally, we have
\begin{equation*}
\mprob\left(|x_i^T (\beta_1 - \beta_2)| \ge \frac{T}{8r} \|\beta_1 - \beta_2\|_2\right) \le c \exp\left(-\frac{c'T^2}{\sigma_x^2 r^2}\right),
\end{equation*}
also by sub-Gaussianity of $x_i$. Combining these bounds with inequality~\eqref{EqnButterfly} then implies the desired result.

%%%%%%%

\subsubsection{Proof of Lemma~\ref{LemExpZ}}
\label{AppLemExpZ}

We first bound $\E[Z(\delta)]$. Following the argument in the proof of Lemma 11 of Loh and Wainwright~\cite{LohWai13}, we have
\begin{equation*}
\E[Z(\delta)] \le 2 \sqrt{\frac{\pi}{2}} \E\left[\sup_{(\beta_1, \beta_2) \in B_\delta} \frac{1}{\|\Delta\|_2^2} \left| \frac{1}{n} \sum_{i=1}^n g_i \cdot \varphi_{\frac{T\|\Delta\|_2}{8r} } \left(x_i^T (\beta_1 - \beta_2)\right) \psi_{\frac{T}{2}}(\epsilon_i) \psi_{\frac{T}{4}} \left(x_i^T (\beta_2 - \betastar)\right) \right|\right],
\end{equation*}
where we denote $\Delta \defn \beta_1 - \beta_2$, and the $g_i$'s are i.i.d.\ standard Gaussians. Define
\begin{equation*}
Z_{\beta_1, \beta_2} \defn \frac{1}{\|\Delta\|_2^2} \cdot \frac{1}{n} \sum_{i=1}^n g_i \cdot \varphi_{\frac{T \|\Delta\|_2}{8r}} \left(x_i^T (\beta_1 - \beta_2)\right) \psi_{\frac{T}{2}} (\epsilon_i) \psi_{\frac{T}{4}} \left(x_i^T (\beta_2 - \betastar)\right),
\end{equation*}
and note that conditioned on the $x_i$'s, each variable $Z_{\beta_1, \beta_2}$ is a Gaussian process. Furthermore, for distinct pairs $(\beta_1, \beta_2)$ and $(\beta'_1, \beta'_2)$, we have
\begin{equation*}
\var\left(Z_{\beta_1, \beta_2} - Z_{\beta'_1, \beta'_2}\right) \le 2 \var\left(Z_{\beta_1, \beta_2} - Z_{\beta_2' + \Delta, \, \beta_2'}\right) + 2 \var\left(Z_{\beta_2 + \Delta', \, \beta_2'} - Z_{\beta'_1, \beta'_2}\right)
\end{equation*}
Continuing to condition on the $x_i$'s, and denoting $\Delta' \defn \beta_1' - \beta'_2$, note that
\begin{multline}
\label{EqnPeach}
\var\left(Z_{\beta'_2 + \Delta, \, \beta'_2} - Z_{\beta'_1, \beta'_2}\right) = \frac{1}{n^2} \sum_{i=1}^n \psi_{\frac{T}{2}}^2 (\epsilon_i) \psi_{\frac{T}{4}}^2 \left(x_i^T (\beta_2 - \betastar)\right) \\
\cdot \left(\frac{1}{\|\Delta\|_2^2} \varphi_{\frac{T\|\Delta\|_2}{8r}} \left(x_i^T \Delta \right) - \frac{1}{\|\Delta'\|_2^2} \varphi_{\frac{T \|\Delta'\|_2}{8r}}\left(x_i^T \Delta'\right)\right)^2.
\end{multline}
Furthermore, $\varphi$ satisfies the homogeneity property that
\begin{equation*}
\frac{1}{c^2} \cdot \varphi_{ct}(cu) = \varphi_t(u), \qquad \forall c > 0.
\end{equation*}
Hence, inequality~\eqref{EqnPeach} implies that
\begin{align*}
\var\left(Z_{\beta'_2 + \Delta, \, \beta'_2} - Z_{\beta'_1, \beta'_2}\right) & \le \frac{1}{n^2} \sum_{i=1}^n \frac{1}{\|\Delta\|_2^4} \left(\varphi_{\frac{T \|\Delta\|_2}{8r}} \left(x_i^T \Delta\right) - \varphi_{\frac{T \|\Delta\|_2}{8r}} \left(x_i^T \Delta' \cdot \frac{\|\Delta\|_2}{\|\Delta'\|_2} \right)\right)^2 \\
& \le \frac{1}{n^2} \sum_{i=1}^n \frac{1}{\|\Delta\|_2^4} \cdot \frac{T^2 \|\Delta\|_2^2}{64r^2} \left(x_i^T \Delta - x_i^T \Delta' \cdot \frac{\|\Delta\|_2}{\|\Delta'\|_2}\right)^2 \\
& = \frac{1}{n^2} \sum_{i=1}^n \frac{T^2}{64r^2} \left(\frac{x_i^T \Delta}{\|\Delta\|_2} - \frac{x_i^T \Delta'}{\|\Delta'\|_2}\right)^2,
\end{align*}
where the second inequality uses the Lipschitz property of $\varphi$. Similarly, we may calculate
\begin{multline}
\label{EqnAnt}
\var\left(Z_{\beta_1, \beta_2} - Z_{\beta'_2 + \Delta, \, \beta'_2}\right) = \frac{1}{n^2} \sum_{i=1}^n \frac{1}{\|\Delta\|_2^4} \psi^2_{\frac{T}{2}} (\epsilon_i) \\
\cdot \varphi_{\frac{T \|\Delta\|_2}{8r}}^2 \left(x_i^T \Delta\right) \left(\psi_{\frac{T}{4}} \left(x_i^T (\beta_2 - \betastar)\right) - \psi_{\frac{T}{4}} \left(x_i^T (\beta'_2 - \betastar)\right)\right)^2
\end{multline}
Using the fact that $\psi_{T/4}$ is $\frac{8}{T}$-Lipschitz and $\varphi_{\frac{T \|\Delta\|_2}{8r}} \le \frac{T^2 \|\Delta\|_2^2}{256r^2}$, inequality~\eqref{EqnAnt} implies that
\begin{equation*}
\var\left(Z_{\beta_1, \beta_2} - Z_{\beta'_2 + \Delta, \, \beta'_2}\right) \le \frac{1}{n^2} \sum_{i=1}^n \frac{T^4}{256^2 r^4} \cdot \frac{64}{T^2} \left(x_i^T (\beta_2 - \beta'_2)\right)^2 = \frac{1}{n^2} \sum_{i=1}^n \frac{T^2}{32^2 r^4} \left(x_i^T (\beta_2 - \beta'_2)\right)^2.
\end{equation*}
If we define the second Gaussian process
\begin{equation*}
Y_{\beta_1, \beta_2} \defn \frac{T}{16r^2} \cdot \frac{1}{n} \sum_{i=1}^n g'_i \cdot x_i^T \beta_2 + \frac{T}{4r} \cdot \frac{1}{n} \sum_{i=1}^n g''_i \cdot \frac{x_i^T (\beta_1 - \beta_2)}{\|\beta_1 - \beta_2\|_2},
\end{equation*}
where $g'_i$ and $g''_i$ are independent standard Gaussians, the above calculation implies that
\begin{equation*}
\var\left(Z_{\beta_1, \beta_2} - Z_{\beta'_1, \beta'_2}\right) \le \var\left(Y_{\beta_1, \beta_2} - Y_{\beta'_1, \beta'_2}\right).
\end{equation*}
Hence, Lemma 14 in Loh and Wainwright~\cite{LohWai13} implies that
\begin{equation*}
\E\left[\sup_{(\beta_1, \beta_2) \in B_\delta} Z_{\beta_1, \beta_2}\right] \le 2 \cdot \E\left[\sup_{(\beta_1, \beta_2) \in B_\delta} Y_{\beta_1, \beta_2}\right],
\end{equation*}
where the expectations are no longer conditional. By an argument from Ledoux and Talgrand~\cite{LedTal91}, we also have
\begin{equation*}
\E\left[\sup_{(\beta_1, \beta_2) \in B_\delta} |Z_{\beta_1, \beta_2}|\right] \le \E\left[|Z_{\beta'_1, \beta'_2}|\right] + 2 \cdot \E\left[\sup_{(\beta_1, \beta_2) \in B_\delta} Z_{\beta_1, \beta_2}\right],
\end{equation*}
for any fixed $(\beta'_1, \beta'_2) \in B_\delta$. Furthermore,
\begin{equation*}
\E\left[|Z_{\beta'_1, \beta'_2}|\right] \le \sqrt{\frac{2}{\pi}} \cdot \sqrt{\var\left(Z_{\beta'_1, \beta'_2}\right)} \le \sqrt{\frac{2}{\pi}} \cdot \frac{T^2}{256r^2} \cdot \frac{1}{\sqrt{n}},
\end{equation*}
by conditioning on the $x_i$'s and using the bounds on $\varphi$ and $\psi$. We also have the bound
\begin{align*}
\E\left[\sup_{(\beta_1, \beta_2) \in B_\delta} Y_{\beta_1, \beta_2}\right] & \le \frac{RT}{16r^2} \cdot \E\left[\left\|\frac{1}{n} \sum_{i=1}^n g'_i x_i\right\|_\infty\right] + \frac{\delta T}{4r} \cdot \E\left[\left\|\frac{1}{n} \sum_{i=1}^n g''_i x_i\right\|_\infty \right] \\
& \le c\sigma_x \left(\frac{RT}{r^2} + \frac{\delta T}{r} \right) \sqrt{\frac{\log p}{n}}.
\end{align*}
Hence,
\begin{equation}
\label{EqnBoba}
\E[Z(\delta)] \le c'\sigma_x^2 \left(\frac{RT}{r^2} + \frac{\delta T}{r} \right) \sqrt{\frac{\log p}{n}}.
\end{equation}

Further note that for $(\beta_1, \beta_2) \in B_\delta$, each summand in $f(\beta_1, \beta_2)$ lies in the interval $\left[0, \; \frac{T^2}{64r^2}\right]$. Hence, by the bounded differences inequality, we have
\begin{equation}
\label{EqnZTail}
\mprob\left(\left|Z(\delta) - \E[Z(\delta)]\right| \ge t\right) \le c \exp\left(-\frac{c' r^2}{T^2} nt^2\right).
\end{equation}
Combining inequalities~\eqref{EqnBoba} and~\eqref{EqnZTail} then gives the desired result.

%%%%%%%

\subsubsection{Proof of Lemma~\ref{LemPeel}}
\label{AppLemPeel}

We parallel the peeling argument constructed in the proof of Lemma 11 in Loh and Wainwright~\cite{LohWai13}. Define the event
\begin{equation*}
\scriptE \defn \left\{\text{inequality}~\eqref{EqnPeelBd} \text{ holds } \forall \beta_1, \beta_2 \in \scriptD \right\},\end{equation*}
and define the functions
\begin{align*}
\htil(\beta_1, \beta_2; X) & \defn \alpha_T \cdot \frac{\lambda_{\min}(\Sigma_x)}{8} - \frac{\T(\beta_1, \beta_2)}{\|\beta_1 - \beta_2\|_2^2}, \\
g(\delta) & \defn \frac{c''(\alpha_T + \kappa_2) \sigma_x T}{2c_0 r} \cdot \delta \sqrt{\frac{\log p}{n}}, \\
h(\beta_1, \beta_2) & \defn \frac{\|\beta_1 - \beta_2\|_1}{\|\beta_1 - \beta_2\|_2}.
\end{align*}
By inequality~\eqref{EqnToad}, we have
\begin{equation}
\label{EqnLemonPeel}
\mprob\left(\sup_{\stackrel{(\beta_1, \beta_2) \in \scriptD:}{h(\beta_1, \beta_2) \le \delta}} \htil(\beta_1, \beta_2; X) \ge g(\delta)\right) \le c_1 \exp(-c_2 \log p),
\end{equation}
for any $1 \le \delta \le \frac{c_0 \alpha_T \lambda_{\min}(\Sigma_x) r}{\sigma_x (\alpha_T + \kappa_2)T} \sqrt{\frac{n}{\log p}}$. Since $\frac{\|\beta_1 - \beta_2\|_1}{\|\beta_1 - \beta_2\|_2} \ge 1$, we have
\begin{equation*}
1 \le h(\beta_1, \beta_2) \le \frac{c_0 \alpha_T \lambda_{\min}(\Sigma_x)}{\sigma_x (\alpha_T + \kappa_2) T} \sqrt{\frac{n}{\log p}},
\end{equation*}
over the region of interest. For each integer $m \ge 1$, define the set
\begin{equation*}
V_m \defn \left\{(\beta_1, \beta_2): 2^{m-1} \mu \le g(h(\beta_1, \beta_2)) \le 2^m \mu\right\} \cap \scriptD,
\end{equation*}
where $\mu \defn \frac{c c_0 \sigma_x T}{r} \sqrt{\frac{\log p}{n}}$. A union bound gives
\begin{equation*}
\mprob(\scriptE^c) \le \sum_{m=1}^M \mprob\left\{\exists (\beta_1, \beta_2) \in V_m: \htil(\beta_1, \beta_2; X) \ge 2 g(h(\beta_1, \beta_2)) \right\},
\end{equation*}
where $M \defn \left \lceil \log\left(c \sqrt{\frac{n}{\log p}}\right) \right \rceil$. Then
\begin{equation*}
\mprob(\scriptE^c) \le \sum_{m=1}^M \mprob\left(\sup_{\stackrel{(\beta_1, \beta_2) \in \scriptD:}{h(\beta_1, \beta_2) \le g^{-1}(2^m \mu)}} \htil(\beta_1, \beta_2; X) \ge 2^m \mu\right) \le M \cdot c_1 \exp(-c_2 \log p),
\end{equation*}
using inequality~\eqref{EqnLemonPeel}. Hence,
\begin{equation*}
\mprob(\scriptE^c) \le Mc_1 \exp\left(-c_2 \log p + \log \log \left(\frac{n}{\log p}\right)\right) \le c_1' \exp(-c_2' \log p).
\end{equation*}
Inequality~\eqref{EqnPeelBd2} holds by applying the arithmetic mean-geometric mean inequality to inequality~\eqref{EqnPeelBd}.

%%%%%%%%%

\subsection{Proof of Proposition~\ref{PropMallows}}
\label{AppPropMallows}

Again defining $\T(\beta_1, \beta_2)$ as in equation~\eqref{EqnSquirrel}, we have
\begin{equation}
\label{EqnPhilly}
\T(\beta_1, \beta_2) = \frac{1}{n} \sum_{i=1}^n w(x_i) \left(\ell'(x_i^T \beta_1 - y_i) - \ell'(x_i^T \beta_2 - y_i)\right) x_i^T (\beta_1 - \beta_2).
\end{equation}
Defining the event $A_i$ as in equation~\eqref{EqnADefn}, inequality~\eqref{EqnPhilly} implies that
\begin{align*}
\T(\beta_1, \beta_2) & \ge \frac{1}{n} \sum_{i=1}^n w(x_i) \left(\ell'(x_i^T \beta_1 - y_i) - \ell'(x_i^T \beta_2 - y_i)\right) x_i^T (\beta_1 - \beta_2) \\
& \ge \alpha_T \cdot \frac{1}{n} \sum_{i=1}^n w(x_i)\left(x_i^T (\beta_1 - \beta_2)\right)^2 1_{A_i} - \kappa_2 \cdot \frac{1}{n} \sum_{i=1}^n w(x_i) \left(x_i^T (\beta_1 - \beta_2)\right)^2 1_{A_i^c} \\
& = (\alpha_T + \kappa_2) \cdot \frac{1}{n} \sum_{i=1}^n w(x_i) \left(x_i^T (\beta_1 - \beta_2)\right)^2 1_{A_i} - \kappa_2 \cdot \frac{1}{n} \sum_{i=1}^n w(x_i) \left(x_i^T (\beta_1 - \beta_2)\right)^2.
\end{align*}
Note that $w(x_i) x_i$ is a sub-Gaussian vector with parameter $cb_0^2$. Defining the truncation functions $\varphi$ and $\psi$ as in equations~\eqref{EqnTrunc}, we then have
\begin{equation*}
\T(\beta_1, \beta_2) \ge (\alpha_T + \kappa_2) \cdot f(\beta_1, \beta_2) - \kappa_2 \cdot \ftil(\beta_1, \beta_2),
\end{equation*}
as in inequality~\eqref{EqnCricket}, where
\begin{align*}
f(\beta_1, \beta_2) & \defn \frac{1}{n} \sum_{i=1}^n w(x_i) \cdot \varphi_{T \|\beta_1 - \beta_2\|_2/8r} \left(x_i^T (\beta_1 - \beta_2)\right) \cdot \psi_{T/2} (\epsilon_i) \cdot \psi_{T/4} \left(x_i^T (\beta_2 - \betastar)\right), \\
\ftil(\beta_1, \beta_2) & \defn \frac{1}{n} \sum_{i=1}^n w(x_i) \left(x_i^T (\beta_1 - \beta_2)\right)^2.
\end{align*}

We first obtain an analog of Lemma~\ref{LemButterfly}, as follows:
\begin{lem*}
\label{LemCaterpillar}
We have the bound
\begin{equation*}
\E\left[\ftil(\beta_1, \beta_2)\right] - \E[f(\beta_1, \beta_2)] \le c b_0 \sigma_x^2 \|\beta_1 - \beta_2\|_2^2 \left(\epsilon_T^{1/2} + \exp\left(- \frac{c'T}{\sigma_x r}\right)\right).
\end{equation*}
\end{lem*}

\begin{proof}

We have
\begin{align}
\label{EqnBrown}
\E\left[w(x_i) \left(x_i^T (\beta_1 - \beta_2)\right)^2\right] - & \E[f(\beta_1, \beta_2)] \notag \\
& \le \E\left[w(x_i) \left(x_i^T (\beta_1 - \beta_2)\right)^2 1\left\{|x_i^T (\beta_1 - \beta_2)| \ge \frac{T}{8r} \|\beta_1 - \beta_2\|_2\right\}\right] \notag \\
& \qquad + \E\left[w(x_i) \left(x_i^T (\beta_1 - \beta_2)\right)^2 1 \left\{|\epsilon_i| \ge \frac{T}{2}\right\}\right] \notag \\
& \qquad + \E\left[w(x_i) \left(x_i^T (\beta_1 - \beta_2)\right)^2 1 \left\{|x_i^T (\beta_2 - \betastar)| \ge \frac{T}{4} \right\}\right].
\end{align}
Applying the Cauchy-Schwarz inequality to each term, the right-hand side of inequality~\eqref{EqnBrown} is then upper-bounded by
\begin{align*}
& \E\left[w^2(x_i) \left(x_i^T (\beta_1 - \beta_2)\right)^4\right]^{1/2} \cdot \Bigg\{\mprob\left(|x_i^T (\beta_1 - \beta_2)| \ge \frac{T}{8r} \|\beta_1 - \beta_2\|_2\right)^{1/2} \\
& \qquad \qquad \qquad \qquad \qquad \qquad \qquad \qquad + \mprob\left(|\epsilon_i| \ge \frac{T}{2}\right)^{1/2} + \mprob\left(|x_i^T (\beta_2 - \betastar)| \ge \frac{T}{4}\right)^{1/2}\Bigg\} \\
& \qquad \qquad \qquad \qquad \le c b_0 \sigma_x^2 \|\beta_1 - \beta_2\|_2^2 \left(\epsilon_T^{1/2} + \exp\left(- \frac{c'T}{\sigma_x r}\right)\right),
\end{align*}
using the assumption that $x_i$ is sub-exponential.

\end{proof}

Note that the statement of Lemma~\ref{LemExpZ} holds without modification, because the additional factor of $w(x_i)$ vanishes in the Gaussian comparison argument in the proof of the lemma, since $w(x_i) \le 1$. Furthermore, $\ftil(\beta_1, \beta_2)$ is again a quadratic form in $\beta_1 - \beta_2$, and since $w(x_i) x_i$ is bounded and $x_i$ is sub-exponential, the quantity $\frac{1}{n} \sum_{i=1}^n w(x_i) (x_i^T v)^2$ is an i.i.d.\ average of sub-exponential terms with parameter proportional to $b_0 \sigma_x^2$. Hence, a version of inequality~\eqref{EqnRoti} holds, with $\sigma_x^2$ replaced by $b_0 \sigma_x^2$. Then Lemma~\ref{LemPeel} follows by an identical peeling argument. Putting together the pieces, we arrive at the desired result.

%%%%%

\subsection{Proof of Proposition~\ref{PropHillRyan}}
\label{AppPropHillRyan}

This is very similar to the proof of Proposition~\ref{PropMRSC}. Again using the notation for the Taylor remainder defined in equation~\eqref{EqnSquirrel}, we have
\begin{equation*}
\T(\beta_1, \beta_2) = \frac{1}{n} \sum_{i=1}^n w(x_i) x_i^T (\beta_1 - \beta_2) \left\{\ell'\left((x_i^T \beta_1 - y_i) w(x_i)\right) - \ell'\left((x_i^T \beta_2 - y_i) w(x_i)\right)\right\}.
\end{equation*}
Defining
\begin{equation*}
A_i \defn \left\{|\epsilon_i| \le \frac{T}{2}\right\} \cap \left\{|w(x_i) x_i^T (\beta_1 - \beta_2)| \le \frac{T}{8r} \|\beta_1 - \beta_2\|_2\right\} \cap \left\{|w(x_i) x_i^T (\beta_2 - \betastar)| \le \frac{T}{4} \right\},
\end{equation*}
we have that on the event $A_i$ and for $\|\beta_1 - \betastar\|_2, \|\beta_2 - \betastar\|_2 \le r$,
\begin{equation*}
|w(x_i) (x_i^T \beta_2 - y_i)| \le |w(x_i) x_i^T (\beta_2 - \betastar)| + |w(x_i) \epsilon_i| \le |w(x_i) x_i^T (\beta_2 - \betastar)| + |\epsilon_i| \le T,
\end{equation*}
and
\begin{equation*}
|w(x_i) (x_i^T \beta_1 - y_i)| \le |w(x_i) x_i^T (\beta_1 - \beta_2)| + |w(x_i) x_i^T (\beta_2 - \betastar)| + |w(x_i) \epsilon_i| \le \frac{T}{4} + \frac{T}{4} + \frac{T}{2},
\end{equation*}
using the fact that $w(x_i) \le 1$. Hence, we have
\begin{align*}
\T(\beta_1, \beta_2) & \ge \alpha_T \cdot \frac{1}{n} \sum_{i=1}^n \left(w(x_i) x_i^T (\beta_1 - \beta_2)\right)^2 1_{A_i} - \kappa_2 \cdot \frac{1}{n} \sum_{i=1}^n \left(w(x_i) x_i^T (\beta_1 - \beta_2)\right)^2 1_{A_i^c} \\
& = (\alpha_T + \kappa_2) \cdot \frac{1}{n} \sum_{i=1}^n \left(w(x_i) x_i^T (\beta_1 - \beta_2)\right)^2 1_{A_i} - \kappa_2 \cdot \frac{1}{n} \sum_{i=1}^n \left(w(x_i) x_i^T (\beta_1 - \beta_2)\right)^2.
\end{align*}
We may define the truncation functions exactly as in the proof of Proposition~\ref{PropMRSC}, the only modification being that $x_i$ is replaced by $w(x_i) x_i$. Furthermore, since $|w(x_i) x_i^T v| \le b_0$ for every unit vector $v$, the vector $w(x_i) x_i$ is always sub-Gaussian with parameter $b_0^2$, regardless of the distribution of $x_i$. It follows that with
\begin{align*}
f(\beta_1, \beta_2) & \defn \frac{1}{n} \sum_{i=1}^n \varphi_{T \|\beta_1 - \beta_2\|_2/8r} \left(w(x_i) x_i^T (\beta_1 - \beta_2)\right) \cdot \psi_{T/2} (\epsilon_i) \cdot \psi_{T/4} \left(w(x_i) x_i^T (\beta_2 - \betastar)\right), \\
\ftil(\beta_1, \beta_2) & \defn \frac{1}{n} \sum_{i=1}^n \left(w(x_i) x_i^T (\beta_1 - \beta_2)\right)^2,
\end{align*}
we arrive at the familiar inequality,
\begin{equation*}
\T(\beta_1, \beta_2) \ge (\alpha_T + \kappa_2) \cdot f(\beta_1, \beta_2) - \kappa_2 \cdot \ftil(\beta_1, \beta_2).
\end{equation*}
The remainder of the proof is identical to the proof of Proposition~\ref{PropMRSC}, with $x_i$ replaced by $w(x_i) x_i$, which is sub-Gaussian with parameter $b_0^2$.

\section{Proof of Corollary~\ref{CorDist}}
\label{AppCorDist}

The proof of this corollary is a fairly immediate consequence of Theorem~\ref{ThmOracle} and the following result from He and Shao~\cite{HeSha00}:

\begin{lem*} [Corollary 2.1, He and Shao~\cite{HeSha00}]
\label{LemHeShao}
Suppose we have i.i.d.\ observations from the usual linear regression model
\begin{equation*}
y_i = x_i^T \betastar + \epsilon_i,
\end{equation*}
where $\betastar \in \real^p$. Suppose
\begin{equation*}
\Loss_n(\beta) = \frac{1}{n} \sum_{i=1}^n \ell(x_i^T \beta - y_i),
\end{equation*}
and the following conditions are satisfied:
\begin{itemize}
\item[(i)] In probability, $0 < \lambda_{\min}\left(\frac{X^TX}{n}\right)$ and $\lambda_{\max}\left(\frac{X^TX}{n}\right) < \infty$.
\item[(ii)] $\ell$ is convex and smooth, $\ell''$ and $\ell'''$ are bounded, and $\E\left[\ell''(\epsilon_i)\right] \in (0, \infty)$.
\item[(iii)] $\max_{1 \le i \le n} \frac{\|x_i\|_2^2}{p} = \order_P(1)$ and $\sup_{\|u\|_2 = \|w\|_2 = 1} \frac{1}{n} \sum_{i=1}^n |x_i^T u|^2 |x_i^T w|^2 = \order_P(1)$.
\end{itemize}
Suppose $\Loss_n$ has a unique minimizer given by $\betahat$. If $\frac{p \log ^3p}{n} \rightarrow 0$, then $\|\betahat - \betastar\|_2 = \order_p\left(\sqrt{\frac{p}{n}}\right)$. If $\frac{p^2 \log p}{n} \rightarrow 0$, then for any unit vector $v \in \real^p$, we have
\begin{equation*}
\frac{\sqrt{n}}{\sigma_v} \cdot v^T (\betahat - \betastar) \stackrel{d}{\longrightarrow} N(0, 1),
\end{equation*}
where
\begin{equation*}
\sigma^2_v \defn \frac{1}{\E\left[\ell''(\epsilon_i)\right] \cdot \E\left[\left(\ell'(\epsilon_i)\right)^2\right]} \cdot v^T \left(\frac{X^TX}{n}\right) v.
\end{equation*}
\end{lem*}

We apply the result to the oracle estimator $\boracle_S$ defined in equation~\eqref{EqnBoracle}, with $k$ taking the place of $p$. Although Lemma~\ref{LemHeShao} requires $\Loss_n$ to be convex, a careful inspection of the proofs in He and Shao~\cite{HeSha00} reveals that the results still hold if we restrict our attention to a subset of $\real^p$ on which $\Loss_n$ is convex and $\betahat$ is the unique minimizer. By Lemma~\ref{LemConvex}, this is exactly the case over the restricted region $S_r$, when $\Loss_n$ satisfies the RSC condition~\eqref{EqnLocalRSC}. Furthermore, it is straightforward to check that conditions (i)--(iii) of Lemma~\ref{LemHeShao} under the given assumptions. Note that by Theorem 2.1 in Hsu et al.~\cite{HsuEtal12}, we have
\begin{equation*}
\mprob\left(\frac{\|x_i\|_2^2}{k} \ge t\right) \le c_1 \exp(-c_2 k), \qquad \forall i,
\end{equation*}
when the $x_i$'s are sub-Gaussian, implying that
\begin{equation*}
\mprob\left(\max_{1 \le i \le n} \frac{\|x_i\|_2^2}{k} \ge t \right) \le n \cdot c_1 \exp(-c_2 k).
\end{equation*}
Hence, for $k \ge C \log n$, the right-hand expression is bounded above by $c_1 \exp(-c_2' k)$, and the first part of condition (iii) is satisfied.

We conclude that the desired results hold for the oracle estimator $\boracle$, and by Theorem~\ref{ThmOracle}, also for $\betatil$.

%%%%%%

\section{Proofs of additional lemmas}
\label{AppLemmas}

In this section, we provide proofs of additional technical lemmas appearing in the body of the paper.

%%%%

\subsection{Proof of Lemma~\ref{LemConvex}}
\label{AppLemConvex}

For $\beta_1, \beta_2 \in S_r$, we have
\begin{equation*}
\|\beta_1 - \beta_2\|_1 \le \sqrt{k} \|\beta_1 - \beta_2\|_2.
\end{equation*}
Hence, the RSC condition~\eqref{EqnLocalRSC} implies that
\begin{equation*}
\inprod{\nabla \Loss_n(\beta_1) - \nabla \Loss_n(\beta_2)}{\beta_1 - \beta_2} \ge \left(\alpha - \tau \frac{k \log p}{n}\right) \|\beta_1 - \beta_2\|_2^2,
\end{equation*}
implying the desired conclusion.

%%%%%

\subsection{Proof of Lemma~\ref{LemStableLasso}}
\label{AppLemStableLasso}

Note that if $X = (X_1, \dots, X_n)$ is a vector of i.i.d.\ $\alpha$-stable random variables with $\gamma = 1$, equation~\eqref{EqnAlphaChar} implies that for $w \in \real^n$, we have
\begin{equation*}
\E\left[\exp\left(it \cdot w^T X\right)\right] = \exp\left(-\|w\|_\alpha^\alpha \; |t|^\alpha \right), \qquad \forall t > 0,
\end{equation*}
where $\|w\|_\alpha \defn \left(\sum_{i=1}^n |w|^\alpha\right)^{1/\alpha}$. Hence, $w^T X$ is also $\alpha$-stable, but with the scale parameter $\|w\|_\alpha$. Furthermore, if $Z \in \real$ is sub-Gaussian with parameter $\sigma_z^2$, then for $\alpha \in (0,2]$, the random variable $|Z|^\alpha$ is sub-exponential with parameter $c \sigma_z^2$. Indeed, the moments of $|Z|^\alpha$ may be bounded as
\begin{equation*}
\E\left[|Z|^{\alpha p}\right]^{1/p} \le \E\left[|Z|^{2p}\right]^{\frac{\alpha}{2p}} \le \left(c \sigma_z \sqrt{p}\right)^\alpha \le c'\sigma_z^2 p,
\end{equation*}
where the first inequality comes from H\"{o}lder's inequality and the second inequality follows because $Z$ is sub-Gaussian~\cite{Ver12}. Hence, $Z^{\alpha}$ is sub-exponential. Consequently, for any $1 \le j \le p$, the quantity
\begin{equation*}
\left\|\frac{Xe_j}{n^{1/\alpha}}\right\|_\alpha^\alpha = \frac{1}{n} \|X e_j\|_\alpha^\alpha = \frac{1}{n} \sum_{i=1}^n |x_{ij}|^\alpha
\end{equation*}
exhibits sub-exponential concentration to $\E[|X_{ij}|^\alpha]$.

In the context of ordinary least squares regression with the Lasso, note that for an arbitrary $1 \le j \le p$, we have
\begin{equation}
\label{EqnChaat}
\mprob\left(\left\|\frac{X^T \epsilon}{n}\right\|_\infty \ge \lambda\right) = \mprob\left(\left\|\frac{X^T \epsilon}{n^{1/\alpha}}\right\|_\infty \ge n^{1-1/\alpha} \lambda\right) \ge \mprob\left(\left|\frac{e_j^T X^T \epsilon}{n^{1/\alpha}} \right| \ge n^{1 - 1/\alpha} \lambda\right).
\end{equation}
Since $\left |\frac{e_j^T X \epsilon}{n^{1/\alpha}}\right |$ is $\alpha$-stable with scale parameter $\Theta\left(\E[|X_{ij}|^\alpha]\right)$, by the above discussion, the right-hand expression in inequality~\eqref{EqnChaat} is bounded below by a constant $c_\alpha$ whenever $n^{1 - 1/\alpha} \lambda \rightarrow 0$. In particular, this is the case when $\alpha < 2$. Hence, we conclude that the bound
\begin{equation*}
\left\|\frac{X^T \epsilon}{n}\right\|_\infty \precsim \sqrt{\frac{\log p}{n}}
\end{equation*}
does \emph{not} hold w.h.p.\ when the entries of $\epsilon$ are drawn from an $\alpha$-stable distribution with $\alpha < 2$.

Finally, recall that if $\betahat$ is a global solution for the Lasso and $\lambda \succsim \left\|\frac{X^T \epsilon}{n}\right\|_\infty$, we have the $\ell_2$-error bound
\begin{equation*}
\|\betahat - \betastar\|_2 \le c\sqrt{k} \cdot \max\left\{\lambda, \; \left\|\frac{X^T \epsilon}{n}\right\|_\infty\right\},
\end{equation*}
with high probability~\cite{BicEtal08}. This establishes the inconsistency of the Lasso estimator.

\bibliography{refs.bib}

\end{document}